%% file: MasterIhara.tex
\documentclass[12pt, a4paper,english]{smfart}

\usepackage{pstricks-add}

\input{macros.sty}

\theoremstyle{plain}

\title[Ihara's lemma: the limite case]{Ihara's Lemma for $\GL_d$: the limit case}

\author{Pascal Boyer}
\email{boyer@math.univ-paris13.fr}
\address{Universit\'e Paris 13, Sorbonne Paris Nord \\
LAGA, CNRS, UMR 7539\\ 
F-93430, Villetaneuse (France) \\
Coloss: ANR-19-PRC}

\begin{document}

\setcounter{tocdepth}{3}

\maketitle

\begin{abstract}
Clozel, Harris and Taylor proposed in \cite{CHT} conjectural 
generalizations of the
classical Ihara's lemma for $\GL_2$, to higher dimensional similitude groups.
We prove these conjectures in the so called \emph{limit case}, which
after base change is the essential one, under any hypothesis allowing
level raising as for example theorem 5.1.5 in \cite{gee-annalen}. 
\end{abstract}

\maketitle

\tableofcontents


\input{intro}
\input{preliminaries}

\input{filtrations}

\input{proof}
\bibliography{Biblio}
\bibliographystyle{amsalpha}

\end{document}

%% file: intro.tex
\section{Introduction}

\subsection{Ihara's lemma: origin and proofs}

In the Taylor-Wiles method Ihara's lemma is the key ingredient to extend
a $R=T$ property from the minimal case to a non minimal one. It is usually
formulated by the injectivity of some map as follows. 

Let 
$\Gamma=\Gamma_0(N)$ be the usual congruence subgroup of
$SL_2(\Zm)$ for some $N>1$, and for a prime $p$ not dividing $N$
let $\Gamma':=\Gamma \cap \Gamma_0(p)$. We then have two degeneracy maps
$$\pi_1,\pi_2: X_{\Gamma'} \longrightarrow X_\Gamma$$
between the compactified modular curves of levels $\Gamma'$ and $\Gamma$
respectively, induced by the inclusion
$$\Gamma' \hookrightarrow \Gamma \hbox{ and }
\left ( \begin{array}{cc} p & 0 \\ 0 & 1 \end{array} \right )
\Gamma' \left ( \begin{array}{cc} p & 0 \\ 0 & 1 \end{array} \right )^{-1}
\hookrightarrow \Gamma.$$
For $l \neq p$, we then have a map
$$\pi^*:=\pi_1^*+\pi_2^*: H^1(X_\Gamma,\Fm_l)^2 \longrightarrow 
H^1(X_{\Gamma'},\Fm_l).$$

\begin{thm}
Let $\mathfrak m$ be a maximal ideal of the Hecke algebra acting on these
cohomology groups which is non Eisenstein, i.e. that corresponds to an
irreducible Galois representation. Then after localizing at $\mathfrak m$,
the map $\pi^*$ is injective.
\end{thm}

Diamond and Taylor in \cite{D-T} proved an analogue of Ihara's lemma for
Shimura curves over $\Qm$. For a general totally real number field $F$ with
ring of integers $\mathcal{O}_F$,
Manning and Shotton in \cite{ShMa} succeeded to prove it under
some large image hypothesis. Their strategy is entirely different from those of
\cite{D-T}
but consists roughly 
\begin{itemize}
\item to carry Ihara's lemma for a compact Shimura curve $Y_{\bar K}$ 
associated to a definite quaternion
algebra $\overline D$ ramified at some auxiliary place $v$ of $F$,
in level $\bar K=\bar K^v \bar K_v$ an open compact subgroup of 
$D \otimes \Am_{F,f}$ unramified at $v$,

\item to the indefinite situation $X_K$ relatively to a quaternion division
algebra $D$ ramified at all but one infinite place of $F$, and isomorphic
to $\bar D$ at all finite places of $F$ different to $v$, and with level
$K$ agreing with $\bar K^v$ away from $v$.
\end{itemize}

Indeed in the definite case Ihara's statement is formulated by the injectivity of
$$\pi^*=\pi_1^*+\pi_2^*: H^0(Y_{\bar K},\Fm_l)_{\mathfrak m} \oplus
H^0(Y_{\bar K},\Fm_l)_{\mathfrak m} \longrightarrow 
H^0(Y_{\bar K_0(v)},\Fm_l)_{\mathfrak m}$$
where both $\overline D$ and $\bar K$ are unramified at the place $v$
and $\bar K_0(v)_v$ is the subgroup of $\GL_2(F_v)$ of elements which
are upper triangular modulo $p$.

The proof goes like this, cf. \cite{ShMa} theorem 6.8.
Suppose $(f,g) \in \ker \pi^*$. Regarding $f$ and $g$
as $K^v$-invariant function on $\overline G(F) \backslash \overline G(\Am_{F,f})$,
then $f(x)=-g(x \omega)$ where 
$\omega=\left ( \begin{array}{cc} \varpi_v & 0 \\ 0 & 1 \end{array} \right )$,
$\varpi_v$ being an uniformizer for $F_w$ and $\overline G$ being
the algebraic group over $\mathcal{O}_F$ associated to 
$\mathcal{O}_{\overline D}^\times$
the inversible group of the maximal order $\mathcal{O}_{\overline D}$ of 
$\overline D$:
note that $\overline G(F_v) \simeq \GL_2(F_v)$.
Then $f$ is invariant under $K^v$ and $\omega^{-1} K^v \omega$ so that,
using the strong approximation theorem for the subgroup of $\overline G$
of elements of reduced norm $1$, then $f$ factors through the reduced norm map,
and so is supported on Eisenstein maximal ideals.

The link between $X_K$ and $Y_{K^v}$ is given by the geometry of the
integral model of the Shimura curve $X_{K_0(v)}$ with
$\Gamma_0(v)$-level structure. The main new ingredient of \cite{ShMa} to
carry this geometric link to Ihara's lemma goes through the patching technology
which allows to obtain maximal Cohen-Macaulay modules over deformation rings. 
Using a flatness property and Nakayama's lemma, there are then able 
to extend a surjective property, dual to the injectivity in the Ihara's lemma, from
the maximal unipotent locus on the deformation space to the whole space,
and recover the Ihara's statement reducing by the maximal ideal of the
deformation ring. 
%

\subsection{Generalisations of Ihara's Lemma}

To generalize the classical Ihara's lemma for $\GL_d$, there are essentially
two approaches. 

\textit{The first natural one} developed by Clozel, Harris and Taylor in
their first proof of the Sato-Tate theorem \cite{CHT}, 
focuses on the $H^0$ with coefficients
in $\Fm_l$ of a zero dimensional Shimura variety associated to higher dimensional
definite division algebras. More precisely consider a totally real field
$F^+$ and a imaginary quadratic extension $E/\Qm$ and define $F=F^+E$.
We then consider $\overline G/\Qm$ an unitary group with $\overline G(\Qm)$ 
compact so that
$\overline G$ becomes an inner form of $\GL_d$ over $F$. This means, cf.
\S \ref{subsec-KHT}, we have fixed a division algebra $\overline B$ with center $F$,
of dimension $d^2$, provided with an involution of the second kind such that its
restriction to $F$ is the complex conjugation. 
We moreover suppose that at every place $w$ of $F$,
either $\overline B_w$ is split or a local division algebra.

Let $v$ (resp. $v_1$) be a place of $F$ above a prime number $p$ (resp. $p_1 \neq p$) split in $E$ and
such that $\overline B_v^\times \simeq \GL_d(F_v)$ (resp. $\overline B_{v_1}$ is a central 
division algebra over $F_{v_1}$)
where $F_v$ is the associated local field with ring of integers $\mathcal{O}_{v}$
and residue field $\kappa(v)$. 

\begin{nota} Let $q_v$ be the order of the residue field $\kappa(v)$.
\end{nota}

Consider then an open compact 
subgroup $\overline K^v$ infinite at $v$ in the following sense: $\overline G(\Qm_p)
\simeq \Qm_p^\times \times \prod_{v_i^+} \overline B_{v_i^+}^{op,\times}$ where 
$p= \prod_i v_i^+ $ in $F^+$ and we identify places of $F^+$ over $p=uu^c \in E$ 
with places of $F$ over $u$.
We then ask $\overline K^v_p=\Zm_p^\times \times \prod_{w|u} \overline K_w$ 
to be such 
that $\overline K_v$ is restricted to the identity element.

The associated Shimura variety with level $\overline K=\overline K^v \overline K_v$
for some finite level $\overline K_v$ at $v$, denoted by 
$\overline \Sh_{\overline K}$,
is then such that its $\Cm$-points are $\overline G(\Qm) \backslash 
\overline G(\Am_{\Qm}^\oo)/ \overline K$
and for $l \neq p,p_1$, its $H^0$ with coefficients in 
$\overline \Fm_l$ is then identified with the space 
$$S_{\overline G}(\overline K,\overline \Fm_l)= \{ f:
\overline G(\Qm) \backslash 
\overline G(\Am_{\Qm}^\oo)/ \overline K \longrightarrow \overline \Fm_l 
\hbox{ locally constant} \}.$$
Replacing $\overline K$ by $\overline K^v$, we then obtain
an admissible smooth representation of $\GL_d(F_v)$ equipped with an action
of the Hecke algebra $\Tm(\overline K^v)$
defined as the image of the abstract unramified Hecke algebra, cf. definition
\ref{defi-Tm}, inside $\End ( S_{\overline G}(\overline K^v,\overline \Fm_l) \Bigr )$.

To a maximal ideal $\mathfrak m$ of $\Tm(\overline K^v)$ is associated a
Galois $\overline \Fm_l$-representation $\overline \rho_{\mathfrak m}$, cf.
\S \ref{sec-hecke}.
We consider the case where this representation is irreducible.
Note in particular that such an $\mathfrak m$ is then non pseudo-Eisenstein
in the usual terminology. 

\begin{conj} \label{conj-ihara1} (cf. conjecture B in \cite{CHT}) \\
Any irreducible $\GL_d(F_v)$-submodule of
$S_{\overline G}(\overline K^v,\overline \Fm_l)_{\mathfrak m}$ is generic.
\end{conj}

For rank $2$ unitary groups, we recover the previous statement
as the characters are exactly those representations 
which do not have a Whittaker model, i.e.
are the non generic ones. For $d \geq 2$, over $\overline \Qm_l$, the generic
representations of $\GL_d(F_v)$ are the irreducible parabolically induced
representations $\st_{t_1}(\pi_{v,1}) \times \cdots \times \st_{t_r}(\pi_{v,r})$
where for $i=1,\cdots,r$, 
\begin{itemize}
\item $\pi_{v,i}$ is an irreducible cuspidal representation of $\GL_{g_i}(F_v)$,
\item $\st_{t_i}(\pi_{v,i})$ is a Steinberg representations, cf. definition \ref{defi-LT},
\item $\sum_{i=1}^r t_ig_i=d$ where
the Zelevinsky segments $[\pi_{v,i}\{ \frac{1-t_i}{2} \} ,\pi_{v,i} \{ 
\frac{t_i-1}{2} \} ]$are not linked in the sense of \cite{zelevinski2}. 
\end{itemize}

Over $\overline \Fm_l$
every irreducible generic representation is obtained as the unique generic
subquotient of the modulo $l$ reduction of a generic representation. It can also
be characterized intrinsically using representation of the mirabolic subgroup, cf.
\S \ref{para-gen}.


\begin{defin} \label{defi-ihara1} (cf. definition of \cite{CHT} 5.1.9) \\
An admissible smooth $\overline \Fm_l[\GL_d(F_v)]$-module $M$ is said to have
the weak Ihara property if for every 
$m \in  M^{\GL_d(\mathcal{O}_v)} $ which is an eigenvector of 
$\overline \Fm_l[\GL_d(\mathcal{O}_v) \backslash 
\GL_d(F_v)/\GL_d(\mathcal{O}_v)]$,
every irreducible submodule of the $\overline \Fm_l[\GL_d(F_v)]$-module 
generated by $m$, is generic.
\end{defin}

\rem If we ask 
$S_{\overline G}(\overline K^v,\overline \Fm_l)_{\mathfrak m}$ to verify
the weak Ihara property, then it should
have non trivial unramified vectors so that the supercuspidal support
of the restriction $\overline \rho_{\mathfrak m,v}$ of 
$\overline \rho_{\mathfrak m}$ to the decomposition subgroup at $v$, is
made of unramified characters.

\medskip

\textit{The second approach} asks to find a map playing the same role as
$\pi^*=\pi_1^*+\pi_2^*$. It is explained in section 5.1 of \cite{CHT} with
the help of the element
$$\theta_v \in \Zm_l [K_1(v^n) \backslash \GL_d(F_v)/\GL_d(\mathcal{O}_{F_v})]$$
constructed by Russ Mann, cf. proposition 5.1.7 of \cite{CHT},
where $F_v$ is here a finite extension of $\Qm_p$ with ring of integers 
$\mathcal{O}_v$.

\begin{defin} \label{defi-ihara2}
An admissible smooth $\overline \Fm_l[\GL_d(F_v)]$-module $M$ is said to have the
almost Ihara property if 
$\theta_v: M^{\GL_d(\mathcal{O}_v)} \longrightarrow M$
is injective.
\end{defin}

Recall that $l$ is called quasi-banal for $\GL_d(F_v)$ if either
$l \nmid \sharp \GL_d(\kappa_v)$ (the banal case) or $l>d$ and
$q_v \equiv 1 \mod l$ (the limit case).

\begin{prop} (cf. \cite{CHT} lemma 5.1.10) \\
Suppose that $l$ is quasi-banal and $M$ is a $\overline \Fm_l[\GL_d(F_v)]$-module
verifying the Ihara property. If 
$\ker (\theta_v: M^{\GL_d(\mathcal{O}_v)} \longrightarrow M)$
is a $\Fm_l[\GL_d(\mathcal{O}_{F_v}) \backslash \GL_d(F_v)/
\GL_d(\mathcal{O}_{F_v})]$-module, then $M$ has the almost Ihara property.
\end{prop}

\noindent \textbf{Applications}: the generalizations of the classical Ihara's lemma
were introduced in \cite{CHT} to prove a non minimal $R=\mathbb T$ theorem.
The weaker statement $R^{red}=\mathbb T$ where $R^{red}$ is the reduced
quotient of $R$, was later obtained unconditionally
using Taylor's \emph{Ihara avoidance} method, cf. \cite{taylor-avoidance} which was
enough to prove the Sato-Tate conjecture. However, the full $R = \mathbb T$ 
theorem would have applications to special values of the adjoint $L$-function
and would imply that $R$ is a complete intersection. It should also be useful 
for generalizing the local-global compatibility results of \cite{emerton-local-global}.

In \cite{Moss-Ihara}, the author also proved that Ihara's property in the quasi-banal case
is equivalent to the following result. 

\begin{prop} (cf. \cite{Moss-Ihara} corollary 9.5) \\
Let $\mathfrak m$ be a non-Eisenstein
maximal ideal of $\Tm^S$ and 
$f \in S_{\overline G}(\overline K^v\GL_d(\OC_v),\overline \Fm_l)$.
Let $\Iw_v$ be the Iwahori subgroup of $\GL_d(\OC_v)$, then the
$\overline \Fm_l[\Iw_v \backslash \GL_d(F_v) / \GL_d(\OC_v)]$-submodule
of $S_{\overline G}(\overline K^v \Iw_v,\overline \Fm_l)$ generated by $f$
is of dimension $d!$.
\end{prop}

\subsection{Main results}

With the previous notations, let $q_v$ be the order of the residue field of $F_v$.
We fix some prime number $l$ unramified in $F^+$ and split in $E$ and we place
ourself in the limit case where $q_v \equiv 1 \mod l$ with $l>d$, 
which is, after by base change, the crucial case to consider.

\begin{defin} (see the introduction of \cite{BLGGT}) \\
Let $K$ be a finite extension of $\Qm_l$.
A potentially crystalline representation $\rho$ of the absolute
Galois group of a $K$ is said to be potentially diagonalizable
if there is a finite extension $K_0$ of $K$ such that $\rho_{|G_{K_0}}$ lies on the same
irreducible component of the universal crystalline lifting ring of $\overline{\rho_{|G_{K_0}}}$
(with fixed Hodge-Tate numbers) as a sum of characters lifting $\overline{\rho_{|G_{K_0}}}$.
%
%
%
\end{defin}
%

\rem  Ordinary crystalline representations
are potentially diagonalizable, as are crystalline representations in the
Fontaine-Laffaille range (i.e. over an absolutely unramfied base and with Hodge-Tate numbers 
in the range $[0,l-2]$). Potential diagonalizability is also preserved
under restriction to the absolute Galois group of a finite extension.

\begin{thm} \label{thm-main}
\textbf{In the limit case}, suppose that there exists
a prime $p_0=u_0 \bar u_0$ split in $E$ with a place $v_0 | u_0$ of $F$ such
that $\overline B_{v_0}$ is a division algebra.
Consider $\mathfrak m$ such that
$$\overline \rho_{\mathfrak m}: G_F \longrightarrow \GL_d(\overline \Fm_l)$$
is an irreducible representation which is unramified at all places of $F$ 
lying above primes which do not split in $E$ and which satisfies the following
hypothesis:
\begin{itemize}
%


%
%
\item $\zeta_l \not \in F$;

\item $\overline \rho_{\mathfrak m,|G_{F(\zeta_l)}}$ is irreducible;

\item $l>2(d+1)$;

\item there exists a automorphic lifting of $\overline \rho_{\mathfrak m}$ such that for 
every $w~|~ l$, its restriction at $w$ is potentially diagonalizable and has distincts Hodge-Tate
numbers for every $\tau:F_w \hookrightarrow \overline \Qm_l$.
\end{itemize}
Then Ihara's lemma of the conjecture \ref{conj-ihara1} is true, i.e. 
every irreducible $\GL_d(F_v)$-submodule of 
$S_{\overline G}(\overline K^v,\overline \Fm_l)_{\mathfrak m}$ is generic.
\end{thm}

\begin{itemize}
%
%
\item The last four hypothesis come from theorem 
4.4.1 of \cite{BLGGT} which is
some level raising and lowering statement, cf. theorem \ref{thm-raising}. Any other similar
statement, for example theorem 5.1.5 of \cite{gee-annalen}, 
with different hypothesis can then be used to modify the hypothesis of the theorem above.

\item The hypothesis $q_v \equiv 1 \mod l$ with $l>d$ is essentially used  
to avoid cuspidal $\overline \Fm_l$-representations which are not supercuspidal:
it is useful 
\begin{itemize}
\item to prove, cf. \S \ref{para-filt-strict}, that the graded parts of the filtration given by the nearby cycles
spectral sequence is strict,

\item and to construct, cf. \S \ref{para-mono}, 
the $\overline \Zm_l$ version of the geometric monodromy operator.
\end{itemize}
In particular the statement should also true in the banal case with the same other hypothesis
which are quite different to those of \cite{boyer-ihara} where one asks that 
$\overline \rho_{\mathfrak m,v}$ is multiplicity free.
\end{itemize}


\medskip

The basic idea\footnote{this explains the hypothesis on the existence of $p_0$ in the statement}, cf. \S \ref{para-geo},  as in \cite{boyer-ihara},
is to introduce geometry and move from the
Shimura variety associated to $\overline G$ which is of dimension zero,
to another Shimura variety $\sh_{K}$
associated to some reductive group $G$ and level $K$ with $K^{v_1}=\overline K^{v_1}$. 
This KHT-Shimura variety is now 
of strictly positive dimension and, if $K_{v_1}$ is small enough depending of $\overline K_{v_1}$,
then, cf. lemma 2.3.1 of \cite{boyer-ihara}, $S_{\overline G}(\overline K,\overline \Fm_l)$
appears in the middle degree cohomology group of $\sh_{K}$ with coefficient in $\overline \Fm_l$. 

To compute this $\overline \Fm_l$-cohomology group, the idea is to look at the 
cohomology over $\overline \Zm_l$ and then take modulo $l$
reduction. For this we will need to fix the following data.
\begin{nota}
Consider a coefficient field $L$ which is a large enough
finite extension of $\Qm_l$, with ring of integers $\OC_L$ and residue field 
$\OC_L/\varpi_F\OC_L \Fm_L$ for some fixed uniformizer $\varpi_L$.
\end{nota}
We then want to construct a filtration of the middle cohomology group of $\sh_K$
with coefficients in $\OC_L$ so that the graded parts, which are expected 
to be more easy to handle with, all verify the genericity 
property of their irreducible sub-spaces.
More explicitly we study the middle degree cohomology group with coefficients
in $\overline \Zm_l$, of the
KHT Shimura variety $\sh_{K^v(\oo)}$ associated to some similitude group
$G/\Qm$ such that $G(\Am_\Qm^{\oo,p}) \simeq \overline G(\Am_\Qm^{\oo,p})$,
cf. \S \ref{subsec-KHT} for more details, and with level $K^v(\oo):=\overline K^v$ 
meaning finite level outside $v$ and infinite level at $v$. 
The localization at $\mathfrak m$ of the cohomology groups of $\sh_{K^v(\oo)}$
can be computed as the cohomology of the
geometric special fiber $\sh_{K^v(\oo),\bar s_v}$ 
of $\sh_{K^v(\oo)}$, with coefficient in the complex of nearby cycles 
$\Psi_{K^v(\oo),v}$.

The Newton stratification of $\sh_{K^v(\oo),\bar s_v}$ gives us a filtration of 
$\Psi_{K^v(\oo),v}$, cf. \cite{boyer-FT}, and so 
a filtration $\Fil^\bullet(K^v(\oo))$
of $H^{d-1}(\sh_{K^v(\oo),\bar \eta_v},\overline \Zm_l)_{\mathfrak m}$ 
and the main 
point of \cite{boyer-ihara} is to prove that the modulo $l$ reduction of 
each graded part of this filtration verifies
the Ihara property, i.e. each of their irreducible sub-space are generic.
To realize this strategy 
\begin{itemize}
\item we need first the cohomology groups of 
$\sh_{K^v(\oo)}$ to be torsion free: this point is now essentially settled by 
the main result of \cite{boyer-jep2}. 

\item More crucially the previous filtration
$\Fil^\bullet(K^v(\oo))$
should be strict, i.e. its graded parts have to be torsion free, cf. theorem \ref{thm-recall}.

\item For each of the graded parts, the $\overline \Qm_l$-cohomology can be described
by a set of automorphic representations, each of them giving an automorphic contribution.
We choose any numbering of this discrete set and for each $n$, we 
choose a coefficient field $L$ large enough to be able to separate the first $n$
automorphic contributions. The cohomology over $\OC_L$ gives us a lattice 
which then allows us to
deal with the modulo $l$ reduction meaning reduction modulo $\varpi_L \OC_L$.
\end{itemize}
\rem In the following we will just write modulo $l$ reduction for this construction.

It appears that the graded parts of $\Fil^\bullet(K^v(\oo))$ 
are parabolically induced and in the limit case when 
the order $q_v$ of the residue field is such that $q_v \equiv 1 \mod l$,
the socle of the modulo $l$ reduction of these parabolic induced representations are 
no more
irreducible and do not fulfill the Ihara property, i.e. some of their subspaces are not
generic. It then appears that we have at least
\begin{itemize}
\item to verify that the modulo $l$ reduction of the 
first non trivial graded part of $\Fil^\bullet(K^v(\oo))$ verifies 
the genericity property of its irreducible submodule. For this we 
need a level raising statement as theorem 5.1.5 in \cite{gee-annalen}, cf. theorem \ref{thm-raising}, or theorem 4.4.1 of \cite{BLGGT}.

\item Then we have to understand that
the extensions between the graded parts of 
$\Fil^\bullet(K^v(\oo)) \otimes_{\overline \Zm_l} \overline \Fm_l$
are non split. 
\end{itemize}
One problem about this last point is that the $\overline \Qm_l$-cohomology
is split. For any irreducible automorphic representation $\Pi$ of $G(\Am)$
cohomological for, say, the trivial coefficients, the $\overline \Zm_l$-cohomology
defines a lattice $\Gamma(\Pi)$ of $(\Pi^\oo)^{K^v(\oo)} \otimes \sigma(\Pi)_v$ 
whose modulo
$l$ reduction gives a subspace of the $\overline \Fm_l$-cohomology: Ihara's lemma
predicts that the socle of this subspace is still generic, i.e. it gives informations
about the lattice $\Gamma(\Pi)$. We then see that
non splitness of $\Fil^\bullet(K^v(\oo))\otimes_{\overline \Zm_l} \overline \Fm_l$ 
should be understood in a very flexible
point of view. We then naturally face to prove the following result.

\begin{prop} (cf. \ref{prop-BLC2}) 
The contribution $\Gamma(\Pi)$ of an automorphic representation $\Pi$
to the cohomology viewed as a subrepresentation, 
defines a stable lattice of $\Pi_v$ uniquely defined by the property
that the socle of its modulo $l$ reduction is irreducible and generic.
\end{prop}

In particular this lattice should depend only on $\Pi_v$ and not on 
the global representation $\Pi$. This statement looks similar to the Breuil lattice conjecture 
which is stated when $l=p$ and $K_v$-types. We then first prove a simple version
of this result, cf. proposition \ref{prop-BLC} which can be stated as follows.
\begin{itemize}
\item Consider some fixed $K_v$-type $\sigma_{v,\overline \Qm_l}$
\item and a system $\lambda$ of Hecke eigenvalues associated to some automorphic
representation $\Pi$ as above 
\item such that $\sigma_{v,\overline \Qm_l}$ appears with multiplicity one in $\Pi_v$.
\end{itemize}
Then the lattice of $\sigma_{v,\overline \Qm_l} \cap \Gamma(\Pi)$
depends only on the modulo $l$ reduction of $\lambda$.
One combinatorial problem is then to recover the information on the $\GL_d(F_v)$-lattices
from this vague obversation, cf. \S \ref{para-generalcase}.
The idea to do so is to start from the filtration $\Fil^\bullet(K^v(\oo))$ 
coming from a filtration of the nearby cycles perverse sheaf and modify it step by step,
cf. \S \ref{para-filtration},
until we arrive at an automorphic filtration, cf. \S \ref{para-s3}, i.e. where the graded parts correspond
to the contribution of some automorphic representation.
The main ingredient to construct modifications of filtrations is to consider
situations as illustrated in the figure \ref{fig-exchange}.
\begin{itemize}
\item A filtration $\Fil^\bullet$ of 
$H^{d-1}(\sh_{K^v(\oo),\bar \eta_v},\overline \Zm_l)_{\mathfrak m}$ whose
graded parts $\gr^\bullet$ are torsion free;

\item let $k$ and $X:=\Fil^k/\Fil^{k-2}$ such that 
$X \otimes_{\overline \Zm_l} \overline \Qm_l \simeq
( gr^{k-1} \otimes_{\overline \Zm_l} \overline \Qm_l) \oplus 
(\gr^{k} \otimes_{\overline \Zm_l} \overline \Qm_l)$.

\item We can then define 
$\widetilde{\gr^k}:=X \cap (\gr^k \otimes_{\overline \Zm_l} \overline \Qm_l)$
and the quotient $\widetilde{\gr^{k-1}}$: note that $T$ is torsion. 
We choose a $L$ finite over $\Qm_l$,
so that this diagram is defined over $\OC_L$ and which allows to look
modulo $l$, meaning modulo $\varpi_L \OC_L$, where we then obtained 
a priori two distinct filtrations.
\end{itemize}

Let us first explain why something interesting should happen during this process.
\\
(a) We can define a $\overline \Fm_l$-monodromy operator for the Galois 
action at the place $v$.\footnote{Note that over $\overline \Fm_l$ the usual arithmetic 
approach for defining the nilpotent monodromy operator, 
is hopeless because, up to consider a finite extension of $F_v$, such
a $\overline \Fm_l$-representation has a trivial action of the inertia group.}
We are looking for a geometric monodromy operator $N^{geo}_v$ which then exists
whatever are the coefficients, $\overline \Qm_l$, $\overline \Zm_l$
and $\overline \Fm_l$, compatible with tensor products.
One classical construction is known in the semi-stable
reduction case, cf. \cite{ill} \S 3, which corresponds to the case where the level at $v$ of our Shimura
variety is of Iwahori type.\footnote{This corresponds to automorphic representations
$\Pi$ such that the cuspidal support of $\Pi_v$ is made of unramified characters, 
and so with the weak form of Ihara's lemma of definition \ref{defi-ihara1}.}
Using our knowledge of the $\overline \Zm_l$-nearby cycles described
completely in \cite{boyer-duke}, we can construct such a geometric
nilpotent monodromy operator which generalizes the semi-stable
case,
cf. \S \ref{para-mono}.
%
%
\\
(b) Taking this geometric monodromy operator, we then obtain
a cohomological monodromy operator $N_{v,\mathfrak m}^{coho}$ acting on
$H^0(\sh_{K,\bar s_v},\Psi_{K^v(\oo),v})_{\mathfrak m}$
One of the main point, cf theorem \ref{thm-recall}, is that the graded parts of the
filtration of $H^0(\sh_{K,\bar s_v},\Psi_{K^v(\oo),v})_{\mathfrak m}$ induced by
the Newton filtration on
the nearby cycles spectral sequence, are all torsion free, so that in particular
we are in position to understand quite enough the action of
$\overline N^{coho}_{v,\mathfrak m}:= N^{coho}_{v,\mathfrak m} \otimes_{\overline \Zm_l} \overline \Fm_l$ on
$H^0(\sh_{K,\bar s_v},\Psi_{K^v(\oo),v})_{\mathfrak m} \otimes_{\overline \Zm_l} \overline \Fm_l$,  and prove
that its nilpotency order is as large as possible.
\\
(c) Note that as $\overline \rho_{\mathfrak m}$ is supposed to be irreducible, then
the modulo $l$ reduction of the monodromy operator acting on 
$\rho_{\widetilde{\mathfrak m}}$ does not depend on the choice of the prime ideal
$\widetilde{\mathfrak m} \subset \mathfrak m$ so that it is usually trivial.
\begin{figure}[!ht]
$$\xymatrix{
& \widetilde{\gr^k} \ar@{^{(}->}[d] \ar@{^{(}->}[dr] \\
\gr^{k-1} \ar@{^{(}->}[r] \ar@{^{(}->}[dr] & X \ar@{->>}[r] \ar@{->>}[d] & 
\gr^k \ar@{->>}[dr] \\
& \widetilde{\gr^{k-1}} \ar@{->>}[dr] & & T \\
& & T \ar@{=}[ur]
}$$
\caption{Exchange process}\label{fig-exchange}
\end{figure}
Finally, as $N_v^{coho} \otimes_{\overline \Zm_l} \overline \Fm_l$ is far from being
trivial, there should be non split extensions between the graded parts of
%
The heart of our proof is then divided in three main steps:
\begin{itemize}
\item prove that the lattice of $\Pi_v$ induced by $\Gamma(\Pi)$ depends only on
the modulo $l$ reduction of the system of Hecke eigenvalues of a globalization $\Pi$ of $\Pi_v$;

\item using the modulo $l$ monodromy operator $\overline N_v$, prove that this lattice
verifies the Ihara property;

\item prove that the graded parts of our final filtration, which do not verify Ihara's property
can not give a non generic subspace of the all cohomology.
\end{itemize}

To conclude this long introduction, note that
Ihara's lemma in Clozel-Harris-Taylor formulation, was stated in order to be able to
do level raising. In our proof we use level raising statements, proved thanks to Taylor's Ihara
avoidance in \cite{taylor-avoidance}, in order to prove Ihara's lemma. 
Then we can see our arguments as the proof that
\textit{level raising implies Ihara's lemma in the limit case}.

%% file: preliminaries.tex
\section{Preliminaries}

\subsection{Representations of $\GL_d(M)$}
\label{para-gen}

Consider a finite extension $M/\Qm_p$ with residue field $\Fm_q$. 
We denote by $|-|$ its absolute value.
For a representation $\pi$ of $\GL_d(M)$ and $n \in \frac{1}{2} \Zm$, set 
$$\pi \{ n \}:= \pi \otimes q^{-n \val \circ \det}.$$

\begin{nota} \label{nota-ind}
For $\pi_1$ and $\pi_2$ representations of respectively $\GL_{n_1}(M)$ and
$\GL_{n_2}(M)$, we will denote by
$$\pi_1 \times \pi_2:=\ind_{P_{n_1,n_1+n_2}(M)}^{\GL_{n_1+n_2}(M)}
\pi_1 \{ \frac{n_2}{2} \} \otimes \pi_2 \{-\frac{n_1}{2} \},$$
the normalized parabolic induced representation where for any sequence
$\underline r=(0< r_1 < r_2 < \cdots < r_k=d)$, we write $P_{\underline r}$ for 
the standard parabolic subgroup of $\GL_d$ with Levi
$$\GL_{r_1} \times \GL_{r_2-r_1} \times \cdots \times \GL_{r_k-r_{k-1}}.$$ 
\end{nota}

Recall that a representation
$\varrho$ of $\GL_d(M)$ is called \emph{cuspidal} (resp. \emph{supercuspidal})
if it is not a subspace (resp. subquotient) of a proper parabolic induced 
representation. When the field of coefficients is of characteristic zero, these 
two notions coincides, but this is no more true over $\overline \Fm_l$.

\begin{defin} \label{defi-LT}
\label{defi-rep} (see \cite{zelevinski2} \S 9 and \cite{boyer-compositio} \S 1.4)
Let $g$ be a divisor of $d=sg$ and $\pi$ an irreducible cuspidal 
$\overline \Qm_l$-representation of $\GL_g(M)$. 
The induced representation
\begin{equation} \label{eq-ind-rep}
\pi\{ \frac{1-s}{2} \} \times \pi \{ \frac{3-s}{2} \} \times \cdots \times \pi \{ \frac{s-1}{2} \}
\end{equation}
holds an unique irreducible quotient (resp. subspace) denoted $\st_s(\pi)$ (resp.
$\speh_s(\pi)$); it is a generalized Steinberg (resp. Speh) representation.
Their cuspidal support is the Zelevinsky segment 
$$[\pi \{ \frac{1-s}{2} \},\pi \{ \frac{s-1}{2} \}]:=\Bigl \{ 
\pi\{ \frac{1-s}{2} \},  \pi \{ \frac{3-s}{2} \}, \cdots, \pi \{ \frac{s-1}{2} \} \Bigr \}.$$

\end{defin}
%
%
More generally the set of sub-quotients of the induced representation (\ref{eq-ind-rep})
is in bijection with the following set.
$$\Dec(s)=\{ (t_1,\cdots,t_r), \hbox{ such that } t_i \geq 1 \hbox{ and }
\sum_{i=1}^r t_i=s \}.$$ 
For any $\underline s \in \Dec(s)$, we then denote by $\st_{\underline s}(\pi)$
the associated irreducible sub-quotient of (\ref{eq-ind-rep}).
Following Zelevinsky, we fix this bijection such that $\speh_s(\pi)$ corresponds
to $(s)$ and $\st_s(\pi)$ to $(1,\cdots,1)$. The Lubin-Tate representation
$LT_{h,s}(\pi)$ will also appear in the following, it corresponds with
$(\overbrace{1,\cdots,1}^h,s-h)$.

\begin{defin} For $1 \leq s-1$, we say that $\underline s \in \Dec(s)$ is $t$-small if
the number of consecutive $1$ is less than $t$. We say that an irreducible subquotient of 
(\ref{eq-ind-rep}) is $(\pi,t)$-small if its parameter $\underline s \in \Dec(s)$
is $t$-small. 
\end{defin}

\begin{prop} \label{prop-red-modl} (cf. \cite{vigneras-livre} III.5.10)
Let $\pi$ be an irreducible cuspidal representation of $\GL_g(M)$ with a stable
$\overline \Zm_l$-lattice\footnote{As before, we fix $L/\Qm_l$ large enough and consider an $\OC_L$-lattice: we say that $\pi$ is integral.}, then its modulo $l$ 
reduction is irreducible and cuspidal (but not necessary supercuspidal).
\end{prop}

We now suppose as explained in the introduction that 
$$q \equiv 1 \mod l \quad \hbox{ and } \quad l>d$$ 
so the following facts are verified (cf. \cite{vigneras-livre} \S III):
\begin{itemize}
\item the modulo $l$ reduction of every irreducible 
cuspidal representation of $\GL_g(M)$ for $g \leq d$, is supercuspidal\footnote{In the
banal case this is not always the case but it is when the cuspidal support contains only
characters.}:
with the notation of \cite{boyer-repmodl} proposition 1.3.5, $m(\varrho)=l>d$ for any irreducible $\overline \Fm_l$-supercuspidal representation $\varrho$.

\item For a $\overline \Fm_l$-irreducible supercuspidal representation $\varrho$
of $\GL_g(M)$, 
the parabolic induced representation $\varrho \times \cdots \times \varrho$,
with $s$ copies of $\varrho$,
is semi-simple with irreducible constituants\footnote{some of them might be isomorphic to each others} the modulo $l$ reduction of the set
of elements of $\{ \st_{\underline s}(\pi) \hbox{ such that } \underline s \in \Dec(s) \}$, 
where $\pi$ is any cuspidal representation whose modulo $l$ reduction is isomorphic to
$\varrho$.
\end{itemize}
%
%
%
%
%

\medskip

Concerning the notion of genericity, consider the mirabolic subgroup
$\Mm_d(M)$ of $\GL_d(M)$ as the set of matrices with last row $(0,\cdots,0,1)$: we denote
by
$$\Vm_d(M)=\{ (m_{i,j}) \in \Mm_d(M):~m_{i,j}= \delta_{i,j} \hbox{ for } j < d \}.$$ 
its unipotent radical. We fix a non trivial character $\psi$ of $L$ and let $\theta$ be the
character of $\Vm_d(M)$ defined by $\theta( (m_{i,j}))=\psi(m_{d-1,d})$.
For $G=\GL_r(M)$ or $\Mm_r(M)$, we denote by $\alg(G)$ the abelian category of smooth
representations of $G$ and, following \cite{zelevinski1}, we introduce
$$\Psi^-: \alg(\Mm_d(M)) \longrightarrow \alg(\GL_{d-1}(M)),$$
and
$$\Phi^-: \alg (\Mm_d(M)) \longrightarrow \alg (\Mm_{d-1}(M)),$$
defined by $\Psi^-=r_{\Vm_d,1}$ (resp. $\Phi^-=r_{\Vm_d,\theta}$) the functor of $\Vm_{d}$
coinvariants (resp. $(\Vm_{d},\theta)$-coinvariants), cf. \cite{zelevinski1} 1.8.
For $\tau \in \alg(\Mm_d(M))$, the representation 
$$\tau^{(k)}:=\Psi^- \circ (\Phi^-)^{k-1}(\tau)$$
is called the $k$-th derivative of $\tau$. If $\tau^{(k)}\neq 0$ and $\tau^{(m)}=0$
for all $m > k$, then $\tau^{(k)}$ is called the highest derivative of $\tau$. 
In the particular case where
$k=d$, there is an unique irreducible representation $\tau_{nd}$ of $\Mm_d(M)$
with derivative of order $d$.

\begin{defin} An irreducible representation $\pi$ of $\GL_d(M)$ is said generic,
if its restriction to the mirabolic subgroup admits $\tau_{nd}$ as a subquotient.
\end{defin}

Let $\pi$ be an irreducible generic $\overline \Qm_l$-representation of $\GL_d(M)$ and consider
any stable lattice which gives us by modulo $l$ reduction a $\overline \Fm_l$-
representation uniquely determined up to semi-simplification. Then this modulo $l$
reduction admits an unique generic irreducible constituant.

\begin{lem} \label{lem-tsmall}
Let $1 \leq t \leq s$ and consider 
\begin{itemize}
\item $\underline s=(l_1 \geq \cdots \geq l_r)$ a partition of $s$ with $l_1 < t$;

\item $\pi_1, \cdots,\pi_r$ irreducible cuspidal representations whose modulo $l$ reductions
are isomorphic to $\varrho$.
\end{itemize}
There exists an irreducible subquotient $\tau$ of 
$\st_{l_1} (\pi_1) \times \cdots \times \st_{l_r}(\pi_r)$ such that
\begin{itemize}
\item whatever is $\underline s'=(l'_1 \geq \cdots \geq l'_{r'})$ with $l'_1 \geq t$,

\item and $\pi'_1,\cdots, \pi'_{r'}$ irreducible cuspidal representations with modulo $l$
reduction isomorphic to $\varrho$,
\end{itemize}
then $\tau$ is not a subquotient of the modulo $l$ reduction of
$\st_{l'_1} (\pi'_1) \times \cdots \times \st_{l'_r}(\pi'_{r'})$.
\end{lem}

\begin{proof}
Note that $\st_{l_1} (\pi_1) \times \cdots \times \st_{l_r}(\pi_r)$ has the same modulo $l$
reduction as 
\begin{equation} \label{eq-st-line}
\st_{l_1} (\pi_1\{ \frac{l_1-s}{2} \}) \times \st_{l_2}(\pi_2 \{ \frac{l_2-s}{2} + l_1-1 \} \times
\cdots \times \st_{l_r}(\pi_1 \{ \frac{l_r-s}{2}+ l_1+\cdots+l_{r-1} \})
\end{equation}
where the shifts are chosen so that $\st_{l_1} (\pi\{ \frac{l_1-s}{2} \}) \times
\cdots \times \st_{l_r}(\pi \{ \frac{s-l_r}{2} \})$ is the subquotient of (\ref{eq-ind-rep})
associated to 
$$(\overbrace{1,\cdots,1}^{l_1-1},2,\overbrace{1,\cdots,1}^{l_1-1},2,\cdots,
\overbrace{1,\cdots,1}^{l_r-1}) \in \Dec(s).$$
Note that this irreducible constituant of (\ref{eq-ind-rep}) is the less non degenerate 
subquotient and we denote by $\tau$ is its modulo $l$ reduction which remains
irreducible.

The property stated in the lemma then follows from the fact that whatever is an
irreducible subquotient of the modulo $l$ reduction of
$\st_{l'_1} (\pi'_1) \times \cdots \times \st_{l'_r}(\pi'_{r'})$, it has a non zero derivative of
order $l'_1$, and so a non zero derivative of order $t$, while the derivative of order $t$
of $\tau$ is zero.

\end{proof}

\begin{defin} \label{defi-rhotsmall}
Elements $\tau$ constructed in the above lemma will be said to be
$(\varrho,t)$-small.
\end{defin}

Consider any non degenerate irreducible representation
$\Pi:=\st_{l_1} (\pi_1) \times \cdots \times \st_{l_r}(\pi_r)$ where the modulo $l$ 
reduction of $\pi_1,\cdots,\pi_r$ is isomorphic to $\varrho$. As the modulo $l$ reduction
of $\Pi$ contains an unique irreducible non degenerate subquotient, there exists then
an unique stable lattice such that its modulo $l$ reduction has
an irreducible generic socle $\tau_{gen}$, 
cf. \cite{EGS} lemma 4.1.1. 

\begin{nota}
We denote by $\Gamma(\Pi)^{gen}$ the lattice of $\Pi$ for which the socle of its
modulo $l$ reduction is generic.
\end{nota}
%

\subsection{Weil--Deligne inertial types}
\label{sub:WD:types}

Recall that a Weil-Deligne representation of $W_M$ is a pair $(r,N)$ where
\begin{itemize}
\item $r:W_M \longrightarrow \GL(V)$ is a smooth\footnote{i.e. continuous for
the discrete topology on $V$} representation on a finite dimensional  vector
space $V$; and

\item $N \in \End(V)$ is nilpotent such that
$$r(g) N r(g)^{-1} = ||g|| N,$$
where $||\bullet ||: W_M \rightarrow W_M/I_M \twoheadrightarrow q^\Zm$
takes an arithmetic Frobenius element to $q$.
\end{itemize}

\rem To a continuous\footnote{relatively to the $l$-adic topology
on $V$} representation on a finite dimensional $\Qm_l$-vector space $V$,
$\rho:W_M \longrightarrow \GL(V)$ is attached a Weil-Deligne representation
denoted by $\WD(\rho)$. A Weil representation of $W_M$ is also said of
Galois type if it comes from a representation of $G_M$.

\textit{Main example}: let $\rho:W_M \longrightarrow \GL(V)$ be a smooth
irreducible representation on a finite dimensional vector space $V$. For
$k \geq 1$ an integer, we then define a Weil-Deligne representation
$$\Sp(\rho,k):=\bigl ( V\oplus V(1) \oplus \cdots \oplus V(k-1), N \bigr ),$$
where for $0 \leq i \leq k-2$, the isomorphism $N: V(i) \simeq V(i+1)$ is induced 
by some choice
of a basis of $\overline L(1)$ and $N_{|V(k-1)}$ is zero.
Then every Frobenius semi-simple Weil-Deligne representation of $W_M$
is isomorphic to some $\bigoplus_{i=1}^r \Sp(\rho_i,k_i)$, for
smooth irreducible representations $\rho_i:W_M \longrightarrow \GL(V_i)$
and integers $k_i \geq 1$. Up to obvious reorderings, such a writing is unique.

Let now $\rho$ be a continuous representation of $W_M$, or its
Weil-Deligne representation $\WD(\rho)$, and consider its restriction to $I_M$,
$\tau:=\rho_{|I_M}$. Such an isomophism class of a finite dimensional continuous
representation of $I_M$ is then called \emph{an inertial type}.

\begin{nota} Let $\mathcal{I}_0$ the set of inertial types that extend to a
continuous irreducible representation of $G_M$.
\end{nota}

\rem $\tau \in \mathcal{I}_0$ might not be irreducible.

Let $\Part$ be the set of decreasing sequences of positive integers
$$\underline d=(\underline d(1) \geq \underline d(2) \geq \cdots)$$ 
viewed as a 
partition of $\sum \underline d:=\sum_i \underline d(i)$. We also denote by
$\Part(s)$ the set of partition of $s$: $\Part=\coprod_{s \geq 1} \Part(s)$.
%
%
%
%
%
%

\begin{nota} \label{nota-in-type}
Let $f:\mathcal{I}_0 \longrightarrow \Part$ with finite support. We then denote by
$\tau_f$ the restriction to $I_M$ of
$$\bigoplus_{\tau_0 \in \mathcal{I}_0} \bigoplus_{i}
\Sp(\rho_{\tau_0},f(\tau_0)(i)),$$
where $\rho_{\tau_0}$ is a fixed extension of $\tau_0$ to $W_M$.
\end{nota}

\rem By lemma 3.3 of \cite{ShMa} the isomorphism class of $\tau_f$
is independent of the choices of the $\rho_{\tau_0}$. 

The map from $\{ f:\mathcal{I}_0 \longrightarrow \Part \}$ to
the set of inertial types given by $f \mapsto \tau_f$, is a bijection.
The dominance order $\preceq$ on $\Part$ 
induces a partial order on the set of inertial types.

%
%
%
%
%


%
%
%
%
%
%

We let $\mathrm{rec}_M$ denote the local reciprocity map of \cite[Theorem A]{harris-taylor}.
Fix an isomorphism $\imath: \overline{\Qm}_\ell\risom \Cm$.
We normalize the local reciprocity map $\mathrm{rec}$ of \cite[Theorem A]{harris-taylor}, defined on isomorphism classes of irreducible smooth representations of $\GL_n(M)$ over $\Cm$ as follows: if $\pi$ is the isomorphism class of an irreducible smooth representation of $\GL_n(M)$ over $\overline{\Qm}_\ell$, then
\[
\rho_\ell(\pi)\defeq \imath^{-1}\circ\mathrm{rec}_M\circ\imath(\pi\otimes_{\overline{\Qm}_\ell}|\det|^{(1-n)/2}).
\]
Then $\rho_\ell(\pi)$ is the isomorphism class of an $n$-dimensional, Frobenius semisimple Weil--Deligne representation of $W_M$ over $\overline{\Qm}_\ell$, independent of the choice of $\imath$.
Moreover, if $\rho$ is an isomorphism class of an $n$-dimensional, Frobenius semisimple Weil--Deligne representation of $W_M$ over $M$, then $\rho_{\ell}^{-1}(\rho)$ is defined over $M$ (cf.~\cite[\S 1.8]{CEGGPS}).

Recall the following compatibility of the Langlands correspondence.

\begin{lem}
If $\pi$ and $\pi'$ are irreducible generic representations of $\GL_d(M)$ such
that $\rho_\ell(\pi)|I_M \simeq \rho_\ell(\pi')|I_M$ then 
$\pi_{|\GL_d(\mathcal{O}_M)} \simeq \pi'_{|\GL_d(\mathcal{O}_M)}.$
\end{lem}
%
%
\begin{thm}\label{thm:ILL} (cf. \cite{B-C} proposition 6.3.3 or \cite{Shotton-n}
theorem 3.7) \\
Let $G = \GL_n$.
Let $\tau$ be a 
inertial type for $M$.
Then there is a smooth irreducible $\GL_n(\OC_M)$-representation $\sigma(\tau)$ over $E$ such that
for any irreducible admissible representation $\pi$ of $\GL_n(M)$, one has:
\begin{enumerate}
\item if $\pi|_{\GL_n(\OC_M)}$ contains $\sigma(\tau)$ then $\rho_{\ell}(\pi)|_{I_M} \preceq \tau$;
\item if $\rho_{\ell}(\pi)|_{I_M} = \tau$, then $\pi|_{\GL_n(\OC_M)}$ contains $\sigma(\tau)$ with multiplicity one; 
\item if $\rho_{\ell}(\pi)|_{I_M} \preceq \tau$ and $\pi$ is generic, then $\pi|_{\GL_n(\OC_M)}$ contains $\sigma(\tau)$ and the multiplicity is one if furthermore $\tau$ is maximal with respect to $\preceq$.
\end{enumerate}
\end{thm}

\rem 
For example take $\tau_0$ the trivial representation and consider the following
partitions $\underline d=(1 \geq 1 \geq \cdots \geq 1)$
(resp. $(d)$) denoted also by $(1^d)$. Denote then by $\tau$ the associated
inertial type of notation \ref{nota-in-type}. Then $\sigma(\tau)$ is
the trivial representation (resp. is inflated from the Steinberg representation
of $\GL_d(\kappa)$). Note also that $\pi$ contains $\sigma(\tau)$ if and only if
$\pi$ is unramified (resp. it implies that $r_{\ell}(\pi)|I_M$ is unipotent).

We need more details about the construction of $\sigma(\tau)$ which rests on the notion
of SZ-stratum of \cite{SZ}. We first recall quickly the basic
notions of type theory of Bushnell and Kutzko. Let $\Km$ denote 
$\overline \Qm_l$ or  $\overline \Fm_l$. Let $V$ be a vector space over $F$ and let
$G=\Aut_F(V)$ and $A=\End_F(V)$. 

\noindent - An $\OC_F$-lattice chain in $V$ is a sequence $\LC=(\Lambda_i)_{i \in \Zm}$
of $\OC_F$-lattices in$ V$ such that $\Lambda_{i+1} \subset \Lambda_i$ for all 
$i \in \Zm$
and such that there exists an period $e \geq 1$ with $\Lambda_{i+i}=\PF_M \Lambda_i$
for all $i \in \Zm$.

\noindent - The hereditary $\OC_M$-orders in $A$ are those orders $\UF$ that arise as the stabiliser 
of some $\OC_F$-lattice chain; it is maximal if and only if it stabilises a lattice
chain of period $e=1$.

\noindent - A hereditary order $\UF \subset A$ has a unique two-sided maximal ideal
$\PF=\{ x \in \UF: x \Lambda_i \subset \Lambda_{i+1} \hbox{ for all } i \in \Zm \}$.
We write $U(\UF)$ for the group of units in $\UF$ and $U^1(\UF):=1+\PF$.

\begin{defin} A simple strata in $A$ is $[\UF,m,0,\beta]$ where
\begin{itemize}
\item $\UF$ is a hereditary $\OC_K$-order in $A$;

\item $m>0$ is an integer;

\item $\beta \in \PF^{-m} \setminus \PF^{1-m}$ is such that $E:=L[\beta]$ is a field and
$E^\times$ is contained in the normalisez of $U(\UF)$;

\item the number $k_0(\beta,\UF(E))$ defined in \cite{BK93} \S 1.4 is strictly negative.
\end{itemize}
\end{defin}
The chain lattice defining $\UF$ can be seen as an $\OC_E$-lattice chain and we denote
$\BF:=\UF \cap \End_E(V)$ its stabiliser. We define the groups $U(\BF)$ and
$U^1(\BF)$ as for $\UF$.
In \cite{BK93}, the authors associate to the simple stratum $[\UF,m,0,\beta]$, 
compact open subgroups $J=J(\beta,\UF)$, $J^1=J^1(\beta,\UF)$ and $
H=H^1(\beta,\UF)$ of $U(\UF)$ such that:
\begin{itemize}
\item $J^1$ is a normal prop-$p$- subgroup of $J$;

\item $H^1$ is a normal subgroup of $J^1$;

\item $U(\BF) \subset J$ and $U^1(\BF) \subset J^1$ and the induced map
$$U(\BF)/U^1(\BF) \longrightarrow J/J^1$$
is a isomorphism.
\end{itemize}
By \cite{BK93} \S 3.2,
there is a set $\CC(\UF,0,\beta)$ of simple characters of $H^1(\beta,\UF)$. For
$\theta \in \CC(\UF,0,\beta)$ there is a unique irreducible representation
$\eta$ of $J^1(\beta,\UF)$ whose restriction to $H^1(\beta,\UF)$ contains $\theta$.
There is then a distinguished class of $\beta$-extensions $\kappa$ of $\eta$ to 
$J(\beta,\UF)$.

It is the main result of \cite{BK99} that every Berstein component of $\Rep_\Km(G)$
has a type and give an explicit construction of them. When $\Omega$ is
supercuspidal then they construct a type $(J,\lambda)$ such that
$J=J(\beta,\UF)$ for a simple stratum $[\UF,m,0,\beta]$ in $A$ in which $\BF$ is a
maximal $\OC_E$-order and $\lambda=\kappa \otimes \nu$ where $\kappa$
is a $\beta$-extension of an irreducible representation $\eta$ containing a simple
character $\theta \in \CC(\UF,0,\beta)$ and $\nu$ is a cuspidal representation
of $J/J^1 \simeq \GL_{n/[E:L]}(k_E)$. This type is called maximal.

From \cite{BK99} \S 4, the character $\theta$ determines a ps-character 
$(\Theta,0,\beta)$ viewed as a function $\Theta$ on the set of simple strata 
$[\UF,m,0,\beta]$ taking such a stratum to an element $\Theta(\UF) \in \CC(\UF,0,\beta)$.
By \cite{BK99} \S 4.5, the endo-class of this ps-character is determine by $\Omega$.

If $[M,\pi]$ is an inertial equivalence class of supercuspidal pair corresponding to a
Berstein component $\Omega$ of $\Rep_\Km(G)$, then let $(J_M,\lambda_M)$
be a maximal type for the supercuspidal Berstein component $\Omega_M$ of
$\Rep_\Km(M)$ containing $\pi$. By the results of \cite{BK99}, there is, what is called
a $G$-cover, $(J,\lambda)$ of $(J_M,\lambda_M)$ which is a type for $\Omega$.

\begin{defin} (cf. \cite{Shotton-n} definition 6.11) \\
A SZ-datum over $\Km$, is a set
$$\SF=\Bigl \{ (E_i,\beta_i,V_i,\BF_i,\lambda_i), ~i=1,\cdots,r \Bigr \}$$
where $r$ is a positive integer and for each $i=1,\cdots,r$ we have
\begin{itemize}
\item $E_i/F$ is a finite extension generated by an element $\beta_i \in E_i$,

\item $V_i$ is an $E_i$-vector space of finite dimension $N_i$,

\item $\BF_i \subset \End_{E_i}(V_i)$ is a maximal hereditary $\OC_{E_i}$-order,

\item let denote $\UF_i$ the associated $\OC_M$-order in $A_i:=\End_M(V_i)$ and
let $m_i:=-v_{E_i}(\beta_i)$. Then $[\UF_i,m_i,0,\beta_i]$ is a simple stratum and
$\lambda_i$ is an $\Km$-representation of $J_i:=J(\beta_i,\UF_i)$ of the form
$\kappa_i \otimes \nu_i$ where $\kappa_i$ is a $\beta_i$-extension of the representation
$\eta_i$ of $J^1_i:=J^1(\beta_i,\UF_i)$ containing some simple character
$\theta_i \in \CC(\UF_i,0,\beta_i)$ of $H^1_i=H^1(\beta_i,\UF_i)$ and $\nu_i$
is an irreducible representation of $U(\BF_i)/U^1(\BF_i) \simeq \GL_{N_i}(k_{E_i})$
over $\Km$;

\item no two of the $\theta_i$ are endo-equivalent in the sense of \cite{BK99} \S 4.
\end{itemize}
\end{defin}

Let $V=\bigoplus_{i=1}^r V_i$, $A=\End_M(V)$ and $G=\Aut_M(V)$.
The Levi subgroup $M=\prod_{i=1}^r \Aut_M(A_i) \subset G$ has compact open
subgroups $J^1_M \triangleleft J_M$ where $J^1_M=\prod_{i=1}^r J^1_i$
and similarly for $J_M$. Let denote by $\eta_M=\bigotimes_{i=1}^r \eta_i$, 
a representation of $J^1_M$ and similarly for the representations $\kappa_M$ and
$\lambda_M$ of $J_M$. Note that $\eta_M$ and $\kappa_M$ are clearly irreducible
and $\lambda_M$ is irreducible by \cite{Shotton-n} propositions 6.12.

Since no two of the $\theta_i$ are endo-equivalent, by \cite{BK99} \S 8,
see also \cite{MS14} proposition 2.28,
we have compact open subgroups $J$ and $J^1$ of $G$ and representations
$\eta$ of $J^1$, $\kappa$ of $J$ and $\lambda$ of $J$ such that
$(J^1,\eta)$ (resp. $(J,\kappa)$, resp. $(J,\lambda)$) is a
$G$-cover of $(J^1_M,\eta_M)$ (resp. $(J_M,\kappa_M)$, resp. $(J_M,\lambda_M)$)
and $J/J^1=J_M/J^1_M$ with $\lambda=\kappa \otimes (\bigotimes_{i=1}^r \nu_i)$
under this identification.

\begin{defin}
For $\SF$ an SZ-stratum and $J$ and $\lambda$ as above. Let $K$ be a maximal
compact subgroup of $G$ such that $J_M \subset K \cap M$. We then denote by
$$\sigma(\SF):=\ind_J^K(\lambda),$$
which by \cite{Shotton-n} theorem 6.16 is irreducible.
\end{defin}

Start now from $\PC \in \IC=\{ f:\IC_0 \longrightarrow \Part \}$ and let
$n=\sum_{\tau_0 \in \IC_0} \sum f(\tau_0) \dim \tau_0$. Let
$V$ be a $L$-vector space of dimension $N$ and let $G=\Aut_M(V)$.
Let $(M^0,\pi)$ be a supercuspidal pair in the inertial equivalence class
associated to the Bernstein component attached to $\PC$. Write
$M^0=\prod_{i=1}^t M_i^0$ where each $M_i^0$ is the stabiliser of some
$n_i$-dimensional subspace $V_i^0$ of $V$. Write $\pi=\bigotimes_{i=1}^t \pi_i$ and let
$\Omega_i$ be the supercuspidal Berstein component containing $\pi_i$. For each
$\Omega_i$ there is an associated endo-class of ps-character $\Theta_i^0=(\Theta_i^0,0,\beta_i^0)$. Let $M=\prod_{i=1}^r M_i$ the Levi subgroup
of $G$ obtaining from $M$ by gathering $M_j^0$ with $M_k^0$ if and only if
$\Theta_j^0=\Theta_k^0$: we then write simply $(\Theta_i,0,\beta_i)$ for the 
common value.

We attach to $\PC$ a SZ-datum 
$$\SF_\PC=\Bigl \{ (E_i,\beta_i,V_i,\BF_i,\kappa_i \otimes \nu_i ): i=1,\cdots,r \Bigr \}$$
as follows. Suppose first that $r=1$ and wrtie $(\Theta,0,\beta)$ for the common
value of $(\Theta_i^0,0,\beta_i^0)$.
\begin{itemize}
\item $E=L[\beta]$;

\item let $(J_i^0,\lambda_i^0)$ be the maximal simple type for the supercuspidal Berstein
component $\Omega_i^0$ with $J_i^0=J(\beta,UF_i^0)$ for a simple stratum
$[\UF_i^0,m_i^0,0,\beta_i^0]$ and $\lambda_i^0$ contains 
$\theta_i^0:=\Theta(\UF_i^0,0,\beta_i)$. As above we then 
have compact open subgroups $J^1 \subset J$ of $G$ and a representation $\eta$
of $J^1$ containing the simple character $\Theta(\UF,0,\beta)$ where $\UF$
is a hereditary $\OC_M$-order in $A$ and $\UC_i \cap B=\BF$ is a maximal hereditary 
$\OC_E$-order.

\item We choose compatible $\beta$-extensions, in the sense of \cite{Shotton-n} \S 6.6,
$\kappa_i^0$ of $\eta_i^0$ coming from a $\beta$-extension $\kappa$ of $\eta$, and
decompose each $\lambda_i^0$ as $\kappa_i^0 \otimes \nu_i^0$ where $\nu_i^0$
is a cuspidal representation of $J_i^0/J_i^{1,0}=U(\BF_i^0)/U^1(\BF_i^0) \simeq
\GL_{n_i^0/[E:L]}(k_E)$ for an integer $n_i^0$. Then $J/J^1 \simeq \GL_{n/[E:L]}(k_E)$
so that we can view each $\nu_i^0$ as an element of $\overline \IC_0$ and
define an element $\overline \PC \in \overline \IC$ by $\overline \PC(\nu_i^0)=\PC(\tau_i)$
where $\tau_i \in \IC_0$ corresponds to $\Omega_i$.

\item We then write $\nu=\sigma_{\overline \PC}$ a representation of
$\GL_{n/[E:L]}(k_E)$ and see it as a representation of $J/J^1$.

\item We repeat this construction for every $i=1,\cdots,r$.
\end{itemize}

\begin{nota} \label{nota-SZ}
For $\tau=\tau_\PC$ we write $\sigma(\tau):=\sigma(\SF_\PC)$.
\end{nota}

\begin{thm} (\cite{Shotton-n} 6.20) \label{thm-sigmamax} \\
Let $\PC' \in \IC$ with degree $n$ and let $(M',\pi')$ be any discrete pair in the
inertial equivalence class associated to $\PC'$. For any parabolic subgroup $Q'$ of $G$
with Levi $M'$, we have
$$\dim \Hom_K(\sigma(\SF_\PC),\ind_{Q'}^G \pi' ) =\prod_{\tau_0 \in \IC_0} 
m(\PC(\tau_0), \PC'(\tau_0)),$$
where for any two partition $\underline \lambda, \underline \mu$ of $n$, 
$m(\underline \lambda, \underline \mu)$ is the usual
Kostka number which counts the number of possible ways to fill the Young tableau
with lines of respective size $\lambda_1,\cdots$ with $\mu_1$ one, $\mu_2$ two and
so on, with the following properties: the sequence of labels in each columns is strictly
increasing, while it is increasing in each lines.
\end{thm}

Note that
\begin{itemize}
\item if $\underline \lambda <\underline \mu$ then 
$m(\underline \lambda, \underline \mu)=0$;

\item $m(\underline \lambda, \underline \lambda)=1$;

\item $m((n), \underline \mu)=1$.
\end{itemize}
In particular if $\PC$ is maximal that is if for every $\tau_0 \in \IC_0$ we have
$\PC(\tau_0)=(\sum \PC(\tau_0))$, then the multiplicity of $\sigma(\SF_\PC)$
in $\ind_{Q'}^G \pi'$ is $1$ if $\sum \PC'(\tau_0)=\sum \PC(\tau_0)$ for
every $\tau_0 \in \IC_0$, otherwise it is $0$.

Let $\SF=\Bigl \{ (E_i,\beta_i,V_i,\BF_i,\lambda_i); i=1,\cdots,r \Bigr \}$ be an
SZ-datum. Write $\lambda_i=\kappa_i \otimes \nu_i$ where the $\nu_i$ are
irreducible representations of $J_i/J^1_i$. Note that the modulo $l$
reduction $\overline \kappa_i$ of $\kappa_i$
remains irreducible. We then decompose the semisimplification of the
modulo $l$ reduction of $\nu_i$:
$$\overline \nu_i^{ss}=\bigoplus_{j \in S_i} \mu_{i,j} \nu_{i,j}$$
where $S_i$ is some finite indexing set, $\nu_{i,j}$ are distinct irreducible 
$\overline \Fm_l$-representations and $\mu_{i,j} \in \Nm$. For
$\underline j=(j_1,\cdots,j_r) \in S_1 \times \cdots \times S_r$, define an SZ-datum 
$\SF_{\underline j}$ over $\overline \Fm_l$ by
$$\SF_{\underline j}=\Bigl \{ (E_i,\beta_i,V_i,\BF_i,\overline \kappa_i \otimes \nu_{i,j_i}); i=1,\cdots,r \Bigr \}.$$

\begin{thm} (cf \cite{Shotton-n} 6.23) \\
The semisimplified modulo $l$ reduction of $\sigma(\SF)$ is
$$\bigoplus_{\underline j \in S_1 \times \cdots \times S_r} \mu_{\underline j} \sigma(\SF_{\underline j}),$$
where $\mu_{\underline j}:=\prod_{i=1}^r \mu_{i,j_i}$.
\end{thm}

\subsection{Kottwiz--Harris--Taylor Shimura varieties}
\label{subsec-KHT}

Let $F=F^+ E$ be a CM field where $E/\Qm$ is a quadratic imaginary extension and 
$F^+/\Qm$ is
totally real. 
We fix a real embedding $\tau:F^+ \hookrightarrow \Rm$. For a place $v$ of $F$,
we will denote by $F_v$ the completion of $F$ at $v$,
$\mathcal{O}_v$ its ring of integers with uniformizer $\varpi_v$ and 
residue field $\kappa(v)=\mathcal{O}_v/(\varpi_v)$ of cardinal $q_v$.

Let $B$ be a division algebra with center $F$, of dimension $d^2$ such that at every place $v$ of $F$,
either $B_v$ is split or a local division algebra and suppose $B$ provided with an involution of
second kind $*$ such that $*_{|F}$ is the complex conjugation. For any
$\beta \in B^{*=-1}$, denote by $\sharp_\beta$ the involution $v \mapsto v^{\sharp_\beta}=\beta v^*
\beta^{-1}$ and let $G/\Qm$ be the group of similitudes, denoted by 
$G_\tau$ in \cite{harris-taylor}, defined for every $\Qm$-algebra $R$ by 
$$
G(R)  \simeq   \{ (\lambda,g) \in R^\times \times (B^{op} \otimes_\Qm R)^\times  \hbox{ such that } 
gg^{\sharp_\beta}=\lambda \}
$$
with $B^{op}=B \otimes_{F,c} F$. 
If $x$ is a place of $\Qm$ split $x=yy^c$ in $E$ then 
\begin{equation} \label{eq-facteur-v}
G(\Qm_x) \simeq (B_y^{op})^\times \times \Qm_x^\times \simeq \Qm_x^\times \times
\prod_{v^+_i} (B_{v^+_i}^{op})^\times,
\end{equation}
where $x=\prod_i v^+_i$ in $F^+$ and we identify places of $F^+$ over $x$ with places of $F$ over $y$.

\begin{conv} 
For $x=yy^c$ a place of $\Qm$ split in $M$ and $v$ 
a place of $F$ over $y$, we shall make throughout the text the following abuse of notation: we denote 
$G(F_v)$ the factor $(B_{v|_{F^+}}^{op})^\times$ in the formula (\ref{eq-facteur-v})
so that 
$$G(\Am_\Qm^{\oo,v}):=G(\Am_\Qm^{\oo,p}) \times \Bigl ( 
\Qm_p^\times \times \prod_{v_i^+ \neq v|F^+} (B_{v_i^+}^{op})^\times \Bigr ).$$
\end{conv}

In \cite{harris-taylor}, the authors justify the existence of some $G$ like before such that 
\begin{itemize}
\item if $x$ is a place of $\Qm$ non split in $M$ then $G(\Qm_x)$ is quasi split;

\item the invariants of $G(\Rm)$ are $(1,d-1)$ for the embedding $\tau$ and $(0,d)$ for the others.
\end{itemize}

As in \cite[page 90]{harris-taylor}, a compact open subgroup $K$ of $G(\Am^\oo_{\Qm})$ is said to be
\emph{sufficiently small}
if there exists a place $x$ of $\Qm$ such that the projection from $K^x$ to $G(\Qm_x)$ does not contain any 
element of finite order except identity.

\begin{nota}
Denote by $\mathcal{K}$ the set of sufficiently small compact open subgroups of $G(\Am^\oo)$.
For $K \in \mathcal{K}$, write $\Sh_{K,\eta} \longrightarrow \Spec F$ for the associated
Shimura variety of Kottwitz-Harris-Taylor type.
\end{nota}

\begin{defin} \label{defi-spl}
Denote by $\spl$ the set of  places $w$ of $F$ such that $p_w:=w_{|\Qm} \neq l$ is split in $E$ and
$B_w^\times \simeq \GL_d(F_w)$.  For each $K \in \mathcal{K}$, we write
$\spl(K)$ for the subset of $\spl$ of places such that $K_v$ is the 
standard maximal compact of $\GL_d(F_v)$.
\end{defin}

In the sequel, we fix a place $v$ of $F$ in $\spl$. 
The scheme $\Sh_{K,\eta}$ has a projective model $\Sh_{K,v}$ over 
$\Spec \mathcal{O}_v$
with special geometric fiber $\Sh_{K,\bar s_v}$. 
We have a projective system $(\Sh_{K,\bar s_v})_{K\in \mathcal{K}}$ 
which is naturally equipped with an action of $G(\Am^\oo_{\Qm}) \times \Zm$ such that any
$w_v\in W_{F_v}$ acts by $-\deg (w_v) \in \Zm$,
where $\deg=\val \circ \art_{F_v}^{-1}$ and $\art_{F_v}: F_v^\times\stackrel{\sim}{\ra} W_{F_v}^{ab}$.

\begin{nota} 
For $K \in \mathcal{K}$, the Newton stratification of the geometric special fiber $\Sh_{K,\bar s_v}$ is denoted by
$$\Sh_{K,\bar s_v}=:\Sh^{\geq 1}_{K,\bar s_v} \supset \Sh^{\geq 2}_{K,\bar s_v} \supset \cdots \supset 
\Sh^{\geq d}_{K,\bar s_v}$$
where $\Sh^{=h}_{K,\bar s_v}:=\Sh^{\geq h}_{K,\bar s_v} - \Sh^{\geq h+1}_{K,\bar s_v}$ is an affine 
scheme, which is smooth and pure of dimension $d-h$.
It is built up by the geometric 
points such that the connected part of the associated Barsotti--Tate group has rank $h$
For each $1 \leq h <d$, write
$$i_{h}:\Sh^{\geq h}_{K,\bar s_v} \hookrightarrow \Sh^{\geq 1}_{K,\bar s_v}, \quad
j^{\geq h}: \Sh^{=h}_{K,\bar s_v} \hookrightarrow \Sh^{\geq h}_{K,\bar s_v},$$
and $j^{=h}=i_h \circ j^{\geq h}$.
\end{nota}

For $n \geq 1$, with our previous abuse of notation, consider $K^v(n):=K^vK_v(n)$ where
$$K_v(n):=\ker(\GL_d(\OC_v) \twoheadrightarrow \GL_d(\OC_v/\cM_v^n)).$$
Recall that $\Sh_{I^v(n),\bar s_v}^{=h}$ is geometrically induced
under the action of the parabolic subgroup $P_{h,d}(\OC_v/\cM_v^n)$, 
defined as the stabilizer of the first $h$ vectors of the canonical basis of $F_v^d$. 
Concretely this means there exists a closed subscheme 
$\Sh_{K^v(n),\bar s_v,1}^{=h}$ stabilized by the Hecke 
action of $P_{h,d}(F_v)$ and such that
\begin{equation} \label{eq-induction}
\Sh_{K^v(n),\bar s_v}^{=h} = \Sh_{K^v(n),\bar s_v,1}^{=h} 
\times_{P_{h,d}(\OC_v/\cM_v^n)} \GL_d(\OC_v/\cM_v^n),
\end{equation}
meaning that $\Sh_{K^v(n),\bar s_v}^{=h} $ is the disjoint union of copies of
$\Sh_{K^v(n),\bar s_v,1}^{=h}$ indexed by 
$\GL_d(\OC_v/\cM_v^n)/P_{h,d}(\OC_v/\cM_v^n)$ and  exchanged by the action of
$\GL_d(\OC_v/\cM_v^n)$. We will denote by $\Sh^{\geq h}_{K^v(n),\bar s_v,1}$
the closure of $\Sh^{=h}_{K^v(n),\bar s_v,1}$ inside $\Sh_{K^v(n),\bar s_v}$.
 
 \begin{nota} 
Let $1 \leq g \leq d$ and $\pi_v$ be an irreducible cuspidal representation of
$\GL_g(F_v)$. For $1 \leq t \leq s:=\lfloor d/g \rfloor$, let $\Pi_t$ any representation of 
 $\GL_{tg}(F_v)$. We denote by
$$\widetilde{HT}_1(\pi_v,\Pi_t):=\cL(\pi_v,t)_{1} 
\otimes \Pi_t^{K_v(n)} \otimes \Xi^{\frac{t-s}{2}}$$ 
the Harris-Taylor local system on the Newton stratum 
$\Sh^{=tg}_{K^v(n),\bar s_v,1}$ where 
\begin{itemize}
\item $\cL(\pi_v,t)_{1}$ is the $\overline \Zm_l$-local system 
given by Igusa varieties of  \cite{harris-taylor} and associated to the 
representation $\pi_v[t]_D$ of the division algebra  $D_{v,tg}/F_v$ with invariant 
$1/tg$, corresponding through Jacquet-Langlands correspondance
to $\st_t(\pi_v^\vee)$: cf. \cite{boyer-invent2} \S 1.4 for more details;

\item $\Xi:\frac{1}{2} \Zm \longrightarrow \overline \Zm_l^\times$ is 
defined by $\Xi(\frac{1}{2})=q^{1/2}$.
\end{itemize}
We also introduce the induced version
$$\widetilde{HT}(\pi_v,\Pi_t):=\Bigl ( \cL(\pi_v,t)_{1} 
\otimes \Pi_t^{K_v(n)} \otimes \Xi^{\frac{t-s}{2}} \Bigr) 
\times_{P_{tg,d}(\OC_v/\cM_v^n)} \GL_d(\OC_v/\cM_v^n),$$
where the unipotent radical of $P_{tg,d}(\OC_v/\cM_v^n)$ 
acts trivially and the action of
$$(g^{\oo,v},\left ( \begin{array}{cc} g_v^c & * \\ 0 & g_v^{et} \end{array} \right ),\sigma_v) 
\in G(\Am^{\oo,v}) \times P_{tg,d}(\OC_v/\cM_v^n) \times W_v$$ 
is given
\begin{itemize}
\item by the action of $g_v^c$ on $\Pi_t^{K_v(n)}$ and 
$\deg(\sigma_v) \in \Zm$ on $ \Xi^{\frac{t-s}{2}}$, and

\item the action of $(g^{\oo,v},g_v^{et},\val(\det g_v^c)-\deg \sigma_v)
\in G(\Am^{\oo,v}) \times \GL_{d-tg}(\OC_v/\cM_v^n) \times \Zm$ on 
$\cL_{\overline \Qm_l} (\pi_v)_{1} \otimes \Xi^{\frac{t-s}{2}}$.
\end{itemize}
We also introduce
$$HT(\pi_v,\Pi_t)_{1}:=\widetilde{HT}(\pi_v,\Pi_t)_{1}[d-tg],$$
and the perverse sheaf
$$P(t,\pi_v)_{1}:=\lexp p j^{=tg}_{1,!*} HT(\pi_v,\St_t(\pi_v))_{1} 
\otimes L_g(\pi_v)^\vee,$$
and their induced version, $HT(\pi_v,\Pi_t)$ and $P(t,\pi_v)$,
where $L_g(\pi_v)^\vee$ is the contragredient of the representation of
dimension $g$ of $\gal(\overline F_v/F_v)$
associated to $\pi_v$ by the Langlands correspondence $L_g$.
\end{nota}

\noindent \textbf{Important property}:
over $\overline \Zm_l$, there are at least 
two notions of intermediate extension associated to 
the two classical $t$-structures $p$ and $p+$.
By proposition 2.4.1 of \cite{boyer-duke}, in the limit case where
all $\overline \Fm_l$-cuspidal representations are supercuspidal, as recalled
after proposition \ref{prop-red-modl},
all the $p$ and $p+$
intermediate extensions of Harris-Taylor local systems coincide.
The arguments in loc. cit. are rather difficult but if one accepts to restrict to the case
where the irreducible constituants of $\overline \rho_{\mathfrak m}$ are all
characters, then the proof of this fact is easy.
Indeed as $\Sh^{\geq h}_{K^v(n),\bar s_v,1}$ is smooth over 
$\Spec  \overline \Fm_p$, then $HT(\chi_v,\Pi_h)_{1}$ is perverse for the two 
$t$-structures with
$$i^{h \leq +1,*}_{1} HT(\chi_v,\Pi_h)_{1} \in \lexp p \DC^{< 0} \hbox{ and }
i^{h \leq +1,!}_{1} HT(\chi_v,\Pi_h)_{1} \in \lexp {p+} \DC^{\geq 1}.$$

Let now denote by 
$$\Psi_{K,v}:=R\Psi_{\eta_v}(\overline \Zm_l[d-1])(\frac{d-1}{2})$$
the nearby cycles autodual free perverse sheaf on $\Sh_{K,\bar s_v}$.
Recall, cf. \cite{boyer-duke} proposition 3.1.3, that 
\begin{equation} \label{eq-psi-split}
\Psi_{K,v} \simeq \bigoplus_{1 \leq g \leq d} \bigoplus_{\varrho \in \scusp(g)} 
\Psi_{K,\varrho},
\end{equation}
where 
\begin{itemize}
\item $\scusp(g)$ is the set of equivalence classes
of irreducible supercuspidal $\overline \Fm_l$-representations of $\GL_g(F_v)$.

\item The irreducible sub-quotients
of $\Psi_{K,\varrho} \otimes_{\overline \Zm_l} \overline \Qm_l$ 
are the Harris-Taylor perverse sheaves of
$\Psi_{K,\overline \Qm_l}$ associated to irreducible cuspidal representations
$\pi_v$ with modulo $l$ reduction having supercuspidal support a Zelevinsky
segment associated to $\varrho$.
\end{itemize}

In the limit case when $q_v \equiv 1 \mod l$ and $l>d$, recall that 
we do not have to bother 
about cuspidal $\overline \Fm_l$-representation which are not supercuspidal. 
In particular in the previous formula we can 
\begin{itemize}
\item replace $\scusp(g)$ by the set $\cusp(g)$ of
equivalence classes of cuspidal representations,

\item and the Harris-Taylor perverse sheaves of 
$\Psi_{K,\varrho} \otimes_{\overline \Zm_l} \overline \Qm_l$
are those associated to $\pi_v$ such that its modulo $l$ reduction is isomorphic
to $\varrho$.
\end{itemize}

%% file: filtrations.tex
\section{Nearby cycles and filtrations}

\subsection{Filtrations of stratification of $\Psi_\varrho$}
\label{para-exchange}

We now fix an irreducible $\overline \Fm_l$-cuspidal representation $\varrho$
of $\GL_g(F_v)$ for some $1 \leq g \leq d$. We also introduce $s=\lfloor d/g \rfloor$.

Using the Newton stratification and following the constructions of \cite{boyer-torsion},
we can define a $\overline \Zm_l$-filtration 
$$\Fil_!^0(\Psi_{K,\varrho}) \hookrightarrow \cdots \hookrightarrow 
\Fil_!^s(\Psi_{K, \varrho}) =\Psi_{K,\varrho}$$
where $\Fil^t_!(\Psi_{K,\varrho})$ is the saturated image of
$j^{=tg}_! j^{=tg,*} \Psi_{K,\varrho} \longrightarrow \Psi_{K,\varrho}$. 
We also denote by $\coFil_!^t(\Psi_\varrho):=\Psi_\varrho/\Fil^t_!(\Psi_\varrho)$.
Dually we can define a  cofiltration
$$\Psi_{K,\varrho}=\coFil_*^s (\Psi_{K,\varrho}) \twoheadrightarrow \cdots
\twoheadrightarrow \coFil_*^1 (\Psi_{K,\varrho})$$
where $\coFil_*^t (\Psi_{K,\varrho})$ is the saturated image of
$\Psi_{K,\varrho} \longrightarrow j^{=tg}_* j^{=tg,*} \Psi_{K,\varrho}$: cf. figure
\ref{fig-psi-filtration} for an illustration. We denote by
$\Fil^t_*(\Psi_\varrho):=\ker (\Psi_\varrho \twoheadrightarrow \coFil^t_*(\Psi_\varrho))$.

\begin{figure}[htbp]
\begin{center}
\input{fig-psi.tex}
\end{center}
\caption{\label{fig-psi-filtration} Filtrations of stratification of $\Psi_{K,\varrho}$}
\end{figure}

Over $\overline \Qm_l$, the filtration $\Fil^\bullet_!(\Psi_{K,\varrho})$
coincides with the iterated kernel of $N_v$,
i.e. $\Fil^t_!(\Psi_\varrho) \otimes_{\overline \Zm_l} \overline \Fm_l \simeq
\ker (N_v^t \otimes_{\overline \Zm_l} \overline \Fm_l)$.
Dually the cofiltration $\coFil^\bullet_!(\Psi_{K,\varrho})$ coincides 
with the iterated image of $N_v$,
i.e. the kernel of $\Psi_{K,\varrho} \twoheadrightarrow \coFil_*^t (\Psi_{K,\varrho})$
is the image of $N_v^t$.  Note that by Grothendieck-Verdier duality, 
we have $D(\Fil^t_!(\Psi_{K,\varrho})) \simeq \coFil^t_*(\Psi_{K,\varrho})$.

The graded parts $\gr^t_!(\Psi_{K,\varrho})$ are, by construction, free and 
admit a strict\footnote{meaning the graded parts are free} filtration, cf. \cite{boyer-torsion} corollary 3.4.5
$$\Fil^{s-1}_*(\gr^k_!(\Psi_{K,\varrho})) \hookrightarrow \cdots \hookrightarrow
\Fil^{k-1}_*(\gr^k_!(\Psi_{K,\varrho}))=\gr^k_!(\Psi_{K,\varrho})$$
with 
$$\gr^{i-1}_*(\gr^k_!(\Psi_{K,\varrho})) \otimes_{\overline \Zm_l} \overline \Qm_l
\simeq  \bigoplus_{\pi_v \in \cusp(\varrho)}P(i,\pi_v)(\frac{i+1-2k}{2}),$$
where $\cusp(\varrho)$ is the set of equivalence classes of irreducible cuspidal
representations with modulo $l$ reduction isomorphic to $\varrho$. 

Dually, $\cogr^k_*(\Psi_{K,\varrho})$ has a cofiltration
$$\cogr^k_*(\Psi_{K,\varrho})=\coFil^{k-1}_! (\cogr^k_*(\Psi_{K,\varrho})) 
\twoheadrightarrow
\cdots \twoheadrightarrow \coFil^{s-1}_! (\cogr^k_*(\Psi_{K,\varrho})),$$
with 
$$\cogr^{i-1}_! (\cogr^k_*(\Psi_{K,\varrho}))\otimes_{\overline \Zm_l} \overline \Qm_l
\simeq  \bigoplus_{\pi_v \in \cusp(\varrho)}P(i,\pi_v)(\frac{2k-i-1}{2}).$$

Concerning the $\overline \Zm_l$-structures, cf. the third global
result of the introduction of \cite{boyer-duke} , for every 
$1 \leq k \leq i \leq s$, we have strict epimorphisms\footnote{strict means that the cokernel is torsion free}
$$j^{=ig}_! j^{=ig,*} \Fil^{i-1}_*(\gr^k_!(\Psi_{K,\varrho})) \twoheadrightarrow
\Fil^{i-1}_*(\gr^k_!(\Psi_{K,\varrho}))$$
as well as strict monomorphisms
$$\coFil^{i-1}_!(\cogr^k_*(\Psi_{K,\varrho})) \hookrightarrow j^{=ig}_* j^{=ig,*}
\coFil^{i-1} (\cogr^k_*(\Psi_{K,\varrho})).$$

\noindent \textit{Exchange basic step}:   
to go from one filtration to another, one can repeat the following process to exchange
the order of appearance of two consecutive subquotient:
$$\xymatrix{
& P_1' \ar@{^{(}->}[d] \ar@{^{(}->}[dr] \\
P_2  \ar@{^{(}->}[dr] \ar@{^{(}->}[r] & X \ar@{->>}[d] \ar@{->>}[r] & P_1 \ar@{->>}[dr]  \\
& P_2' \ar@{->>}[dr] & & T \\
& & T, \ar@{=}[ur]
}$$
where 
\begin{itemize}
\item $P_1$ and $P_2$ are two consecutive subquotient in a given filtration
and $X$ is the subquotient gathering them as a subquotient of this filtration.

\item Over $\overline \Qm_l$, the extension $X \otimes_{\overline \Zm_l} \overline \Qm_l$
is split, so that on can write $X$ as an extension of $P'_2$ by $P'_1$ with
$P'_1 \hookrightarrow P_1$ and $P_2 \hookrightarrow P'_2$ have the same cokernel
$T$, a perverse sheaf of torsion.
\end{itemize}

\rem In the particular case when $P_1$ and $P_2$ are intermediate extensions 
of local systems 
living on respective strata of index $h_1$ and $h_2$ with $h_1 \neq h_2$, 
such that the two associated intermediate extensions
for the $p$ and $p+$ $t$-structure are isomorphic, then $T$ is necessary zero and
$X$ is then split over $\overline \Zm_l$. Indeed  if $T$ was not zero, then 
seen as a quotient of $P_1$
(resp. $P'_2$) it has to be supported on the $\sh^{\geq h_1}_{K,\bar s_v}$
(resp.  $\sh^{\geq h_1}_{K,\bar s_v}$) with 
$j^{=h_1,*} T \neq 0$ (resp. $j^{=h_2,*} T \neq 0$): the two conditions are 
then incompatible.

\subsection{The canonical filtration of 
$H^0(\sh_{K,\bar s_v},\Psi_{K,\overline \Zm_l})_{\mathfrak m}$ is strict}
\label{para-filt-strict}

Let us first recall some known facts about irreducible algebraic representations of $G$,
see for example \cite{harris-taylor} p.97. Let $\sigma_0:E \hookrightarrow
\overline{\Qm}_l$ be a fixed embedding and et write $\Phi$ the set of embeddings 
$\sigma:F \hookrightarrow \overline \Qm_l$ whose restriction to $E$ equals $\sigma_0$.
There exists then an explicit bijection between irreducible algebraic representations $\xi$ of $G$ 
over $\overline \Qm_l$ and $(d+1)$-uple
$\bigl ( a_0, (\overrightarrow{a_\sigma})_{\sigma \in \Phi} \bigr )$
where $a_0 \in \Zm$ and for all $\sigma \in \Phi$, we have $\overrightarrow{a_\sigma}=
(a_{\sigma,1} \leq \cdots \leq a_{\sigma,d} )$.

For $L \subset \overline \Qm_l$ a finite extension of $\Qm_l$ such that the representation
$\iota^{-1} \circ \xi$ of highest weight
$\bigl ( a_0, (\overrightarrow{a_\sigma})_{\sigma \in \Phi} \bigr )$,
is defined over $L$, write $W_{\xi,L}$ the space of this representation and $W_{\xi,\OC_L}$
a stable lattice under the action of the maximal open compact subgroup $G(\Zm_l)$, 
where $\OC_L$ is the ring of integers of $L$ with uniformizer $\varpi_L$.

\rem if $\xi$ is supposed to be $l$-small, in the sense that for all $\sigma \in \Phi$ and all
$1 \leq i < j \leq n$ we have $0 \leq a_{\tau,j}-a_{\tau,i} < l$, then such a stable lattice is unique
up to a homothety.

\begin{nota} \label{nota-Vxi}
We will denote by $\LC_{\xi,\OC_L/\varpi_L^n}$ the local system on $\sh_{I,v}$ as well as
$$\LC_{\xi,\OC_L}=\lim_{\atop{\longleftarrow}{n}} \LC_{\xi,\OC_L/\varpi_L^n} \quad \hbox{and} 
\quad \LC_{\xi,L}=\LC_{\xi,\OC_L} \otimes_{\OC_L} L.$$
For $\overline \Zm_l$ and $\overline \Qm_l$ version, we will write respectively
$\LC_{\xi,\overline \Zm_l}$ and $\LC_{\xi,\overline \Qm_l}$. 
\end{nota}

\rem We'll add the symbol $\xi$ on a sheaf to indicate its torsion by $\LC_{\xi,\overline \Zm_l}$: 
for example $HT_\xi(\pi_v,\Pi_t):=HT(\pi_v,\Pi_t) \otimes \LC_{\xi,\overline \Zm_l}$ or
$\Psi_{K,\varrho,\xi}:=\Psi_{K,\varrho} \otimes \LC_{\xi,\overline \Zm_l}$.

We have spectral sequences
\begin{equation} \label{eq-ss-grk}
E_{1}^{p,q}=H^{p+q}(\Sh_{K,\bar s_v},\gr^{-p}_*
(\gr^k_!(\Psi_{K,\varrho,\xi})))
\Rightarrow H^{p+q}(\Sh_{K,\bar s_v} \gr^k_!(\Psi_{K,\varrho,\xi})),
\end{equation}
and
\begin{equation} \label{eq-ss-psi}
E_{1}^{p,q}=H^{p+q}(\Sh_{K,\bar s_v}, \gr^{-p}_!(\Psi_{K,\varrho,\xi}))
\Rightarrow H^{p+q}(\Sh_{K,\bar s_v},\Psi_{K,\varrho,\xi}).
\end{equation}

\begin{defin} \label{defi-Tm}
For a finite set $S$ of places of $\Qm$ containing the places where $G$ is ramified, 
denote by $\Tm^S_{abs}:=\prod'_{x \not \in S} \Tm_{x,abs}$ the 
abstract unramified  Hecke algebra
where
$\Tm_{x,abs} \simeq \Zm_l[X^{un}(T_x)]^{W_x}$ for $T_x$ a split torus,
$W_x$ the spherical Weyl group and $X^{un}(T_x)$ the set of $\overline \Zm_l$-unramified 
characters of $T_x$. 
\end{defin}

\noindent \textit{Example}.
For $w \in \spl$, we have
$$\Tm_{w|_{\Qm},abs}=\Zm_l \bigl [T_{w',i}:~ i=1,\cdots,d, ~ w'| (w|_{\Qm}) \bigr ]$$
where $T_{w',i}$ is the characteristic function of
$$\GL_d(\OC_{w'}) \diagonal(\overbrace{\varpi_{w'},\cdots,\varpi_{w'}}^{i}, \overbrace{1,\cdots,1}^{d-i} ) 
\GL_d(\OC_{w'}) \subset  \GL_d(F_{w'}).$$

Recall that $\Tm^S_{abs}$ acts through correspondances on each of the
$H^i(\Sh_{K,\bar \eta},\overline \Zm_l)$ where $K \in \mathcal{K}$ 
is maximal at each places outside $S$.

\begin{nota} For $K$ unramified outside $S$,
we denote by $\Tm_\xi(K)$ the image of $\Tm^S_{abs}$ inside $\End_{\overline \Zm_l}(H^{d-1}(\Sh_{K,\bar \eta},\LC_{\xi,\overline \Zm_l}))$.
\end{nota}
We also denote by
$$H^{d-1}(\sh_{K^v(\oo),\bar \eta},\LC_{\xi,\overline \Zm_l}):= 
\lim_{\genfrac{}{}{0pt}{}{\longrightarrow}{K_v}} H^{d-1}(\sh_{K^vK_v,\bar \eta},\LC_{\xi,\overline \Zm_l}),$$
where $K_v$ describe the set of open compact subgroup of $\GL_d(\OC_v)$.
We also use similar notation for others cohomology groups.

\begin{thm} \label{thm-recall} 
Let $\mathfrak m$ be a maximal ideal of $\Tm_\xi(K^v(\oo))$ such that
$\overline \rho_{\mathfrak m}$ is irreducible, cf. \S \ref{sec-hecke}.\footnote{Recall also that we suppose $q_v \equiv 1 \mod l$ and $l >d$.} Then
\begin{itemize}
\item[(i)] 
$H^{i}(\Sh_{K^v(\oo),\bar \eta},\LC_{\xi,\overline \Zm_l})_{\mathfrak m}$ is zero if $i \neq d-1$ and
otherwise torsion free.

\item[(ii)] 
Moreover the spectral 
sequences (\ref{eq-ss-grk}) and (\ref{eq-ss-psi}), localized at $\mathfrak m$, 
degenerate at $E_1$ and the $E_{1,\mathfrak m}^{p,q}$ are zero for $p+q \neq 0$
and otherwise torsion free. 
\end{itemize}
\end{thm}

\begin{proof}
(i) It is the main theorem of \cite{boyer-jep2}.

(ii) We follow closely the arguments of \cite{boyer-jep2} dealing with all
irreducible cuspidal representations instead of only characters in loc. cit.
using crucially that in the limit case, the $p$ and $p+$ intermediate extensions
coincide exactly as it was the case for characters in loc. cit.

From (\ref{eq-psi-split}) we are led to study the initial terms of the spectral
sequence given by the filtration of $\Psi_{K^v(\oo),\varrho,\xi}$ for $\varrho$ 
a irreducible $\overline \Fm_l$-supercuspidal representation
associated through local Langlands correspondance to an irreducible constituant 
of $\overline \rho_{\mathfrak m,v}$. Recall also, as
we are in the limit case, that 
\begin{itemize}
\item as there do not exist irreducible $\overline \Qm_l$-cuspidal representation
of $\GL_g(F_v)$ for $g \leq d$ with modulo $l$ reduction being not supercuspidal,
the irreducible
constituants of $\Psi_{K,\varrho,\xi} \otimes_{\overline \Zm_l} \overline \Qm_l$
are the Harris-Taylor perverse sheaves $P(t,\pi_v)(\frac{t-1-2k}{2})$ where
the modulo $l$ reduction of $\pi_v$ is isomorphic to $\varrho$ and
$0 \leq k < t$. 

\item Over $\overline \Zm_l$, we do not have to worry about
the difference between $p$ and $p+$ intermediate extensions.
\end{itemize}
From \cite{boyer-duke} \S 2.3, consider the following equivariant resolution
\begin{multline} \label{eq-resolution0}
0 \rightarrow j_!^{=sg} HT(\pi_v,\Pi_t \{ \frac{t-s}{2} \}  \times 
\speh_{s-t}(\pi_v\{ t/2 \} ))
 \otimes \Xi^{\frac{s-t}{2}} \longrightarrow \cdots  \\
\longrightarrow j_!^{=(t+1)g} HT(\pi_v,\Pi_t\{ -1/2 \} \times \pi_v \{ t/2 \} ) 
\otimes \Xi^{\frac{1}{2}} \longrightarrow \\ j_!^{=tg} HT(\pi_v,\Pi_t) 
\longrightarrow  \lexp p j_{!*}^{=tg} HT(\pi_v,\Pi_t) \rightarrow 0,
\end{multline}
where $\Pi_t$ is any representation of $\GL_{tg}(F_v)$, also called the infinitesimal 
part of the perverse sheaf $\lexp p j_{!*}^{=tg} HT(\pi_v,\Pi_h)$.\footnote{In
$P(t,\pi_v)$ the infinitesimal part $\Pi_t$ is $\st_t(\pi_v)$.}

By adjunction property, for $1 \leq \delta \leq s-t$, the map
\begin{multline} \label{eq-map1}
j_!^{=(t+\delta)g} HT(\pi_v,\Pi_t\{ \frac{-\delta}{2} \}  \times \speh_{\delta}
(\pi_v \{ t/2 \} )) \otimes \Xi^{\delta/2} \\
\longrightarrow j_!^{=(t+\delta-1)g} HT(\pi_v,\Pi_t \{ \frac{1-\delta}{2} \}  
\times  \speh_{\delta-1}(\pi_v\{ h/2 \} )) \otimes \Xi^{\frac{\delta-1}{2}}
\end{multline}
is given by 
\begin{multline}\label{eq-map2}
HT(\pi_v,\Pi_t \{ \frac{-\delta}{2} \}  \times \speh_{\delta}(\pi_v\{ t/2 \} )) 
\otimes \Xi^{\delta/2} \longrightarrow \\ j^{=(t+\delta)g,*} (
\lexp p i^{(t+\delta)g,!}  ( 
j_!^{=(t+\delta-1)g} HT(\pi_v,\Pi_t \{ \frac{1-\delta}{2} \}  
\times \speh_{\delta-1}(\pi_v\{ t/2 \} )) \otimes \Xi^{\frac{\delta-1}{2}}))
\end{multline}
To compute this last term we use the resolution (\ref{eq-resolution0}) for $t+\delta-1$.
Precisely denote by 
$\HC:=HT(\pi_v,\st_t(\pi_v \{ \frac{1-\delta}{2} \} ) \times 
\speh_{\delta-1}(\pi_v\{ t/2 \} )) \otimes \Xi^{\frac{\delta-1}{2}}$,
and write the previous resolution for $t+\delta-1$ as follows
$$0 \rightarrow K \longrightarrow  j_!^{=(t+\delta)g} 
\HC' \longrightarrow Q \rightarrow 0,$$
$$0 \rightarrow Q \longrightarrow  j_!^{=(t+\delta-1)g} \HC 
\longrightarrow  \lexp p j_{!*}^{=(t+\delta-1)g} \HC \rightarrow 0,$$
with 
$$\HC':=HT \Bigl ( \pi_v,  \Pi_t \{ \frac{1-\delta}{2} \}  
\times \bigl ( \speh_{\delta-1}(\pi_v \{ -1/2 \})
 \times \pi_v \{\frac{\delta-1}{2} \}  \bigr ) \{ t/2 \}  \Bigr ) \otimes \Xi^{\delta/2}.$$ 
As the support of $K$ is contained in $\sh^{\geq (t+\delta+1)g}_{I,\bar s_v}$ then
$\lexp p i^{(t+\delta)g,!} K=K$ and
$j^{=(t+\delta)g,*} (\lexp p i^{(t+\delta)g,!} K)$ is zero. Moreover
$\lexp p i^{(t+\delta)g,!} ( \lexp p j_{!*}^{=(t+\delta-1)g} \HC)$ is zero by
construction of the intermediate extension. We then deduce that
\begin{multline} \label{eq-map3}
 j^{=(t+\delta)g,*} (\lexp p i^{(t+\delta)g,!} ( j_!^{=(t+\delta-1)g} HT(\pi_v,\Pi_t 
 \{ \frac{1-\delta}{2} \}  \times 
\speh_{\delta-1}(\pi_v\{ t/2 \} )) \otimes \Xi^{\frac{\delta-1}{2}})) \\ \simeq 
HT \Bigl ( \pi_v,  \Pi_t \{ \frac{1-\delta}{2} \}  \\ \times \bigl 
( \speh_{\delta-1}(\pi_v \{ -1/2 \})
 \times \pi_v \{\frac{\delta-1}{2} \}  \bigr ) \{ t/2 \}  \Bigr ) \otimes \Xi^{\delta/2} 
\end{multline}
In particular, up to homothety, the map (\ref{eq-map2}), and so (\ref{eq-map1}), is unique.
Finally as the maps of (\ref{eq-resolution0}) are strict, the given maps (\ref{eq-map1}) 
are uniquely determined, that is, if we forget the infinitesimal parts, 
these maps are independent of the 
chosen $t$ in (\ref{eq-resolution0}), i.e. only depends on $t+\delta$.

For every $1 \leq t \leq s$, let denote by $i(t)$ the smallest index $i$ such that 
$H^i(\sh_{K^v(\oo),\bar s_v},\lexp p j^{=tg}_{!*} HT_\xi(\pi_v,\Pi_t))_{\mathfrak m}$
has non trivial torsion: if it does not exist then we set $i(t)=+\oo$ and note that
it does not depend on the choice of the infinitesimal part $\Pi_t$. By duality, as 
$\lexp p j_{!*}=\lexp {p+} j_{!*}$ for ours Harris-Taylor local systems, 
note that when $i(t)$ is finite then $i(t) \leq 0$. 
Suppose by absurdity there exists
$t$ with $i(t)$ finite and denote $t_0$ the biggest such $t$.

\begin{lem} \label{lem-ih}
For $1 \leq t \leq t_0$ then $i(t)=t-t_0$.
\end{lem}


\begin{proof}
a) We first prove that for every $t_0 < t \leq s$, the cohomology 
groups of  $j^{=tg}_! HT(\pi_v,\Pi_t)$ are torsion free. Consider the
following strict filtration in the category of free perverse sheaves
\begin{multline} \label{eq-fil-j}
(0)=\Fil^{-1-s}(\pi_v ,h) \harrow \Fil^{-s}(\pi_v ,h) \harrow \cdots \\
\harrow \Fil^{-t}(\pi_v ,t)=j^{=tg}_{!} HT(\pi_v ,\Pi_t)
\end{multline}
where the symbol $\harrow$ means a strict\footnote{i.e. the cokernel is free} monomorphism, 
with graded parts 
$$\gr^{-k}(\pi_v,t) \simeq \lexp p j^{=kg}_{!*} 
HT(\pi_v,\Pi_t \{\frac{t-k}{2} \} \otimes \st_{k-t}(\pi_v\{t/2 \} ))(\frac{t-k}{2}).$$ 
Over $\overline \Qm_l$, the result is proved in
\cite{boyer-invent2} \S 4.3. Over $\overline \Zm_l$, the result follows from the
general constructions of \cite{boyer-torsion} and the fact that the $p$ and $p+$
intermediate extensions are isomorphic for Harris-Taylor perverse sheaves
associated to characters. The associated
spectral sequence localized at $\mathfrak m$, is then concentrated in middle degree and torsion free which gives the claim.

b) Before watching the cases $t \leq t_0$, note that
the spectral sequence associated to (\ref{eq-resolution0}) for $t=t_0+1$, 
has all its $E_1$ terms torsion free
and degenerates at its $E_2$ terms. As by hypothesis the aims of this spectral sequence is free
and equals to only one $E_2$ terms, we deduce that all the maps
\begin{multline} \label{eq-map1-coho}
H^0 \bigl (\sh_{K^v(\oo),\bar s_v},j_!^{=(t+\delta)g} HT_\xi(\pi_v,\st_t(\pi_v \{ \frac{-\delta}{2} \} ) \times \speh_{\delta}
(\pi_v\{ t/2 \} )) \otimes \Xi^{\delta/2} \bigr )_{\mathfrak m} \\
\longrightarrow \\ H^0 \bigl (\sh_{K^v(\oo),\bar s_v},
j_!^{=(t+\delta-1)g} HT_\xi(\pi_v,\st_t(\chi_v \{ \frac{1-\delta}{2} \} ) \\ \times 
\speh_{\delta-1}(\chi_v\{ t/2 \} )) \otimes \Xi^{\frac{\delta-1}{2}} \bigr )_{\mathfrak m}
\end{multline}
are saturated, i.e. their cokernel are free $\overline \Zm_l$-modules. Then from the previous fact stressed after (\ref{eq-map3}), this property
remains true when we consider the associated spectral sequence for 
$1 \leq t' \leq t_0$.

c) Consider now $t=t_0$ and the spectral sequence associated to 
(\ref{eq-resolution0}) where
\begin{multline} \label{eq-E2pq}
E_2^{p,q}=H^{p+2q}(\sh_{K^v(\oo),\bar s_v}, j_!^{=(t+q)g} \\ 
HT_\xi(\pi_v,\st_t(\pi_v (-q/2)) \times \speh_q(\pi_v \{ t/2 \} )) 
\otimes \Xi^{\frac{q}{2}})_{\mathfrak m}
\end{multline}
By definition of $t_0$, we know that  some of
the $E_\oo^{p,-p}$ should have a non trivial torsion subspace.
We saw that 
\begin{itemize}
\item the contributions from the deeper strata are torsion free and

\item $H^i(\sh_{K^v(\oo),\bar s_v},j^{=t_0g}_! HT_\xi(\pi_v,\Pi_{t_0}))_{\mathfrak m}$
are zero for $i<0$ and is torsion free for $i=0$, whatever is $\Pi_{t_0}$.

\item Then there should exist
a non strict map $d_1^{p,q}$. But, we have just seen that it
can not be maps between deeper strata.

\item Finally, using the previous points, the only possibility is that the 
cokernel of
\begin{multline} \label{eq-map1-coho2}
H^0 \bigl (\sh_{K^v(\oo),\bar s_v},j_!^{=(t_0+1)g} HT_\xi(\pi_v,\st_{t_0}(\pi_v 
\{ \frac{-1}{2} \} ) 
\times \pi_v \{ t_0/2 \} )) \otimes \Xi^{1/2} \bigr )_{\mathfrak m} \\
\longrightarrow \\ H^0 \bigl (\sh_{K^v(\oo),\bar s_v},
j_!^{=t_0g} HT_\xi(\pi_v,\st_{t_0}(\pi_v)) \bigr )_{\mathfrak m}
\end{multline}
has a non trivial torsion subspace. 
\end{itemize}
In particular we have $i(t_0)=0$.

d) Finally using the fact 2.18 and the previous points, for any $1 \leq t \leq t_0$, 
in the spectral sequence (\ref{eq-E2pq})
\begin{itemize}
\item by point a), $E_2^{p,q}$ is torsion free for $q \geq t_0-t+1$ and so it
is zero if $p+2q \neq 0$;

\item by affiness of the open strata, cf. \cite{boyer-imj} theorem 1.8,
$E_2^{p,q}$ is zero for $p+2q<0$ and torsion free for $p+2q=0$;

\item by point b), the maps $d_2^{p,q}$ are saturated for $q \geq t_0-t+2$;

\item by point c), $d_2^{-2(t_0-t+1),t_0-t+1}$ has a cokernel with a non trivial 
torsion subspace.

\item Moreover, over $\overline \Qm_l$, the spectral sequence degenerates
at $E_3$ and $E_3^{p,q}=0$ if $(p,q) \neq (0,0)$.
\end{itemize}
We then deduce that 
$H^i(\sh_{K^v(\oo),\bar s_v},\lexp p j^{=tg}_{!*} HT_\xi(\pi_v,\Pi_t))_{\mathfrak m}$
is zero for $i < t-t_0$ and for $i=t-t_0$ it has a non trivial torsion subspace.
\end{proof}

Consider now the filtration of stratification of $\Psi_\varrho:=\Psi_{K^v(\oo),\varrho}$\footnote{i.e. with infinite level at $v$} constructed using the
adjunction morphisms $j^{=h}_! j^{=h,*}$ as in \cite{boyer-torsion}
\begin{equation} \label{eq-fil-psi}
\Fil^1_!(\Psi_{\varrho}) \harrow \Fil^{2}_!(\Psi_{\varrho})
\harrow \cdots \harrow \Fil^{s}_!(\Psi_{\varrho})
\end{equation}
where $\Fil^{t}_!(\Psi_{\varrho})$ is the saturated image of 
$j^{=tg}_!j^{=tg,*} \Psi_{\varrho} \longrightarrow \Psi_{\varrho}$.
For a fixed $\pi_v$, let denote by
$\Fil^1_{!,\pi_v}(\Psi) \harrow \Fil^1_!(\Psi_\varrho)$ such that
$\Fil^1_{!,\pi_v}(\Psi) \otimes_{\overline \Zm_l} \overline \Qm_l \simeq \Fil^1_!(\Psi_{\pi_v})$
where $\Psi_{\pi_v}$ is the direct factor of $\Psi_\varrho \otimes_{\overline \Zm_l} \overline \Qm_l$
associated to $\pi_v$, cf. \cite{boyer-torsion}. 
From \cite{boyer-duke} 3.3.5,
we have the following resolution of $\gr^t_{!,\pi_v}(\Psi_\varrho)$
\begin{multline} \label{eq-resolution-psi2}
0 \rightarrow j^{=sg}_! HT(\chi_v,LT_{t,s}(\pi_v)) \otimes \pi^{\vee}_v(\frac{s-t}{2})
\longrightarrow \\
 j^{=(s-1)g}_!HT(\pi_v,LT_{t,s-1}(\pi_v)) \otimes \pi^{\vee}_v(\frac{s-t-1}{2}) 
\longrightarrow \\ \cdots \longrightarrow j^{=tg}_! HT(\pi_v,\st_t(\pi_v)) \otimes 
\pi^{\vee}_v \longrightarrow \gr^{t}_{!,\pi_v}(\Psi_\varrho) \rightarrow 0,
\end{multline}
where $LT_{t,t+\delta}(\pi_v) \hookrightarrow \st_t(\pi_v \{ -\delta/2 \}) \times \speh_{\delta}(\pi_v \{ t/2 \}),$
is the only irreducible sub-space of this induced representation,

We can then apply the previous arguments a)-d) above:
for $t \leq t_0$ (resp. $t > t_0$) 
the torsion of $H^i(\sh_{K^v(\oo),\bar s_v},\gr^{t}_{!,\pi_v}(\Psi_{v,\xi}))_{\mathfrak m}$
is trivial for any $i \leq t-t_0$ (resp. for all $i$) 
and the free parts are concentrated for $i=0$. 
Using the spectral sequence associated to the previous filtration, 
we can then conclude that $H^{1-t_0}(\sh_{K^v(\oo),\bar s_v},\Psi_{v,\xi})_{\mathfrak m}$ 
would have non trivial torsion 
which is false as $\mathfrak m$ is supposed to be KHT-free.
\end{proof}

\noindent
In particular the previous spectral sequence gives us a filtration of 
$H^{d-1}(\sh_{K^v(\oo),\bar \eta_v},\LC_{\xi,\overline \Fm_l})_{\mathfrak m}$ whose graded parts 
are
$$H^{0}(\Sh_{K^v(\oo),\bar s_v},\gr^{-p} (\gr^k_!(\Psi_{K,\varrho,\xi})))_{\mathfrak m} \otimes_{\overline \Zm_l} \overline \Fm_l,$$
for $\varrho$ describing the equivalence classes of irreducible $\overline \Fm_l$-supercuspidal representation of $\GL_g(F_v)$ with $1 \leq g \leq d$, and then
$1 \leq k \leq p \leq \lfloor \frac{d}{g} \rfloor$.

\subsection{Local and global monodromy}
\label{para-mono}

Consider a fixed irreducible $\overline \Fm_l$-supercuspidal representation $\varrho$
of $\GL_g(F_v)$ and denote by 
$\Psi_\varrho$ the
direct factor of $\Psi_{K^v(\oo),v}$ associated to $\varrho$. 

Over $\overline \Qm_l$, the monodromy operator
define a nilpotent morphism 
$N_{\varrho,\overline \Qm_l}: \Psi_\varrho \otimes_{\overline \Zm_l} \overline \Qm_l
\longrightarrow \Psi_\varrho(1)  \otimes_{\overline \Zm_l} \overline \Qm_l$
compatible with the filtration $\Fil^\bullet_!(\Psi_\varrho)$ in the sense that
$\Fil^t_!(\Psi_\varrho)  \otimes_{\overline \Zm_l} \overline \Qm_l$ coincides
with the kernel of $N_{\varrho,\overline \Qm_l}^t$. The aim of this section is to
construct a $\overline \Zm_l$-version $N_\varrho$ of $N_{\varrho,\overline \Qm_l}$ such that
$\Fil^t_!(\Psi_\varrho)  \otimes_{\overline \Zm_l} \overline \Fm_l$ coincides with
the kernel of $N_\varrho^t  \otimes_{\overline \Zm_l} \overline \Fm_l$.

\bigskip

\noindent \textit{First step}: consider 
$$0 \rightarrow \Fil^{1}_!(\Psi_\varrho) \longrightarrow \Psi_\varrho \longrightarrow
\coFil^1_!(\Psi_\varrho) \rightarrow 0,$$
and the following long exact sequence
$$
0 \rightarrow \hom(\coFil^1_!(\Psi_\varrho),\Psi_\varrho(1)) \longrightarrow
\hom(\Psi_\varrho,\Psi_\varrho(1)) \\ \longrightarrow \hom(\Fil^1_!(\Psi_\varrho),\Psi_\varrho(1))
\longrightarrow \cdots
$$
where $\hom$ is taken
in the category of equivariant Hecke perverse sheaves with an action of
$\gal(\overline F_v/F_v)$.
As $\Fil^1_!(\Psi_\varrho)\otimes_{\overline \Zm_l} \overline \Qm_l$ coincides with 
the kernel of $N_{\varrho,\overline \Qm_l}$, then
$N_{\varrho,\overline \Qm_l} \in \hom(\Psi_\varrho,\Psi_\varrho) \otimes_{\overline \Zm_l}
\overline \Qm_l$ comes from 
$\hom(\coFil^1_!(\Psi_\varrho),\Psi_\varrho(1))\otimes_{\overline \Zm_l}
\overline \Qm_l$, so that we focus on  $\hom(\coFil^1_!(\Psi_\varrho),\Psi_\varrho(1))$.
From
$$0 \rightarrow \gr^2_!(\Psi_\varrho) \longrightarrow \coFil^1_!(\Psi_\varrho) 
\longrightarrow \coFil^2_!(\Psi_\varrho) \rightarrow 0,$$
we obtain
\begin{multline*}
0 \rightarrow \hom(\coFil_!^2(\Psi_\varrho),\Psi_\varrho(1)) \longrightarrow
\hom(\coFil_!^1(\Psi_\varrho),\Psi_\varrho(1)) \longrightarrow \\
\hom(\gr_!^2(\Psi_\varrho),\Psi_\varrho(1)) \longrightarrow \ext^1(\coFil_!^2(\Psi_\varrho),\Psi_\varrho(1))) \longrightarrow \cdots
\end{multline*}
The socle of $\Psi_\varrho \otimes_{\overline \Zm_l} \overline \Qm_l$
being contained in $\Fil^1_!(\Psi_\varrho)\otimes_{\overline \Zm_l} \overline \Qm_l$,
any map $\coFil_!^2(\Psi_\varrho) \longrightarrow \Psi_\varrho(1))$ can not be
equivariant for the Galois action, so that we are led to look at 
$$\hom(\gr_!^2(\Psi_\varrho),\Psi_\varrho(1)) \simeq
\hom(\gr_!^2(\Psi_\varrho), \Fil_*^1(\gr_!^1(\Psi_\varrho(1))))$$
where
$$0 \rightarrow \Fil_*^1(\gr_!^1(\Psi_\varrho)) \longrightarrow \Fil_!^1(\Psi_\varrho) 
\longrightarrow \coFil_*^1(\Fil_!^1(\Psi_\varrho) \rightarrow 0.$$
Note that 
$\gr_!^2(\Psi_\varrho)\otimes_{\overline \Zm_l} \overline \Qm_l \simeq 
\Fil_*^1(\gr_1^1(\Psi_\varrho(1))) \otimes_{\overline \Zm_l} \overline \Qm_l$ 
and their $\overline \Zm_l$-structure is obtained, cf. the introduction of \cite{boyer-duke}
or equation (\ref{eq-resolution0}), 
through the strict $\overline \Zm_l$-epimorphisms
$$j^{=2g}_! j^{=2g,*}\gr_!^2(\Psi_\varrho) \twoheadrightarrow \gr_!^2(\Psi_\varrho), \quad
\hbox{and} \quad
j^{=2g}_! j^{=2g,*}\Fil_*^1(\gr_!^1(\Psi_\varrho)) \twoheadrightarrow  
\Fil_*^1(\gr_!^1(\Psi_\varrho)),$$
cf. figure \ref{fig-psi-filtration} and the notations of the 
beginning of \S \ref{para-exchange}.

In particular to prove that $\gr_!^2(\Psi_\varrho)$ is isomorphic to
$\Fil_*^1(\gr_1^1(\Psi_\varrho(1)))$, it suffices to prove that 
the two local systems
$j^{=2g,*}\gr_!^2(\Psi_\varrho)$ and $ j^{=2g,*}
\Fil_*^1(\gr_1^1(\Psi_\varrho(1)))$ are isomorphic.
In this case we can take\footnote{As it is not clear that 
$\ext^1(\coFil_!^2(\Psi_\varrho),\Psi_\varrho(1)))$ is torsion free, we can not claim at
this stage that $N_v \in \hom( \Psi_\varrho,\Psi_\varrho(1))$.}
$N_v \in \hom( \Psi_\varrho,\Psi_\varrho(1)) \otimes_{\overline \Zm_l} \overline \Qm_l$ so that, over $\overline \Zm_l$ we have 
$\Fil^1_*(\gr^1_!(\Psi_\varrho(1)))=N_v(\Fil^2_!(\Psi_\varrho))$.

\noindent \textit{More generally} to prove that the two perverse sheaves
$\gr_!^{h+1}(\Psi_\varrho)$ and $\Fil_*^1(\gr_!^h(\Psi_\varrho(1)))$ are isomorphic,
it suffices to prove that the two local systems
$j^{=(h+1)g,*}\gr_!^{h+1}(\Psi_\varrho)$ and $j^{=(h+1)g,*} \Fil_*^1(\gr_!^h(\Psi_\varrho(1)))$ 
are isomorphic.

\bigskip

\noindent \textit{Second step}: we want to prove that the local systems
$j^{=2g,*}\gr_!^2(\Psi_\varrho)$ and $j^{=2g,*}
\Fil_*^1(\gr_!^1(\Psi_\varrho))$ are isomorphic. Consider first the following situation:
let $\LC_k$ and $\LC_{k+1}$ be $\overline \Zm_l$-local systems on a scheme $X$ 
such that:
\begin{itemize}
\item $\LC_k \hookrightarrow \LC_{k+1}$ where the cokernel 
$\gr_{k+1}$ is torsion free;

\item $\LC_{k+1} \otimes_{\overline \Zm_l} \overline \Qm_l \simeq
(\LC_k \otimes_{\overline \Zm_l} \overline \Qm_l) \oplus (\gr_{k+1} \otimes_{\overline \Zm_l} \overline \Qm_l)$ where $\gr_{k+1} \otimes_{\overline \Zm_l} 
\overline \Qm_l$ is supposed to be irreducible;

\item we introduce
$$\xymatrix{
\gr'_{k+1} \ar@{^{(}-->}[r] \ar@{^{(}-->}[d] & \gr_{k+1,\overline \Qm_l} \ar@{^{(}->}[d] \\
\LC_{k+1} \ar@{^{(}->}[r] & \LC_{k+1} \otimes_{\overline \Zm_l} \overline \Qm_l.
}$$
We moreover suppose that  $\gr_{k+1} \otimes_{\overline \Zm_l} \overline \Fm_l$
is also irreducible so the various stable $\overline \Zm_l$-lattices of $\gr_{k+1}$ 
are homothetic.
\end{itemize}
We then have
$$0 \rightarrow \LC_k \oplus \gr'_{k+1} \longrightarrow \LC_{k+1} 
\longrightarrow T \rightarrow 0,$$
where $T$ is torsion and can be viewed as a quotient
$$\LC_k \hookrightarrow \LC'_k \twoheadrightarrow T, \quad 
\gr'_{k+1} \hookrightarrow \gr_{k+1} \twoheadrightarrow T,$$
with 
$$\LC_k \hookrightarrow \LC_{k+1} \twoheadrightarrow \gr_{k+1}, \qquad
\gr'_{k+1} \hookrightarrow \LC_{k+1} \twoheadrightarrow \LC'_k.$$
As $\gr_{k+1} \otimes_{\overline \Zm_l} \overline \Qm_l$ is irreducible, then
$\gr'_{k+1} \hookrightarrow \gr_{k+1}$ is given by multiplication by $l^{\delta_k}$ and,
as the stable lattices of $\gr_k \otimes_{\overline \Zm_l} \overline \Qm_l$
are all isomorphic, the
extension is characterized by this $\delta_k$.

Consider then the $\overline \Zm_l$-local system $\LC:=j^{=g,*} \Psi_\varrho$
and recall that 
$$\LC \otimes_{\overline \Zm_l} \overline \Qm_l \simeq 
\bigoplus_{i=1}^r HT_{\overline \Qm_l}(\pi_{v,i},\pi_{v,i}),$$
where we fix any numbering of $\cusp(\varrho)=\{\pi_{v,1},\cdots,\pi_{v,r}\}$. 
For $k=1,\cdots,r$, we introduce
$$\xymatrix{
\LC_k \ar@{^{(}-->}[r] \ar@{^{(}-->}[d] & \bigoplus_{i=1}^k HT(\pi_{v,i},\pi_{v,i}) 
\ar@{^{(}->}[d] \\
\LC \ar@{^{(}->}[r] & \LC \otimes_{\overline \Zm_l} \overline \Qm_l.
}$$
Let denote by $T_{k+1}$ the torsion local system such that
$$0 \rightarrow \LC_k \oplus \gr_{k+1} \longrightarrow \LC_{k+1} \longrightarrow T_{k+1}
\rightarrow 0,$$
where $\gr_{k+1}:=\LC_{k+1}/\LC_k$, as above. We can apply the previous remark
and denote by $\delta_k$ the power of $l$ which define the homothety
$\gr'_{k+1} \hookrightarrow \gr_{k+1} \twoheadrightarrow T_{k+1}$. The set 
$\{ \delta_k: k=1,\cdots,r\}$ is then a numerical data to characterize $\LC$ inside
$j^{=1,*} \Psi_\varrho \otimes_{\overline \Zm_l} \overline \Qm_l$. 

(i) To control $j^{=2g,*} \Fil_*^1(\gr^1_!(\Psi_\varrho))$, we use the general description
above with
\begin{itemize}
\item local systems $\LC^+_k$ for $k=1,\cdots,r$
so that $\LC^+_k \otimes_{\overline \Zm_l} \overline \Qm_l \simeq \bigoplus_{i=1}^k
HT_{\overline \Qm_l}(\pi_{v,i},\st_2(\pi_{v,i})(-1/2)$;

\item with $\gr^{+,'}_{k+1}$ defined, as before, with
$$0 \rightarrow \LC^+_k \oplus \gr^{+,'}_{k+1} \longrightarrow \LC^+_{k+1}
\longrightarrow T_{k+1} \rightarrow 0,$$
where $T_{k+1}$ is killed by $l^{\delta^+_{k+1}}$.
\end{itemize}
We want to prove that $\delta^+_k=\delta_k$ for every $k=1,\cdots,r$ where
$\{ \delta_k: k=1,\cdots,r \}$ is the numerical data associated to $j^{=1,*} \Psi_\varrho$.

Let denote by
$$j^{=1}_{\neq 1}: \sh_{K,\bar s_v} \setminus \sh_{K,\bar s_v,1}^{\geq 1}
\hookrightarrow \sh_{K,\bar s_v}, \qquad i_1^1:\sh_{K,\bar s_v,1}^{\geq 1}
\hookrightarrow \sh_{K,\bar s_v}^{\geq 1}=\sh_{K,\bar s_v}.$$
From \cite{boyer-duke} lemma B.3.2,
$j^{=2g,*} \Fil_*^1(\gr^1_!(\Psi_\varrho))$ is obtained as follows. Let 
$$P:=\lexp p h^{-1} i_1^{1,*} j^{=1}_{\neq 1,*} j^{=1,*}_{\neq 1} \Psi_\varrho$$ 
so that
$$0 \rightarrow P \longrightarrow 
j^{=1}_{\neq 1,!} j^{=1,*}_{\neq 1} \Psi_\varrho \longrightarrow  
\lexp p j^{=1}_{\neq 1,!*} j^{=1,*}_{\neq 1} \Psi_\varrho \rightarrow 0.$$
Then $P$ is the cosocle of $i_1^{1,*} \Fil_*^1(\gr^1_!(\Psi_\varrho))$ so that
$$j^{=2g,*} \Fil_*^1(\gr^1_!(\Psi_\varrho)) \simeq j^{=2g,*} P \times_{P_{1,d}(F_v)}
\GL_d(F_v),$$
where induction has the same meaning as in (\ref{eq-induction}).

Note then that the numerical data associated to $j^{=2g,*} P$ are also given by
$\{ \delta_k^+: k=1,\cdots,r \}$.
With the previous notations, consider the data associated to $\LC:=j^{=g,*} \Psi_\varrho$,
i.e. a filtration 
$$\LC_1 \subset \LC_2 \subset \cdots \subset \LC_r=\LC$$
with graded parts $\gr^k$ and 
$\gr'_k \hookrightarrow \gr_k$ 
is given by multiplication by $l^{\delta_k}$. 
We then have a strict filtration
$$\lexp p h^{-1} i_1^{1,*} j^{=1}_{\neq 1,*} \LC_1 \subset \cdots
\subset \lexp p h^{-1} i_1^{1,*} j^{=1}_{\neq 1,*} \LC_r=P,$$
with graded parts $\lexp p h^{-1}i_1^{1,*} j^{=1}_{\neq 1,*} \gr_k$.
Indeed we have
\begin{multline*}
\lexp p h^{-2} i_1^{1,*} j^{=1}_{\neq 1,*} \gr_{k+1}=0 \longrightarrow
\lexp p h^{-1} i_1^{1,*} j^{=1}_{\neq 1,*} \LC_k \longrightarrow \\
\lexp p h^{-1} i_1^{1,*} j^{=1}_{\neq 1,*} \LC_k \longrightarrow
\lexp p h^{-1} i_1^{1,*} j^{=1}_{\neq 1,*} \gr_{k+1}
\longrightarrow \lexp p h^0 i_1^{1,*} j^{=1}_{\neq 1,*} \LC_k
\end{multline*}
where the free quotient of $\lexp p h^0 i_1^{1,*} j^{=1}_{\neq 1,*} \LC_k$ is zero.
Moreover it is torsion free because its torsion corresponds to the difference between
$p$ and $p+$ intermediate extensions which are equal here from the main result
of \cite{boyer-duke}. We then apply the exact functor $j^{=2g,*}$ and we
induce from $P_{1,d}(F_v)$ to $\GL_d(F_v)$ to obtain the filtration
$\LC_\bullet^+$ of $j^{=2g,*}\Fil^1_*(\gr^1_!(\Psi_\varrho))$ where
$\gr_k^{+,'} \hookrightarrow \gr_k^{+}$
is given by multiplication by $l^{\delta_k}$.

(ii) Dually the same arguments applied to 
$$0 \rightarrow \lexp p j^{=1}_{\neq 1,!*} j^{=1,*}_{\neq 1} \Psi_\varrho \longrightarrow
\lexp p j^{=1}_{\neq 1,*} j^{=1,*}_{\neq 1} \Psi_\varrho \longrightarrow Q \rightarrow 0,$$
give us that $j^{=2g,*} Q$ is characterized by the data $\{ \delta_k: k=1,\cdots,r \}$.
After inducing from $P_{1,d}(F_v)$ to $\GL_d(F_v)$, we obtain the description
of the local system $\AC:=j^{=2g,*} A$ where $A$ is defined as follows:
$$0 \rightarrow \lexp p j^{=g}_{!*} j^{=g,*} \coFil^1_*(\Psi_\varrho) \longrightarrow
\coFil^1_*(\Psi_\varrho) \longrightarrow A \rightarrow 0.$$
Concretely this means that $\lexp p j^{=2g}_{!*} \AC$ is the socle $A_1$ of $A$,
which corresponds to the square dot in the right side of the figure \ref{fig-psi-filtration}.

We are interested by the local system associated to $j^{=2g,*}$ of the cosocle of
$\Fil^2_!(\Psi_\varrho)$, which corresponds to the square dot in the left 
side of the figure \ref{fig-psi-filtration}.
As explained in \S \ref{para-exchange}, we have to use basic exchange
steps as many times as needed to move $A_1$ until it appears 
as the cosocle of $\Fil^2_!(\Psi_\varrho) \hookrightarrow \Psi_\varrho$.

Note then that all the perverse sheaves which are exchanged with $A_1$ during
this process, are lattice of
$j^{=tg}_{!*} HT_{\overline \Qm_l}(\pi_v,\st_t(\pi_v))(\frac{1-t+\delta}{2})$ with
$t \geq 3$, cf. figure \ref{fig-psi-filtration}. 
As explained in the remark after the definition of the exchange basic step,
as $\lexp p j^{=2g}_{!*} HT(\pi_v,\st_2(\pi_v)) \simeq
\lexp {p+} j^{=2g}_{!*} HT(\pi_v,\st_2(\pi_v))$, for all these exchange, we have
$T=0$ and $A_1$ remains unchanged during all the basic exchange steps.

\medskip

\noindent \textit{Third step}: at this stage we constructed a 
$\overline \Qm_l$-monodromy operator $N_v$ such $\Fil^1_*(\gr^1_!(\Psi_\varrho(1)))=
N_v(\Fil^2_!(\Psi_\varrho))$. 
Recall that this monodromy operator induces 
$$\alpha: \coFil^t_!(\cogr^t_*(\Psi_\varrho)) \longrightarrow  \cogr^{t+1}_*(\Psi_\varrho(1))$$
such that $j^{=(t+1)g,*} \circ \alpha$ is then an isomorphism over $\overline \Zm_l$. 
We say that $\alpha$
is an isomorphism. Indeed consider
$$
0 \rightarrow \lexp p j^{=tg}_{!*} j^{=tg,*} \cogr^t_*(\Psi_\varrho) \longrightarrow
\cogr^t_*(\Psi_\varrho) \longrightarrow \coFil^t_!( \cogr^t_*(\Psi_\varrho) ) \rightarrow 0,
$$
with the following two strict monomorphisms
\begin{equation} \label{eq-j1}
\alpha_1:\cogr^{t+1}_*(\Psi_\varrho) \hookrightarrow j^{=(t+1)g}_* j^{=(t+1)g,*}  
\cogr^{t+1}_*(\Psi_\varrho)
\end{equation}
and
\begin{equation} \label{eq-j2}
\alpha_2:\coFil^t_!( \cogr^t_*(\Psi_\varrho(1)) ) \hookrightarrow j^{=(t+1)g}_* j^{=(t+1)g,*}
\coFil^t_!( \cogr^t_*(\Psi_\varrho(1)) ).
\end{equation}
By composing $\alpha$ with $\alpha_2$ 
in (\ref{eq-j2}), we obtain
\begin{multline} \label{eq-jj}
\alpha_1, \alpha_2 \circ \alpha \in \hom \Bigl ( \cogr^{t+1}_*(\Psi_\varrho),
j^{=(t+1)g}_* j^{=(t+1)g,*}  \cogr^{t+1}_*(\Psi_\varrho) \Bigr ) \\ \simeq \hom
\Bigl ( j^{=(t+1)g,* }\cogr^{t+1}_*(\Psi_\varrho),
j^{=(t+1)g,*}  \cogr^{t+1}_*(\Psi_\varrho) \Bigr ),
\end{multline}
by adjunction. By hypothesis $\alpha_1$ and $\alpha_2 \circ \alpha$
coincides in this last space, so they are equal and $\alpha$ is then
an isomorphism.

\begin{nota} 
Under the hypothesis of theorem \ref{thm-recall} on $\mathfrak m$,
the action of $N_\varrho$ on $\Psi_\varrho$ defined above for every 
$\overline \Fm_l$-character $\varrho$, induces a nilpotent 
monodromy operator $N^{coho}_{\mathfrak m,v}$ on
$H^{0}(\sh_{I,\bar s_v},\Psi_{v,\xi})_{\mathfrak m}$.
We also denote by 
$\overline N^{coho}_{\mathfrak m,v}:=N^{coho}_{\mathfrak m,v} \otimes_{\overline \Zm_l}
\overline \Fm_l$ acting on $H^{0}(\sh_{I,\bar s_v},\Psi_{v,\xi})_{\mathfrak m} 
\otimes_{\overline \Zm_l} \overline \Fm_l$
\end{nota}

%% file: fig-psi.tex
\newrgbcolor{xdxdff}{0.49 0.49 1}
\newrgbcolor{ududff}{0.30 0.30 1}
\psset{xunit=1cm,yunit=1cm,algebraic=true,dimen=middle,dotstyle=o,dotsize=5pt 0,linewidth=2pt,arrowsize=3pt 2,arrowinset=0.25}
\begin{pspicture*}(-10.37,-3.87)(4.25,4.87)
\rput{45.72}(-7.56,1.39){\psellipse[linewidth=2pt](0,0)(2.85,0.595)}
\rput{44.44}(-6.99,0.14){\psellipse[linewidth=2pt](0,0)(2.226,0.427)}
\rput{44.443}(-6.99,0.14){\psellipse[linewidth=2pt](0,0)(2.226,0.4275)}
\rput{46.157}(-6.62,-1.56){\psellipse[linewidth=2pt](0,0)(1.457,0.4035)}
\rput[tl](-9.4,3){$j^{=g}_! j^{=g,*} \twoheadrightarrow$}
\rput[tl](-10.2,-0.4){$\gr^1_!$}
\rput[tl](-9.3,-1.33){$\gr^2_!$}
\rput[tl](-8.3,-2.57){$\gr^3_!$}
\psline[linewidth=2pt](-8,1)(-6,3)
\psline[linewidth=2pt](-7,0)(-6,1)
\rput[tl](-6.5,2.5){$\leadsto \Fil^1_*(\gr^1_!)$}
\rput[tl](-6.33,0.7){$\leadsto \Fil^2_*(\gr^2_1)$}
\rput{-45.1465}(-0.46,-1.49){\psellipse[linewidth=2pt](0,0)(2.835,0.63)}
\rput{-46.1852}(-0.01,0.01){\psellipse[linewidth=2pt](0,0)(2.1142,0.5129)}
\rput{-42.922}(0.47,1.56){\psellipse[linewidth=2pt](0,0)(1.447,0.479)}
\psline[linewidth=2pt](-1,-1)(1,-3)
\psline[linewidth=2pt](0,0)(1,-1)
\rput[tl](-3.23,-2){$j^{=g}_* j^{=g,*} \hookleftarrow$}
\rput[tl](-3.41,1.1){$\cogr^1_*$}
\rput[tl](-2.49,2.2){$\cogr^2_*$}
\rput[tl](-1.5,3.1){$\cogr^3_*$}
\rput[tl](0.49,-1.81){$\leadsto \coFil^1_!(\cogr^1_*)$}
\rput[tl](0.67,0.01){$\leadsto \coFil^2_!(\cogr^2_*)$}
\psline[linewidth=2pt](-3.75,3.79)(-3.75,-3.45)
\begin{scriptsize}
\psdots[dotstyle=*,linecolor=xdxdff](-9,0)
\psdots[dotstyle=*,linecolor=ududff](-8,1)
\psdots[dotstyle=*,linecolor=ududff](-7,2)
\psdots[dotstyle=*,linecolor=ududff](-6,3)
\psdots[dotstyle=square*,dotscale=2 2,linecolor=ududff](-8,-1)
\psdots[dotstyle=*,linecolor=xdxdff](-7,0)
\psdots[dotstyle=*,linecolor=ududff](-6,1)
\psdots[dotstyle=*,linecolor=ududff](-7,-2)
\psdots[dotstyle=*,linecolor=ududff](-6,-1)
\psdots[dotstyle=*,linecolor=ududff](-6,-3)
\psdots[dotstyle=*,linecolor=xdxdff](-2,0)
\psdots[dotstyle=*,linecolor=ududff](-1,1)
\psdots[dotstyle=*,linecolor=xdxdff](0,2)
\psdots[dotstyle=*,linecolor=ududff](1,3)
\psdots[dotstyle=square*,dotscale=2 2,linecolor=ududff](-1,-1)
\psdots[dotsize=4pt 0,dotstyle=*,linecolor=darkgray](0,0)
\psdots[dotstyle=*,linecolor=ududff](1,1)
\psdots[dotstyle=*,linecolor=xdxdff](0,-2)
\psdots[dotstyle=*,linecolor=ududff](1,-1)
\psdots[dotstyle=*,linecolor=ududff](1,-3)
\end{scriptsize}
\end{pspicture*}

%% file: proof.tex
\section{Uniformity of automorphic sub-lattices}

\subsection{Infinite level at $v$}
\label{para-oolevel}

For $K \in \KC$ a compact open subgroup of $G(\Am_\Qm^\oo)$ as before,
we consider the category $\CC(K)$ 
of finitely generated $\OC_L$-modules with a continuous $K$-action and let
$\CC'(K)$ be a Serre subcategory of $\CC(K)$.
Let $\CC(K)_Z$ be the subcategory of $\CC(K)$ consisting of those $\sigma \in \CC$
possessing a central character which lifts 
$\prod_{v \in \Sigma^+(K)} (\det \rbar_{|I_v} \bar \varepsilon)
\circ \art$  and let $\CC'(K)_Z$ be a Serre subcategory of $\CC(K)_Z$.

For each $v \in \Sigma^+(K)$, we fix 
\begin{itemize}
\item a lift of a geometric Frobenius element $\Frob_v \in G_v:=\gal(\overline F_v/F_v)$,

\item an element $\alpha_v \in \overline \Zm_l$ lifting $\det \overline \rho_v(\Frob_v)$,

\item a character $\psi_{\sigma,v}:G_v \longrightarrow \overline \Zm_l^\times$ such that
$\psi_{\sigma,v}(\Frob_v)=\alpha_v$ and the composite with Artin map has a restriction to 
$I_v$ equal to the central character $\sigma$.
\end{itemize}
For $\sigma \in \CC'(K)_Z$, we define
$$S_{K}(\sigma):=H^{0}(\Sh_{K,\bar s_v}, \Psi_v \otimes
\cL_{\sigma^\vee})_{\mathfrak m}^\vee,$$
where $\cL_{\sigma^\vee}$ is the sheaf associated to $\sigma^\vee$.
When we consider only the direct factor $\Psi_\varrho$ of $\Psi_v$, we then
denote by $S_{K,\varrho}(\sigma)$ the associated space.

\rem As the $\mathfrak m$-localized cohomology groups are concentrated
in middle degree, note that the functor 
$\sigma \mapsto S_{K,\varrho}(\sigma)$ is exact.

Let $\sigma^v$ be a continuous finitely generated representation of $K^v$ and let
$\sigma_0(v)$ be the representation of $K=K^vK_v$ given by twisting the action
of $K^v$ on $\sigma^v$ by the character $\psi \circ \art_{F_v} \circ \det$ of $K_v$.
We assume that $\sigma \in \CC(K)_Z$. We then define
$$S_{K^v}(\sigma^v):=
\lim_{\rightarrow n}H^{0}(\sh_{K^v(n),\bar s_v}, 
\Psi_v \otimes \LC_{\sigma_0(v)^\vee)}^\vee)_{\mathfrak m}.$$
We let $\Tm(\sigma^v)_{\mathfrak m}$ denote the image of $\Tm^S$ in
$\End_\OC(S_{K^v}(\sigma^v))$ and we write $\rho(\sigma^v)_{\mathfrak m}$ for the 
composite $\gal_{F,s} \longmapright{\rho^{univ}} \GL_d(R_S^{univ}) \longrightarrow
\GL_d(\Tm(\sigma^v)_{\mathfrak m})$. We set
\begin{equation} \label{eq-Msigmav}
M_{K^v}(\sigma^v) \defeq \Bigl ( \hom_{\Tm(\sigma^v)_{\mathfrak m}[\gal_{F,S}]}
( \rho(\sigma^v)_{\mathfrak m},S_{K^v}(\sigma_v)_{\mathfrak m}))^*,
\end{equation}
as well as, with similar notations
$$M_{K^v}(\sigma) \defeq \Bigl ( \hom_{\Tm(\sigma)_{\mathfrak m}[\gal_{F,S}]}
( \rho(\sigma)_{\mathfrak m},S_{K}(\sigma)_{\mathfrak m}))^*.$$
For a fixed $K_v$ and a representation $\sigma_v$ of $K_v$ such that $\sigma:=\sigma_v \otimes_{\OC} \sigma^v$ is an element of $\CC(K)_Z$, then we have a natural isomorphism
\begin{equation} \label{eq-Msigma-v}
M_K(\sigma) \simeq \hom_{K_v}(M_{K^v}(\sigma^v),\sigma_v^*)^*.
\end{equation}


\subsection{Typicity}
\label{para-mono2}

As explained in \cite{harris-taylor}, the $\overline \Qm_l$-cohomology of $\sh_{K,\bar \eta}$
can be written as
$$H^{d-1}(\sh_{K,\bar \eta},\overline \Qm_l)_{\mathfrak m} \simeq 
\bigoplus_{\pi \in \AC_{K}(\mathfrak m)} (\pi^{\oo})^K \otimes \sigma(\pi^\oo),$$
where 
\begin{itemize}
\item $\AC_{K}(\mathfrak m)$ is the set of equivalence classes of automorphic
representations of $G(\Am)$ with non trivial $K$-invariants and such that its
modulo $l$ Satake's parameters, 
outside the set $S$ of places dividing the level $K$, are prescribed by $\mathfrak m$,

\item and $\sigma(\pi^\oo)$ is a representation of $\gal_{F,S}$. 
\end{itemize}
As $\overline \rho_{\mathfrak m}$ is supposed
to be absolutely irreducible, then as explained in chapter VI of \cite{harris-taylor},
if $\sigma(\pi^\oo)$ is non zero, then $\pi$ is a weak transfer of a cohomological 
automorphic representation $(\Pi,\psi)$ of $\GL_d(\Am_F) \times \Am_F^\times$
with $\Pi^\vee \simeq \Pi^c$ where $c$ is the complex conjugation.
Attached to such a $\Pi$ is a global Galois representation 
$\rho_{\Pi,l}:\gal_{F,S} \longrightarrow \GL_d(\overline \Qm_l)$ which is irreducible.

\begin{thm} (cf. \cite{nekovar-fayad} theorem 2.20) \\
If $\rho_{\Pi,l}$ is strongly irreducible, meaning it remains irreducible when
it is restricted to any finite index subgroup, then $\sigma(\pi^\oo)$ is a semi-simple
representation of $\gal_{F,S}$.
\end{thm}

\rem The Tate conjecture predicts that $\sigma(\pi^\oo)$ is always semi-simple.

\begin{defin} (cf. \cite{scholze-LT} \S 5)   \label{defi-typic}
We say that $\mathfrak m$ is KHT-typic for $K$ if, as 
a $\Tm(K)_{\mathfrak m}[\gal_{F,S}]$-module,
$$H^{d-1}(\sh_{K,\bar \eta},\overline \Zm_l)_{\mathfrak m} \simeq 
\sigma_{\mathfrak m,K} 
\otimes_{\Tm(K)_{\mathfrak m}} \rho_{\mathfrak m,K},$$
for some $\Tm(K)_{\mathfrak m}$-module $\sigma_{\mathfrak m,K}$ on which 
$\gal_{F,S}$ acts trivially and 
$$\rho_{\mathfrak m,K}:\gal_{F,S} \longrightarrow 
\GL_d(\Tm(K)_{\mathfrak m})$$ 
is the stable lattice of
$\bigoplus_{\widetilde{\mathfrak m} \subset \mathfrak m} \rho_{\widetilde{\mathfrak m}}$ introduced in the introduction.
\end{defin}

\begin{prop} \label{prop-typic}
We suppose that for all $\pi \in \AC_{K}(\mathfrak m)$, the Galois representation
$\sigma(\pi^\oo)$ is semi-simple. Then $\mathfrak m$ is KHT-typic for $K$.
\end{prop}

\begin{proof}
By proposition 5.4 of \cite{scholze-LT} it suffices to deal with 
$\overline \Qm_l$-coefficients. From \cite{harris-taylor} proposition VII.1.8 and the
semi-simplicity hypothesis, then $\sigma(\pi^\oo) \simeq 
\widetilde R(\pi)^{\bigoplus n(\pi)}$
where $\widetilde R(\pi)$ is of dimension $d$. We then write
$$(\pi^\oo)^K \otimes_{\overline \Qm_l} R(\pi)\simeq 
(\pi^\oo)^K \otimes_{\Tm(K)_{\mathfrak m,\overline \Qm_l}} 
(\Tm(K)_{\mathfrak m,\overline \Qm_l})^d,$$
and $(\pi^\oo)^K \otimes_{\overline \Qm_l} \sigma(\pi^\oo) \simeq
((\pi^\oo)^K)^{\bigoplus n(\pi)} \otimes_{\Tm(K)_{\mathfrak m,\overline \Qm_l}} 
(\Tm(K)_{\mathfrak m,\overline \Qm_l})^d$
and finally 
$$H^{d-1}(\sh_{K,\bar \eta},\overline \Qm_l)_{\mathfrak m} \simeq
\sigma_{\mathfrak m,K,\overline \Qm_l}  \otimes_{\Tm(K)_{\mathfrak m,\overline \Qm_l}} 
(\Tm(K)_{\mathfrak m,\overline \Qm_l})^d,$$
with $\sigma_{\mathfrak m,K,\overline \Qm_l}  \simeq \bigoplus_{\pi \in \AC_{K}(\mathfrak m)} 
((\pi^{\oo})^I)^{\bigoplus n(\pi)}$. The result then follows from
\cite{harris-taylor} theorem VII.1.9 which insures that $R(\pi) \simeq
\rho_{\widetilde{\mathfrak m}}$, if $\widetilde{\mathfrak m}$ is the prime ideal
associated to $\pi$,
\end{proof}

In particular with the notations of \S \ref{para-oolevel}, we have an isomorphism
$$M_{K^v}(\sigma^v) \otimes_{\Tm(\sigma_v)_{\mathfrak m}} \rho(\sigma^v)_{\mathfrak m}
\longrightarrow S_{K^v}(\sigma^v)_{\mathfrak m}.$$
%

%
%

\subsection{Lattices}
\label{para-lattice}

For each place $w \in S \setminus \{ v \}$ choose an inertial type $\tau_w$
and a lattice $\sigma^0(\tau_w)$ in $\sigma_L(\tau_w)$ and let
$\sigma^v:=\otimes_{w \in S \setminus \{ v \} } \sigma^0(\tau_w)$ the corresponding
representation of $K^v=\prod_{w \in S \setminus \{ v \} } K_w$.

Recall that $\OC_L$ is the ring of integers of a large enough finite extension
$L$ of $\Qm_l$.
Let $\lambda: \Tm_{\mathfrak m} \longrightarrow \OC_L$ be a system\footnote{Note that $\lambda \mod
\varpi_L$ is given by $\mathfrak m$, i.e. the modulo $l$ reduction of $\lambda$ is fixed.} of
Hecke eigenvalues corresponding to some minimal prime ideal of 
$\Tm_{\mathfrak m}:=\Tm(\sigma^v \sigma^0(\tau_v))_{\mathfrak m}$
for some lattice $\sigma^0(\tau_v)$ of $\sigma_L(\tau_v)$ define just after: 
note that
$\Tm_{\mathfrak m}$ does not depend on this choice. 
Let  $\Pi_v:=(M_{K^v}(\sigma^v)^*[1/l])[\lambda]$ which by strong multiplicity one 
is an irreducible representation of $\GL_d(F_v)$. 
We suppose that $\sigma_L(\tau_v)$ appears with multiplicity one in 
$(\Pi_v)_{|K_v}$. The eigenspace 
$M_{K^v}(\sigma^v)^*[\lambda]$ is a lattice in $\Pi_v$ and
$$\sigma_\lambda(\tau_v):=\sigma_L(\tau_v)\cap M_{K^v}(\sigma^v)^*[\lambda]$$
is $K_v$-stable lattice in $\sigma_L(\tau_v)$.

For any lattice $\sigma^1(\tau_v)$ of $\sigma_L(\tau_v)$,
it follows from (\ref{eq-Msigma-v}) that there is a natural isomorphism
\begin{equation} \label{eq-iso-M-hom}
\hom_{K_v}(\sigma^1(\tau_v),\sigma_\lambda(\tau_v))^* \simeq 
M_{K^v}(\sigma^v \sigma^1(\tau_v))/\lambda.
\end{equation}
%
Consider some injection 
$$\iota: \sigma_1(\tau_v) \hookrightarrow \sigma_2(\tau_v)$$
between two lattices of $\sigma_L(\tau_v)$. 
Then $\iota$ induces a injective map
$$
M_{K^v}(\sigma^v\sigma_1(\tau_v)) 
\hookrightarrow M_{K^v}(\sigma^v \sigma_2(\tau_v)).$$

\begin{lem}
Suppose there exists $\lambda_0$ so that the natural map induces by $\iota$
\begin{equation} \label{eq-Moo}
M_{K^v}(\sigma^v\sigma_1(\tau_v))/\lambda_0 \hookrightarrow
M_{K^v}(\sigma^v\sigma_2(\tau_v))/\lambda_0
\end{equation}
is an isomorphism, then this is true for every $\lambda \equiv \lambda_0 \mod \varpi_L$. 
\end{lem}

\begin{proof}
For any $\lambda$ the natural inclusion
$$\OC_L \simeq M_{K^v}(\sigma^v\sigma_1(\tau_v))/\lambda \hookrightarrow
M_{K^v}(\sigma^v\sigma_2(\tau_v))/\lambda \simeq \OC_L$$
is either an isomorphism or multiplication by some power of $\varpi_L$ and so zero modulo $\varpi_L$. 
By hypothesis 
modulo $\varpi_L$, the natural map $\iota$ induces an isomorphism
$$M_{K^v}(\sigma^v\sigma_1(\tau_v))/(\lambda_0,\varpi_L) \simeq
M_{K^v}(\sigma^v\sigma_2(\tau_v))/(\lambda_0,\varpi_L).$$
As the following diagram is commutative
$$\xymatrix{
R_\oo^{\tau_v} \ar[r]^{\lambda_0} \ar[dr]^{\lambda}
& \Zm_L \ar@{->>}[r] & \Fm_L \\
& \Zm_L \ar@{->>}[ur]
}$$
with common kernel $\mathfrak m$, we then deduce that
$\iota$ induces an isomorphism
$$M_{K^v}(\sigma^v\sigma_1(\tau_v))/(\lambda,\varpi_L) \simeq
M_{K^v}(\sigma^v\sigma_2(\tau_v))/(\lambda,\varpi_L),$$
so that the natural map induces by $\iota$
$$M_{K^v}(\sigma^v\sigma_1(\tau_v))/\lambda \hookrightarrow
M_{K^v}(\sigma^v\sigma_2(\tau_v))/\lambda$$
is also an isomorphism.
%
\end{proof}

%
%
%

%
%

\begin{prop} \label{prop-BLC} 
Consider two automorphic irreducible representations $\Pi$ and $\Pi'$ associated
to respectively two systems\footnote{with modulo $l$ reduction fixed by $\mathfrak m$} of Hecke eigenvalues 
$\lambda,\lambda':\Tm_{\mathfrak m} \longrightarrow \OC_L$. 
We denote by
$$\Pi_v:=(M_{K^v}(\sigma^v)^*[1/l])[\lambda] \quad \hbox{and} \quad 
\Pi'_v:=(M_{K^v}(\sigma^v)^*[1/l])[\lambda'].$$
Consider a $K$-type $\tau_v$ such that $\sigma_L(\tau_v)$ appears with
multiplicity one in $(\Pi_v)_{|K_v}$ (resp. in $(\Pi'_v)_{|K_v}$) and let 
$\sigma_\lambda(\tau_v)$ (resp. $\sigma_{\lambda'}(\tau_v)$) 
be the stable lattice in 
$\sigma_L(\tau_v)$ defined by
$$\sigma_\lambda(\tau_v):= \sigma_L(\tau_v) \cap M_{K^v}(\sigma^v)^*[\lambda], 
\quad
\hbox{resp. } \sigma_{\lambda'}(\tau_v):= \sigma_L(\tau_v) \cap 
M_{K^v}(\sigma^v)^*[\lambda'].$$
Then these two lattices are homothetic.
\end{prop}

\begin{proof}
Up to homothety any stable lattice $\sigma_+$ is such that
$$
\sigma_0:=\sigma_\lambda(\tau_v) 
\subsetneq \sigma_+ \subsetneq \varpi_L^{-\delta} \sigma_\lambda(\tau_v),$$
where $\delta$ is minimal and $\sigma_0 \not \subset \varpi_L \sigma_+$.
The natural morphism
$$\OC_L \simeq \hom_{\GL_d(\OC_v)}(\sigma_+,\sigma_\lambda(\tau_v)) \longrightarrow \hom_{\GL_d(\OC_v)}(\sigma_0,
\sigma_\lambda (\tau_v))\simeq \OC_L$$
is given by multiplication by $\varpi_L^\delta$,
and so, from the isomorphism (\ref{eq-iso-M-hom}), $\iota$ induces
\begin{equation} \label{eq-Moo2}
M_{K^v}(\sigma^v \sigma_0)/\lambda=
\varpi_L^\delta M_{K^v}(\sigma^v \sigma_+)/\lambda.
\end{equation}
If $\sigma_{\lambda'}(\tau_v)$ were equal to $\sigma_+$ then the morphism
induced by $\iota$
$$\hom_{\GL_d(\OC_v)}(\sigma_+,\sigma_{\lambda'}(\tau_v)) \hookrightarrow \hom_{\GL_d(\OC_v)}(\sigma_0,\sigma_{\lambda'}(\tau_v))$$
would be an isomorphism and by (\ref{eq-iso-M-hom}), the
natural injection would give the equality
$$M_{K^v}(\sigma^v \sigma_0)/\lambda'=
M_{K^v}(\sigma^v \sigma_+)/\lambda',$$
which, by the previous lemma, is not compatible with (\ref{eq-Moo2}).
%
\end{proof}

Concerning stable lattice of an irreducible admissible representation $\Pi_v$ of $\GL_d(F_v)$ defined
over some finite extension $L$ of $\Qm_l$. Recall
that the modulo $\varpi_L$ reduction of $\Pi_v$ has only one irreducible generic subquotient
$\tau_{gen}$. There exists then, cf. for example \cite{EGS} lemma 4.1.1, an unique stable lattice
$\Gamma^{gen}$ such that the socle of $\Gamma^{gen} \otimes_{\OC_L} \OC_L/\varpi_L$
is isomorphic to $\tau_{gen}$. Consider then any other stable lattices $\Gamma_1,\Gamma_2$ of 
$\Pi_v$ such that
$$\Gamma_1 \hookrightarrow \Gamma_2 \twoheadrightarrow T \neq (0)$$
with $\Gamma_1 \not \subset \varpi_L \Gamma_2$ and $T$ is $l$-torsion.

\begin{lem} \label{lem-lattice-gen}
Suppose that the socle of $T$ as a $\overline \Fm_l[\GL_d(F_v)]$-module
is isomorphic to $\tau_{gen}$.  Then we have $\Gamma_1 \simeq \Gamma^{gen}$.
\end{lem}

\begin{proof}
Up to homothety, we have 
$$\xymatrix{
T_1 \ar@{^{(}->>}[rr] & & T_2 \ar@{->>}[d] \\
\Gamma_1 \ar@{^{(}->}[r] \ar@{->>}[u] & \Gamma_2 \ar@{->>}[r]  \ar@{->>}[ur] & T \\
\Gamma^{gen} \ar@{^{(}->}[u] \ar@{^{(}->}[ur]
}$$
with $T_1 \hookrightarrow T_2 \twoheadrightarrow T$ which is exact. 

We suppose, by absurdity, that $T_1 \neq 0$: as the socle
of $\overline \Gamma_1:=\Gamma_1 \otimes_{\OC_L} \OC_L/\varpi_L$ is generic and the others 
subquotients are non generic, we deduce that $T_1$ has also a generic subquotient. As by hypothesis 
$T$ also has a generic subquotient and 
$\overline \Gamma_2:=\Gamma_2 \otimes_{\OC_L} \OC_L/\varpi_L$ owns exactly one
generic irreducible subquotient, then $\Gamma^{gen} \hookrightarrow \Gamma_2$ factors through
$\varpi_L \Gamma_2$. We denote by $\overline V$ the kernel 
$\overline \Gamma_1 \longrightarrow \overline \Gamma_2$ and 
$\overline W:=\overline \Gamma_1/\overline V$: 
$$\xymatrix{
\overline V \ar@{^{(}->}[r] \ar[dr]^0 & \overline \Gamma_1 \ar@{->>}[r]  \ar[d] & \overline W \ar[dl]
\ar@{-->}[dll]^\sim \\
\overline W \ar@{^{(}->}[r] & \overline \Gamma_2
}$$
Note that, as the socle of $T$ is generic,
then the socle of $\overline V$ is also generic. As $\overline \Gamma^{gen} \longrightarrow
\overline \Gamma_2$ is zero, we deduce that the image of 
$\overline \Gamma^{gen} \longrightarrow \overline \Gamma_1$ is contained in $\overline V$
which is impossible as the socle of both $\overline V$ and those of $\overline \Gamma^{gen}$
are generic and non isomorphic. Finally $T_1=0$ and $\Gamma^1 \simeq \Gamma^{gen}$.
\end{proof}

\section{Proof of Ihara's lemma}
\label{para-3}

\subsection{Supersingular locus as a zero dimensional Shimura variety}
\label{para-geo}

As explained in the introduction, we follow the strategy of \cite{boyer-ihara} which
consists to transfer the genericity property of Ihara's lemma concerning
$\overline G$ to the genericity of the cohomology of KHT-Shimura varieties.

Let $\overline G$ be a similitude
group as in the introduction such that moreover there exists a prime number
$p_0$ split in $E$ and $v_0^+$ a place of $F^+$ above $p_0$, identified as before
to a place $v_0$ of $F$, such that $\overline B_{v_0}$ is a division algebra:
in particular $v_0 \neq v$.
Consider then, with the usual abuse of notation, $G/\Qm$ such that 
$G(\Am_\Qm^{\oo,v_0}) \simeq \overline G(\Am_\Qm^{\oo,v_0})$ with 
$G(F_{v_0}) \simeq \GL_d(F_{v_0})$ and $G(\Rm)$ of signatures $(1,n-1),(0,n)^r$.
The KHT Shimura variety $\sh_{K,v_0} \rightarrow \spec \OC_{v_0}$ associated to $G$ 
with level $K$, has a Newton stratification of its special fiber with supersingular locus
$$\sh_{K,\bar s_{v_0}}^{=d}=\coprod_{i \in  \ker^1(\Qm,G)} 
\sh_{K,\bar s_{v_0},i}^{=d}.$$
For a equivariant sheaf $\FC_{K,i}$ on 
$\sh_{K^{v}(\oo),\bar s_{v_0},i}^{=d}$ seen as  acompatible system over
$\sh_{K^{v}K_{v},\bar s_{v_0},i}^{=d}$ for $K_{v}$ describing the set of open
compact subgroups of $\GL_d(\OC_{v})$,
its fiber at a compatible system $z_{K^{v}(\oo),i}$
of supersingular point $z_{K^{v}K_{v},i}$, has an action of
$\overline G(\Am_\Qm^{\oo}) \times \GL_d(F_{v})^0$ where $\GL_d(F_{v})^0$ 
is the kernel of the valuation of the determinant so that, 
cf. \cite{boyer-invent2} proposition 5.1.1, 
as a $\GL_d(F_{v})$-module, we have
$$H^0(\sh_{K^{v}(\oo),\bar s_{v_0},i}^{=d},\FC_{K^{v}(\oo),i}) \simeq \Bigl (
\ind_{\overline G(\Qm)}^{\overline G(\Am^{\oo,v}) \times \Zm} z_{K^{v_0}(\oo),i}^* \FC_{K^{v_0}(\oo),i} \Bigr )^{K^{v}},$$
with $\delta \in \overline G(\Qm) \mapsto (\delta^{\oo,v_0},\val \circ \rn( \delta_{v_0})) 
\in \overline G(\Am^{\oo,v_0,v}) \times \Zm$ and where the action of
$g_{v_0} \in \GL_d(F_{v_0})$ is given by those of 
$(g_0^{-\val \det g_{v_0}} g_{v_0},\val \det g_{v_0}) \in \GL_d(F_{v_0})^0 \times \Zm$
where $g_0 \in \GL_d(F_{v_0})$ is any fixed element with $\val \det g_0=1$. 
Moreover, cf. \cite{boyer-invent2} corollaire 5.1.2, if $z_{K^{v_0}(\oo),i}^* \FC_{K^{v_0}(\oo),i}$ is provided with an action
of the kernel $(D_{v_0,d}^\times)^0$ of the valuation of the reduced norm, action compatible with those
of $\overline G(\Qm) \hookrightarrow D_{v_0,d}^\times$, then as a $G(\Am^\oo)$-module, we have
\begin{equation} \label{h0-ss}
H^0(\sh_{K^{v}(\oo),\bar s_{v_0},i}^{=d},\FC_{K^{v}(\oo),i}) \simeq 
\Bigl ( \CC^\oo(\overline G(\Qm) \backslash 
\overline G(\Am^{\oo}),\Lambda)
\otimes_{D_{v_0,d}^\times}  \ind_{(D_{v_0,d}^\times)^0}^{D_{v_0,d}^\times} z_i^* \FC_{\IC,i} \Bigr )^{K^v}
\end{equation}

In particular, cf. lemma 2.3.1 of \cite{boyer-ihara}, 
let $\overline \pi$ be 
an irreducible sub-$\overline \Fm_l$-representation of
$\mathcal C^\oo(\overline G(\Qm) \backslash \overline G(\Am)/K^{v},\overline \Fm_l)_{\mathfrak m}$ for $\mathfrak m$ such that $\overline \rho_{\mathfrak m}$
is irreducible. Write its local component $\bar \pi_{v_0} \simeq \pi_{v_0}[s]_D$ 
with $\pi_{v_0}$
an irreducible cuspidal representation of $\GL_g(F_{v_0})$ with $d=sg$. Then
$(\overline \pi^{v_0})^{K^v}$ is a sub-representation of 
$H^0(\sh^{=d}_{K^{v}(\oo),\bar s_{v_0}},HT(\pi_{v_0}^\vee,s))_{\mathfrak m} 
\otimes_{\overline \Zm_l} \overline \Fm_l$ and,
cf. proposition 2.3.2 of \cite{boyer-ihara}, a 
sub-$\overline \Fm_l$-representation of 
$H^{d-1}(\sh_{K^{v}(\oo),\bar \eta_{v_0}},\overline \Fm_l)_{\mathfrak m}$.
Indeed, cf. theorem \ref{thm-recall},
\begin{itemize}
\item by the main result of \cite{boyer-jep2}, as $l>d \geq 2$ and
$\overline \rho_{\mathfrak m}$ is irreducible, then
$\mathfrak m$ is KHT free so that hypothesis (H1) of
\cite{boyer-ihara} is fulfilled. 

\item Theorem \ref{thm-recall} gives us that the filtration of 
$H^{d-1}(\sh_{K^{v}(\oo),\bar \eta_{v_0}},\overline \Zm_l)_{\mathfrak m}$
induced by the filtration of the nearby cycles at $v_0$, 
is strict.\footnote{In \cite{boyer-ihara} hypothesis (H3) was introduced for this property to be true.}
\end{itemize}
Finally if the analog of Ihara's lemma for 
$H^{d-1}(\sh_{K^{v}(\oo),\bar \eta},\overline \Fm_l)_{\mathfrak m}$
is true for the action of $\GL_d(F_v)$, 
then this is also the case for $\overline G$. We now focus on the genericity
of irreducible sub-$\GL_d(F_v)$-modules of $H^0(\sh_{K^v(\oo),\bar \eta},\overline \Fm_l)_{\mathfrak m}$ using the nearby cycles at the place $v$.

\subsection{Level raising}
\label{sec-hecke}

To a cohomological minimal prime 
ideal $\widetilde{\mathfrak m}$ of $\Tm(K)$, which corresponds to a maximal
ideal of $\Tm(K)[\frac{1}{l}]$, is associated both a
near equivalence class of $\overline \Qm_l$-automorphic representation 
$\Pi_{\widetilde{\mathfrak m}}$ and a Galois representation
$$\rho_{\widetilde{\mathfrak m}}:G_F:=\gal(\bar F/F) \longrightarrow \GL_d(\overline \Qm_l)$$ 
such that the eigenvalues of the Frobenius morphism at an unramified place $w$ are given by the 
Satake parameters of the local component $\Pi_{\widetilde{\mathfrak m},w}$ of  
$\Pi_{\widetilde{\mathfrak m}}$. The semi-simple class $\overline{\rho}_{\mathfrak m}$
of the reduction modulo $l$ of 
$\rho_{\widetilde{\mathfrak m}}$ depends only of the maximal
ideal $\mathfrak m$ of $\Tm^S_K$ containing $\widetilde{\mathfrak m}$.

We now allow infinite level at $v$ and we denote by $\Tm(K^v(\oo))$ the associated 
Hecke algebra. We fix a maximal ideal $\mathfrak m$ 
in $\Tm(K^v(\oo))$ such that the associated Galois representation 
$\rbar_{\fm}: G_F\ra\GL_d(\Fm)$ 
is irreducible.
%

\rem For every minimal prime $\widetilde{\mathfrak m} \subset \mathfrak m$,
note that $\Pi_{\widetilde{\mathfrak m},v}$ looks like $\st_{s_1}(\pi_{v,1}) \times
\cdots \times \st_{s_r}(\pi_{v,r})$ where $\pi_{v,i}$ is an irreducible cuspidal
representation of $\GL_{g_i}(F_v)$ and $s_1 g_1+ \cdots + s_rg_r=d$.

Let $\SC_v(\mathfrak m)$ be the supercuspidal support of the modulo $l$ reduction
of any $\Pi_{\widetilde m,v}$ in the near equivalence class associated to a 
minimal prime ideal $\widetilde{\mathfrak m} \subset \mathfrak m$. 
Recall that $\SC_v(\mathfrak m)$ is a multi-set, i.e. a set with multiplicities 
which only depends on $\mathfrak m$. We decompose it according to the set
of Zelevinsky lines: as we supposed $q_v \equiv 1 \mod l$
then every Zelevinsky line is reduced to a single equivalence class of an irreducible 
(super)cuspidal $\overline \Fm_l$-representations $\varrho$ of some 
$\GL_{g(\varrho)}(F_{v})$ with $1 \leq g(\varrho) \leq d$. 
$$\SC_v(\mathfrak m)=\coprod_{1 \leq g \leq d} 
\coprod_{\varrho \in \cusp_{\overline \Fm_l}(g,v)} 
\SC_{\varrho}(\mathfrak m),$$
where $\cusp_{\overline \Fm_l}(g,v)$ is the set of irreducible 
cuspidal $\overline \Fm_l$-representations of $\GL_g(F_v)$.

\begin{nota} \label{nota-fact}
We denote by
$l_{\varrho}(\mathfrak m)$ the multiplicity of $\SC_{\varrho}(\mathfrak m)$.
\end{nota}

For $\widetilde{\mathfrak m} \subset \mathfrak m$, the local component 
$\Pi_{\widetilde{\mathfrak m},v}$ of $\Pi_{\widetilde{\mathfrak m}}$ can then be written
as a full induced representation 
${\displaystyle \bigtimes_{1 \leq g \leq d}
\bigtimes_{\varrho \in \cusp_{\overline \Fm_l}(g,v)}}
\Pi_{\widetilde{\mathfrak m},\varrho}$ where each 
$\Pi_{\widetilde{\mathfrak m},\varrho}$ is also a full induced representation
$$\Pi_{\widetilde{\mathfrak m},\varrho} \simeq \bigtimes_{i=1}^{r_\varrho(\widetilde{\mathfrak m})} 
\St_{l_{\varrho,i}(\widetilde{\mathfrak m})}(\pi_{v,i})$$
where $r_l(\pi_{v,i}) \simeq \varrho$, $l_{\varrho,1}(\widetilde{\mathfrak m}) \geq \cdots
\geq l_{\varrho,r_\varrho(\widetilde{\mathfrak m})}(\widetilde{\mathfrak m})$
and
$\sum_{i=1}^r l_{\varrho,i}(\widetilde{\mathfrak m})=l_{\varrho}(\mathfrak m)$.
%
%

Suppose now that there exists $1 \leq g \leq d$ and
 $\varrho \in \cusp_{\overline \Fm_l}(g,v)$ such that
$\min_{\widetilde{\mathfrak m} \subset \mathfrak m} \{ r_\varrho(\widetilde{\mathfrak m}) \}
\geq 2$ and let $l_{\varrho,1}:=\max_{\widetilde{\mathfrak m} \subset \mathfrak m} \{
l_{\varrho,1}(\widetilde{\mathfrak m}) \}$ which is then strictly less than $l_\varrho(\mathfrak m)$. 

\noindent \textbf{Fact from  \cite{boyer-compositio} \S 3}:
for an irreducible cuspidal representation $\pi_v$ such
that its modulo $l$ reduction is isomorphic to $\varrho$,
$H^0(\sh_{K^v(\oo),\bar s_v},P(t,\pi_v))_{\mathfrak m} \otimes_{\overline \Zm_l}
\overline \Qm_l$ is the sum of the
contributions of $\Pi_{\widetilde{\mathfrak m}}$ with $\widetilde{\mathfrak m}
\subset \mathfrak m$ such that $\Pi_{\widetilde{\mathfrak m}}$ is of the following
shape: $\st_t(\pi'_v) \times \psi$ where $\pi'_v$ is an unramified twist of $\pi_v$ and $\psi$
is any representation of $\GL_{d-tg}(F_v)$ whose cuspidal support is not linked
to those of $\st_t(\pi'_v)$.

In particular for every $t >l_{\varrho,1}$, $H^0(\sh_{K^v(\oo),\bar s_v},
P(t,\pi_v))_{\mathfrak m} \otimes_{\overline \Zm_l}
\overline \Qm_l$ is zero, so that, as everything is torsion free,
$$H^0(\sh_{K^v(\oo),\bar s_v},\gr^{l_{\varrho,1}(\varrho)-1}_*(
\gr^1_!(\Psi_{\varrho})))_{\mathfrak m} \otimes_{\overline \Zm_l} \overline \Fm_l
\hookrightarrow H^0(\sh_{K^v(\oo),\bar s_v},\Psi_{K^v(\oo),v}))_{\mathfrak m} 
\otimes_{\overline \Zm_l} \overline \Fm_l.$$
Moreover this subspace, as a $\overline \Fm_l$-representation of $\GL_d(F_v)$,
has a subspace of the following shape 
$\st_{l_1(\varrho)}(\varrho) \times \tau$ where the supercuspidal support of
$\tau$ contains $\varrho$. In particular as $q_v \equiv 1 \mod l$ and $l > d$,
this induced representation has both a generic and a non generic subspace.

We can then conclude that for the genericity property to be true for KHT Shimura
varieties, one needs a level
raising property as in proposition 3.3.1 of \cite{boyer-ihara}.
Hopefully such statements exist under some rather mild hypothesis as for example
the following result of \cite{BLGGT}.

\begin{thm} (\cite{BLGGT} theorem 4.4.1)\label{thm-raising}
Let $l > 2(d+1)$ such that $\zeta_l \not \in F$ and all primes of $F^+$ above $l$ split in $F$.
Let $S$ be a finite set of finite places of $F^+$, including all places above $l$, 
such that each place in $S$ splits in $F$. For each place $w \in S$ choose a place 
$\tilde w$ of F lying over $w$. Let $\mu$ be an algebraic character of $G_{F^+}$ and let
$\overline r:G_F \longrightarrow \GL_d(\overline \Fm_l)$ be a continuous representation such that
\begin{itemize}
\item $(\overline r,\overline \mu)$ is a polarized mod $l$ representation unramified outside $S$
which is potentially diagonalizably automorphic;

\item $\overline r_{|G_{F(\xi_l)}}$ is irreducible.
\end{itemize}
For $w \in S$ let $\rho_w:G_{F_{\tilde w}} \longrightarrow \GL_d(\overline \Zm_l)$ be a lift of
$\overline r_{|G_{F_{\tilde w}}}$. If $w~|~l$, assume further that $\rho_w$ is potentially
diagonalizable and that, for all $\tau:F_{\tilde w} \hookrightarrow \overline \Qm_l$, the set
of Hodge-Tate weight $HT_\tau(rho_w)$ consists of $d$ distincts integers.

Then there exists a regular algebraic cuspidal polarized automorphic representation 
$(\pi,\chi)$ of $\GL_d(\Am_F)$ such that
\begin{itemize}
\item $\overline r_l(\pi) \simeq \overline r$;

\item $r_l(\chi)\epsilon_l^{1-d}=\mu$;

\item $\pi$ has level potentially prime to $l$;

\item $\pi$ is unramified outside $S$;

\item for $w \in S$, we have $\rho_w \simeq r_l(\pi)_{|G_{F_{\tilde w}}}$.
\end{itemize}
\end{thm}

We will use this result as follow. Start from an irreducible automorphic representation $\Pi$ of
$G(\Am)$ which appears in $H^{d-1}(\sh_{K,\bar \eta_v},\LC_\xi)_\mathfrak m$
for some irreducible algebraic representation $\xi$ of $G$. By base change it corresponds to a 
a regular algebraic cuspidal polarized automorphic representation $(\pi,\chi)$ as above.
We suppose that the Galois representation $r$ associated to $\pi$ through the Langlands
correspondance, is such that for every $w ~ | ~ l$, $r_{|G_{F_{\tilde w}}}$ is potentially diagonalizable.
We then choose $\rho_v$ a lift of $\overline r_{|G_{F_v}}$ and for others $w \in S$ we keep
$\rho_w:=r_{|G_{F_{\tilde w}}}$: in particular we do not modify what happens above $w ~|~l$.
The previous theorem gives us then $\Pi'$ an irreducible automorphic representation of $G(\Am)$
which, thanks to proposition III.2.1 (6) of \cite{harris-taylor}, appears in 
$H^{d-1}(\sh_{K,\bar \eta_v},\LC_\xi)_\mathfrak m$ with the same $\xi$ we started from.

\subsection{Various filtrations}
\label{para-filtration}

Recall that
$$H^{d-1}(\sh_{K,\bar \eta_v},\overline \Zm_l)_{\mathfrak m} \otimes_{\overline \Zm_l} 
\overline \Fm_l \simeq 
\overline \sigma_{K} \otimes_{\overline R_{K}} \overline \rho_{K},
$$
with $\overline \sigma_{K}/\overline{\mathfrak m} \simeq \overline 
\rho_{\mathfrak m}$. It can be decomposed as follows.
\begin{itemize}
\item It is the direct sum 
$$\bigoplus_{1 \leq g \leq d} \bigoplus_{\varrho \in \cusp_{\overline \Fm_l}(g,v)}
H^0(\sh_{K,\bar s_v},\Psi_{\varrho})_{\mathfrak m} \otimes_{\overline \Zm_l} 
\overline \Fm_l.$$

\item For some fixed $\overline \Fm_l$-supercuspidal representation
$\varrho$ of $\GL_g(F_v)$, the maximal ideal $\mathfrak m$ gives us at the place 
$v$ an element 
$\overline \PC \in \overline \IC$ and in particular numbers $s_\varrho$ for every irreducible
$\overline \Fm_l$-supercuspidal representation $\varrho$.

\item For $t > s_\varrho$, and $\pi_v$ irreducible cuspidal with modulo $l$
reduction isomorphic to $\varrho$, then
$H^0(\sh_{K^v(\oo),\bar s_v},\PC(t,\pi_v))_{\mathfrak m}=(0)$.
\end{itemize} 
We then have a filtration
of $H^0(\sh_{K^v(\oo),\bar s_v},\Psi_\varrho)_{\mathfrak m}$ coming from
the previous filtration of $\Psi_\varrho$:
$$(0)=\Fil^0_\varrho \subset \Fil^1_\varrho \subset \cdots \subset 
\Fil^{\frac{s_\varrho(s_\varrho+1)}{2}}_\varrho,$$
where for $k=s_\varrho+(s_\varrho-1)+\cdots+(s_\varrho-t) - \delta$ with 
$0 \leq t \leq s_\varrho-1$ and 
$0\leq \delta < s_\varrho-t-1$ we have
$$\gr_\varrho^k \simeq H^0(\sh_{K,\bar s_v},\PC(\varrho,s_\varrho-\delta))
(\frac{1-s_\varrho+2t}{2})_{\mathfrak m}.$$
The modulo $l$ monodromy operator $\overline N_\varrho$ induces
an isomorphism
$$\gr_\varrho^{s_\varrho+(s_\varrho-1)+\cdots+(s_\varrho-t)-\delta} \simeq 
\gr_\varrho^{s_\varrho+(s_\varrho-1)+\cdots+(s_\varrho-t-1)-\delta}.$$

We then consider
partitions $(l_1(\varrho) \geq l_2(\varrho),\cdots)$ of $s_\varrho$.
The idea of the proof  is to pass, step by step, from the previous filtration, 
coming from a 
sheaf filtration, to another one
$$(0)=\Fil^{fin}_{-1-s_\varrho} \subset \Fil^{fin}_{-s_\varrho} \subset
\cdots \subset \Fil^{fin}_{-1}=H^0(\sh_{K^v(\oo),\bar s_v},\Psi_\varrho)_{\mathfrak m}$$
where for $r=1,\cdots, s_\varrho$, the graded part $\gr^{fin}_{-r}$ is a lattice of
\begin{equation} \label{eq-fil-srho}
\bigoplus_{\underline s_\varrho \in \Part_{r}(s_\varrho)} 
\bigoplus_{\Pi \in \AC_{K^v,\mathfrak m}(\varrho,
\underline s_\varrho)} (\Pi^\oo)^{K^v} \otimes \sigma_\varrho(\Pi),
\end{equation}
where
\begin{itemize}
\item $\Part_{r}(s_\varrho)$ is the subset of partition
$\underline s_\varrho=(l_1 \geq \cdots \geq l_k)$ of
$s_\varrho$ such that $l_1=r$;

\item for $\underline s_\varrho=(l_1 \geq \cdots \geq l_k)$ 
a partition of $s_\varrho$,
$\AC_{K^v,\mathfrak m}(\varrho,\underline s_\varrho)$ is the set of 
isomorphism classes of irreducible automorphic with $\mathfrak m$-$K^v$ non trivial
invariants and such that 
$$\Pi_v \simeq \st_{l_1} (\pi_{v,1}) \times \cdots \times \st_{l_k}
(\pi_{v,k}) \times \psi$$
where the modulo $l$ reduction of $\pi_{v,1},\cdots,\pi_{v,k}$ are isomorphic to 
$\varrho$ and the supercuspidal support of the modulo $l$ reduction of $\psi$
does not contain $\varrho$.

\item $\sigma_\varrho(\Pi)$ is the $\varrho$-part of $\sigma(\Pi)$ i.e. with the
above notations, 
$$\sigma_\varrho(\Pi) \simeq \Sp_{l_1}(\rho_{v,1}) \oplus
\cdots \oplus \Sp_{l_k}(\rho_{v,k})$$ 
where for $i=1,\cdots, k$, the Galois representation $\rho_{v,i}$ is 
the  contragredient of the representation associated to $\pi_{v,i}$ 
by the local Langlands correspondence.
\end{itemize}
At each step $r$ we prove the $\varrho$-part of
Ihara's lemma for representations 
$\Pi \in \AC_{K^v,\mathfrak m}(\varrho,\underline s_\varrho)$
with $\underline s_\varrho \in \Part_r(s_\varrho)$ in the following sense:
\begin{itemize}
\item for such a irreducible automorphic representation $\Pi$, let
$\Gamma(\Pi)$ be the lattice of $\Pi_v$ induced by the $\OC_L$-cohomology
through
$$(\Pi^{\oo,v})^{K^v} \otimes \Pi_v \otimes \sigma(\Pi) \hookrightarrow
H^0(\sh_{K^v(\oo),\bar s_v},\Psi_\varrho)_{\mathfrak m} \otimes_{\overline \Zm_l}
\overline \Qm_l;$$

\item then $\Gamma(\Pi) \otimes_{\OC_L} \OC_L/\varpi_L \OC_L$, as a
$\Fm_L$-representation of $\GL_d(F_v)$, has a $\varrho$-generic socle, i.e
representations of the form $\st_{l_\varrho(\mathfrak m)} (\varrho) \times \psi$
where $\varrho$ does not belong to the supercuspidal support of $\psi$.
\end{itemize}

\subsubsection{Case where $s_\varrho=3$}
\label{para-s3}

\begin{figure}[t]
\input{figure-filtration.tex}
\caption{Three filtrations of $H^0(\sh_{K^v(\oo),\bar s_v},\Psi_\varrho)_{\mathfrak m}$ when $s_\varrho=3$}
\label{fig-s3}
\end{figure}

Before considering the general case, we first want to explain the case where
$s_\varrho=3$ and $d=g s_\varrho$: 
we advise the reader to follow the details of the proof with the help
of figure \ref{fig-s3}. We first describe it.
\begin{itemize}
\item It illustrates three filtrations separated by the symbols $\leadsto$; the
graded parts are represented by circles corresponding to the
cohomology of Harris-Taylor perverse sheaves: compare 
the filtration on the left with those in the figure \ref{fig-psi-filtration}.
In particular each circle can be viewed as some lattice of
$$\bigoplus_{\Pi \in \AC_{K^v,\mathfrak m}(\varrho,\underline s_\varrho)} 
(\Pi^\oo)^{K^v} \otimes \sigma'_\varrho(\Pi)$$
for some partition $\underline s_\varrho$ of $s_\varrho=3$ with
the following precisions.
\begin{itemize}
\item When the circle is filled with diagonal lines with
slope $1$, it corresponds to $\underline s_\varrho=(1,1,1)$ or to the contribution 
of $1$ in the partition $\underline s_\varrho=(2,1)$ in the sense that
$\sigma'_\varrho(\Pi)$ is the Galois representation 
$\rho_l(\pi_{v,2})^\vee$  associated to $\pi_{v,2}$
in $\Pi_v \simeq \st_2(\pi_{v,1}) \times \pi_{v,2}$, i.e. the 
contragredient of the Galois representation associated to $\pi_{v,2}$
by the local Langlands correspondence.

\item When the circle is filled with diagonal lines with slope $\pm 1$, it
corresponds to the contribution of $2$ in the partition $\underline s_\varrho=(2,1)$,
in the sense that $\sigma'_\varrho(\Pi)$ is 
the Galois representation associated either to $\pi_{v,1}\{1/2 \}$
or $\pi_{v,1} \{ -1/2 \}$ in $\Pi_v \simeq \st_2(\pi_{v,1}) \times \pi_{v,2}$:
the sign corresponds to the weight of the perverse Harris-Taylor sheaf 
whose cohomology gives the circle.

\item When the circle is empty, it then corresponds to $\underline s_\varrho=(3)$:
then $\sigma'_\varrho(\Pi)$ correspond to the Galois representation associated
to $\pi_v \{ k \}$ in $\Pi_v \simeq \st_3(\pi_v)$ where $k \in \{ -1,0,1 \}$
corresponds to the weight of the perverse Harris-Taylor sheaf whose cohomology
gives the circle.
\end{itemize}
%
%
\item For each of these three filtrations illustrated in figure \ref{fig-s3}, 
subspaces appears from bottom to top.

\item Arrows correspond to the nilpotent monodromy operator and we explicit
if its modulo $l$ reduction $\overline N$ is zero or non zero, see after for more
details.

\item In the second filtration we gather the contribution of 
$\AC_{K^v,\mathfrak m}(\varrho,(3))$ so that in the large ellipse we obtain as a
quotient the contribution of
$\bigoplus_{\Pi \in \AC_{K^v,\mathfrak m}(\varrho,(3))} 
(\Pi^\oo)^{K^v} \otimes \sigma_\varrho(\Pi)$
for all of $\sigma_\varrho(\Pi)$. The indices $1$ above the circles means that
the lattices are modified which is materialize in the bottom of the figure with
the precision that cokernels are killed by $\varpi_L$ with generic cosocle as
$\Fm_L$-representations.

\item In the last filtration on the right of the figure, we gather in the large circle
the contribution of $\AC_{K^v,\mathfrak m}(\varrho,(2,1)$ with 
another modifications of the lattices: note that one have to separate
the contributions of $(2,1)$ with $(1,1,1)$ in the circle filled with lines
with slope $1$.
\end{itemize}

\noindent \textbf{The automorphic filtration}\\
For each of these circles viewed as the $\overline \Zm_l$-cohomology group
of Harris-Taylor perverse sheaves, we will also consider filtrations as follows.
\begin{itemize}
\item For $n=1$ we fix any numbering of elements in
$\AC_{K^v,\bar s_v}(\varrho,\underline s_\varrho)$ which appears in
$\AC_{K^v(1),\bar s_v}(\varrho,\underline s_\varrho)$, i.e. having
non trivial vectors invariants under $K_v(1)$. We then obtain
successive subspaces of the $\overline \Qm_l$-contribution of this circle:
the intersection with the $\overline \Zm_l$ cohomology gives us successive
lattices $\Gamma(\Pi)$
of $(\Pi^\oo)^{K^v} \otimes \sigma'_\varrho(\Pi)$ for theses elements $\Pi \in
\AC_{K^v,\bar s_v}(\varrho,\underline s_\varrho)$ having non zero vectors
invariant under $K_v(1)$.

\item We then take $n=2$,
and consider elements of $\AC_{K^v(2),\bar s_v}(\varrho,\underline s_\varrho)$
which do not appears with $\AC_{K^v(1),\bar s_v}(\varrho,\underline s_\varrho)$:
we fix a numbering of them and obtain other lattices of their contribution.

\item We keep on this construction for any $n$ and speak about 
\emph{the automorphic filtration} of this circle. 
\end{itemize}

\rem Note that the lattice $\Gamma(\Pi)$ might depends on
the ordering of any of the 
$\AC_{K^v(n),\mathfrak m}(\varrho,\underline s_\varrho)$. To deal with
finite number of graded parts, in the following we will argue for some fixed $n$: this also allows
us to consider have everything defined over some finite extension $L$ of $\Qm_l$ as before.

A- We first explain how to pass from the first filtration to the second one. 

\noindent \textit{Step 1: the exchange is non trivial} \\
Consider as illustrated in the figure \ref{fig-s1}, the three first quotients of the first filtration, which correspond to the three circles in the bottom on the left part of
the figure \ref{fig-s3}. We then 
exchange the first two as explained in figure \ref{fig-exchange}.
\begin{figure}[t]
\begin{center}
\input{figure-filtration2.tex}
\caption{First step}
\label{fig-s1}
\end{center}
\end{figure}
We want first to explain why this exchange is non trivial, i.e. the first extension
is non split or equivalently $T$ is non zero. For this we examine more
precisely $\overline N$ on this space denoted by $V$.

a) On the left side of the figure \ref{fig-s1} 
we know that the arrow is a isomorphism: we 
just write $\overline N \neq 0$ in the figure. We can then read 
the dimension of the image of $\overline N$ from what happens over
$\overline \Qm_l$. More precisely,
over $\overline \Qm_l$, the rank of the monodromy operator is equal to
the cardinal of $\AC_{K^v,\mathfrak m}(\varrho,(3))$ meaning for each finite level
$K_v$, $\Pi \in \AC_{K^v,\mathfrak m}(\varrho,(3))$ contributes to 
$g \dim_{\overline \Qm_l} (\Pi^\oo)^{K^vK_v}$.

\rem It is also possible to argue with $K^v(\oo)$, i.e. with infinite dimension at $v$: to be able
to count something with finite dimension, one can look at the contribution of a $K_v$-type for example.

b) We then look at the right side of the figure \ref{fig-s1}. Looking at the last two
graded parts, over $\overline \Fm_l$, as explained in the introduction, 
the modulo $l$ reduction 
of the nilpotent monodromy operator is zero on any $\sigma_\varrho(\Pi)$ so when
considering a lattice of a direct sum of $\sigma_\varrho(\Pi)$, the rank of 
$\overline N$ has to be strictly less than those of $N$ over $\overline \Qm_l$.
In the above picture we simplify this observation by simply writing $\overline N=0$.

So the ellipse of the figure \ref{fig-s1} cannot be split, i.e.  
the exchange is non trivial and  the $T$ appearing in figure \ref{fig-s1} is non zero.
\begin{itemize}
\item As a quotient of the empty circle on the right of the last line, corresponding
to the contribution of $\AC_{K^v,\mathfrak m}(\varrho,(3))$,
we deduce that $T$ has to be extensions of $\overline \Fm_l$-generic
representations of $\GL_d(F_v)$. To see this, consider an automorphic filtration
of the empty circle and view $T$ as a the limit of $T_n$ where $T_n$
corresponds to the exchange of the automorphic representation for $K^v(n)$.
Then every $T_n$ is a extension of $\varrho$-generic representations.

\rem in view of the next steps we just
remember that $T$ has a $\varrho$-generic socle;

\item We then look at $T$ as a quotient of left part of the last line.
Consider an automorphic filtration of the circle filled with diagonal lines
corresponding to automorphic representations in 
$\AC_{K^v,\mathfrak m}(\varrho,(2,1))$. Let denote by $\Gamma(\Pi)$
be the lattice given by the initial circle and $\Gamma^1(\Pi)$ after the exchange,
i.e. for the circle indexed with a $1$. If $\Gamma(\Pi)$ is modified, for the first time
it becomes, cf. lemma \ref{lem-lattice-gen}, 
$\Gamma(\Pi)^{gen}$ the lattice with a $\varrho$-generic socle. Then it is not
possible to modify $\Gamma(\Pi)^{gen}$ to obtain a new lattice
$\Gamma(\Pi)^2$ such that $\Gamma(\Pi)^{gen}/\Gamma(\Pi)^2$ is $\varrho$-generic:
indeed the modulo $l$ reduction of $\Gamma(\Pi)^{gen}$ contains an unique
$\varrho$-generic constituant and it has to be in the image of
$$\Gamma(\Pi)^2 \otimes_{\overline \Zm_l} \overline \Fm_l \longrightarrow
\Gamma(\Pi)^{gen} \otimes_{\overline \Zm_l} \overline \Fm_l.$$
\end{itemize}
Finally after the exchange, the subspace of the new filtration illustrated
by a circle filled with lines and indexed by $1$ in the figure \ref{fig-s1}, 
provided with its automorphic filtration, is such that the lattices are
either the same as the initial one or it is the one with a $\varrho$-generic socle.
%

There is then at least one 
$\Pi_0 \in \AC_{K^v,\mathfrak m}(\varrho,(2,1))$ such that its lattice is 
$\Gamma(\Pi_0)^{gen}$ and is not the one before the exchange. Take
$\Pi_0$ as the first one. Note that modulo $l$ we have the following
commutative diagram
$$\xymatrix{
\st_3(\varrho) \ar@{^{(}->}[r] & \overline X \ar@{->>}[r] \ar@{=}[d]& 
\overline \Gamma(\Pi_0) \\
\overline \Gamma^{gen}(\Pi_0) \ar@{^{(}->}[r] & \overline X \ar@{->>}[r] &
\st_3(\varrho)
}$$
where before the exchange, $\st_3(\varrho)=\ker( \overline X \twoheadrightarrow 
\overline \Gamma(\Pi_0))$ belongs to the image of $\overline N_\varrho$. Then as 
$$\st_3(\varrho)=\ker (\overline \Gamma^{gen}(\Pi_0) 
\longrightarrow \overline \Gamma(\Pi_0))$$
we then deduce that, after the exchange, modulo $l$, 
the subspace $\st_3(\varrho)$ of 
$\overline \Gamma^{gen}(\Pi_0)$ maps to 
$\st_3(\varrho)=\ker( \overline X \twoheadrightarrow 
\overline \Gamma(\Pi_0)$ and then belongs to the image of $\overline N_\varrho$.

\medskip

\noindent \textit{Step 2: the subspace lattice of $\Pi_0$ has also a $\varrho$-generic socle} \\
Maybe $\Pi_0$ does not appears first in the automorphic filtration. Let denote
by $k$ the index of the graded part corresponding to $\Pi_0$ in the automorphic 
filtration of the circle
on the bottom in the right side of the figure \ref{fig-s1}, i.e. the one indexed by $1$
and filled with diagonal lines: we denote by $W$ the cohomology represented
by this circle.

We may moreover suppose that $k$
is minimal among all modified lattices and we want to prove that
$k=1$. We argue by absurdity by assuming $k \geq 2$. Let then 
denote by $\Gamma(\Pi)$ the lattice which appears just
before $\Gamma(\Pi_0)^{gen}$, i.e. the graded part of index $k-1$. We then
exchange these two graded parts so that before the exchange the successive
graded parts are
$$ W=\Gamma_-, \quad \Gamma(\Pi), \quad \Gamma(\Pi_0)^{gen},
\quad \Gamma_+$$
and after there are
$$ W=\Gamma_-, \quad \Gamma(\Pi_0), \quad \Gamma(\Pi)', \quad \Gamma_+$$
and
$$\xymatrix{
\Gamma(\Pi_0) \ar@{^{(}->}[r] & \Gamma(\Pi_0)^{gen} \ar@{->>}[r] & T 
\ar@{=}[d] \\
\Gamma(\Pi) \ar@{^{(}->}[r] & \Gamma(\Pi)' \ar@{->>}[r] & T,
}$$
where $T$ is $l$-torsion, non zero and $\varrho$-generic. As $k$ is minimal, then
$\Gamma_-, \Gamma(\Pi_0)$ is induced by the circle filled with
diagonal lines in the left part of the figure \ref{fig-s1}, i.e. before any exchange,
so that $T$ is non zero as $\Pi_0$ was chosen so that $\Gamma(\Pi_0)^{gen}$
is not the same lattice as before the exchange.
The modulo $l$ reduction of these two filtrations gives a filtration
of the modulo $l$ reduction $\overline W$ of $W$ with successive graded parts
$$\overline W=\overline \Gamma_-,\quad  \overline \Gamma(\Pi_0), \quad 
\overline \Gamma(\Pi)', \quad \overline \Gamma_+ \quad = \quad
\overline \Gamma_-, \quad \overline \Gamma(\Pi), \quad  
\overline \Gamma(\Pi_0)^{gen}, \quad \overline \Gamma_+.$$
We then focus on the $\varrho$-generic constituant of $\overline \Gamma(\Pi_0)$:
\begin{itemize}
\item it is not in the image of $\overline N$,

\item but viewed in the second filtration it is identified with the $\varrho$-generic constituant
of $\overline \Gamma(\Pi_0)^{gen}$ which is in the image of $\overline N$:
contradiction.
\end{itemize}

\rem Note that here $\varrho$-generic is the same as generic.
%

\medskip

\noindent \textit{Step 3: every subspace lattice has a $\varrho$-generic socle} \\
Consider now $\Pi \in \AC_{K^v,\mathfrak m}(\varrho,(2,1))$ and
the two previous lattices of $\Pi_v$ with
\begin{equation} \label{eq-secGamma}
0 \rightarrow \Gamma(\Pi)^{gen} \longrightarrow \Gamma(\Pi) \longrightarrow
\st_3(\varrho) \rightarrow 0.
\end{equation}
We want to distinguish these two lattices through their restriction to 
$K_v=\GL_d(\OC_v)$.
Note that $\st_3(\varrho)_{|K_v}$ contains an unique
$K_v$-type $\overline \sigma_{\max}$ which is moreover maximal. 
From theorem \ref{thm-sigmamax} $\overline \sigma_{\max}$ appears
with multiplicity one in the modulo $l$ reduction of $\Gamma(\Pi)_{|K_v}$.
Write $\Pi_{v} \simeq \st_2(\pi_{v,1}) \times \pi_{v,2}$. With the notations of
\S \ref{sub:WD:types}, to these two cuspidal representations $\pi_{v,1}$
and $\pi_{v,2}$, is associated a SZ-datum and in particular cuspidal representations
$\nu_1$ and $\nu_2$ of a finite linear group.
\\
(a) We first suppose that $\nu_{2}$ and  $\nu_{1}$ are not isomorphic.
Then $\Pi_v$ contains an unique $K_v$-type $\sigma_L$ whose modulo $l$ reduction
contains two $K_v$-types $\overline \sigma_{\max}$ and another one 
$\overline \sigma_{\min}$.
We denote by $\sigma$ (resp. $\sigma^{gen}$) the lattice of
$\sigma_{L}$ obtained from $\Gamma(\Pi)$
(resp. $\Gamma(\Pi)^{gen}$) through $\sigma_{L}
\hookrightarrow (\Pi_{v})_{|K_v}.$
By the multiplicity one property, the short exact sequence (\ref{eq-secGamma})
gives
$$0 \rightarrow \sigma^{gen} \longrightarrow \sigma \longrightarrow
T \rightarrow 0,$$
where $T$ is non zero as it contains at least $\overline \sigma_{\max}$ as a sub-quotient.
Consider any $\Pi \in \AC_{K^v,\mathfrak m}(\varrho,(2,1))$, we know that its subspace
lattice induced by the cohomology is either $\Gamma(\Pi)$ or 
$\Gamma(\Pi)^{gen}$. Proposition \ref{prop-BLC} tells us that it has to be $\Gamma(\Pi)^{gen}$.
%

\smallskip

\noindent
(b) Consider now the case where the two cuspidal representations
$\nu_1$ and $\nu_2$ are isomorphic. We consider another cuspidal
representation $\nu'_2$ isomorphic to $\nu_2$ modulo $\varpi_L$, which is
not isomorphic to $\nu_1$ and we denote by
$\pi'_{v,2}$ the corresponding cuspidal representation constructed as in
\S \ref{sub:WD:types} using the same data for $\pi_{v,2}$ and replacing
$\nu_2$ by $\nu'_2$.

By \cite{gee-annalen} theorem 5.1.5), cf. theorem \ref{thm-raising}, there exists
another system of Hecke eigenvalues $\lambda'$ so that
$\lambda' \equiv \lambda \mod \varpi_L$ and the associated local
component $\Pi'_v$ is isomorphic to $\st_2(\pi_{v,1}) \times \pi'_{v,2}$:
note that $\lambda'$ now verifies the
property as in the previous case (a) so that $\Pi'_v$
has only one $K_v$-type $\sigma'_L$.
This type is constructed as explained in \S \ref{sub:WD:types} 
so that we can copy the
construction for $\Pi_v$ to obtain a representation $\sigma_L$ which is no more
a $K_v$-type. 

We then consider the induced lattices $\Gamma_{ind}$ (resp. $\Gamma'_{ind}$)
of $\Pi_v$ (resp. $\Pi'_v$) and the associated lattices $\sigma_{ind}$ 
(resp. $\sigma'_{ind}$) of $\sigma_L$ (resp. $\sigma'_L$). Note that 
modulo $\varpi_L$, we have 
$$\overline \Gamma_{ind} \simeq \overline \Gamma'_{ind} \qquad \hbox{and}
\qquad \overline \sigma_{ind} \simeq \overline \sigma'_{ind}.$$ 
Let then denote by $\Gamma_{gen}$ and $\Gamma'_{gen}$ the lattices
such the socle of 
$$\Gamma_{gen}/\varpi_L \Gamma_{gen} \simeq \Gamma'_{gen}/\varpi_L \Gamma'_{gen}$$ 
is generic. We then also introduce the pullbacks
$$\xymatrix{
\Gamma_{gen} \ar@{^{(}->}[r] & \Gamma_{ind} & \Gamma'_{gen} \ar@{^{(}->}[r] &
\Gamma'_{ind} \\
\sigma_{gen} \ar@{^{(}-->}[r] \ar@{^{(}-->}[u] & \sigma_{ind} \ar@{^{(}->}[u] &
\sigma'_{gen} \ar@{^{(}-->}[r] \ar@{^{(}-->}[u] & \sigma'_{ind} \ar@{^{(}->}[u]
}$$
Note that we have $\sigma_{gen}/\varpi_L \sigma_{gen} \simeq 
\sigma'_{gen}/\varpi_L \sigma'_{gen}$. Indeed let denote by 
$\widetilde \sigma_{gen}$ the subspace of $\sigma_{ind}$ such that
$$\sigma_{ind}/\widetilde \sigma_{gen} \simeq \sigma'_{ind}/\sigma'_{gen}.$$
Note that $\overline \sigma_{\max}$ is a subquotient of 
$\Gamma'_{ind}/\Gamma'_{gen}$ and as it is contains once in 
$\Gamma_{ind}/\varpi_L \Gamma_{ind}$, it has to be a subquotient of 
$\sigma'_{ind}/\sigma'_{gen}$. Any $\Gamma_{gen} \hookrightarrow \Gamma_+ \hookrightarrow \Gamma_{ind}$ is obtained by pullback
$$\xymatrix{
\Gamma_{gen} \ar@{=}[d] \ar@{^{(}->}[r] & \Gamma_{ind} \ar@{->>}[r] & T=(\overline \sigma_{\max} // \cdots) \\
\Gamma_{ge} \ar@{^{(}->}[r] & \Gamma_+ \ar@{^{(}-->}[u] \ar@{-->}[r] & T' \ar@{^{(}->}[u]
}$$
but as $\overline \sigma_{\max}$ belongs to the socle of $T$, the cokernel $\Gamma_{ind}/\Gamma_+$
does not contain $\overline \sigma_{\max}$ and so
$\widetilde \sigma_{gen}$ can not be obtained as a pullback using
$\Gamma_+ \supsetneq \Gamma_{gen}$. We then deduce that
$\widetilde \sigma_{gen} \hookrightarrow \sigma_{gen}$. Consider then
$$\sigma'_{gen} \hookrightarrow \widetilde \sigma'_{ind} \hookrightarrow 
\sigma'_{ind}$$
such that $\sigma'_{ind}/\widetilde \sigma'_{ind} \simeq 
\sigma_{gen}/\widetilde \sigma_{gen}$. As
$\overline \sigma_{\max}$ is a subquotient of these cokernel then, arguing as before we obtain
$\widetilde \sigma'_{ind} \hookrightarrow \Gamma'_{gen}$ and so
$\widetilde \sigma_{gen}= \sigma_{gen}$.

Case (a) gives us that the natural morphism 
$$M(\sigma^v\sigma'_{gen})/(\lambda',\varpi_L) \longrightarrow
M(\sigma^v\sigma'_{ind})/(\lambda',\varpi_L)$$
induced by $\iota$ is zero so that it is the same for
\begin{equation} \label{eq-casb}
M(\sigma^v\sigma_{gen})/(\lambda,\varpi_L) \longrightarrow
M(\sigma^v\sigma_{ind})/(\lambda,\varpi_L).
\end{equation}
But if the lattice induced by $\lambda$ were $\Gamma_{ind}$, then
$$\hom_{K_v}(\sigma_{ind}/\varpi_L,\sigma_{ind}/\varpi_L) \longrightarrow
\hom_{K_v}(\sigma_{gen}/\varpi_L,\sigma_{ind}/\varpi_L)$$
is non zero as the image of identity is $\overline \iota \neq 0$ and so
(\ref{eq-casb}) is non zero.

%

\rem In case (b) if we do not modify one of the cuspidal, the problem is that the $K_v$-type
whose modulo $l$ reduction contains $\overline \sigma_{\max}$, remains irreducible
modulo $\varpi_L$ and so have, up to homothety, only one stable lattice. We are then not able
to detect the right $\GL_{3s_\varrho}(F_v)$-lattice. In the general case we
might also have to modify one of the cuspidal representation in the cuspidal support
to avoid such a situation.

\rem The case $s_\varrho=3$ is more easy than the general one because we
are concerned with only one partition $(2,1)$ so that when we consider any
other automorphic contribution $\Pi$ it shares the same partition $(2,1)$
with $\Pi_0$ and the modulo $l$ reduction of the types of $\Pi$ and $\Pi_0$
coincide. In the general the case we will need to manage the existence
of various partitions contributing to the considered circle.

\medskip

\noindent \textit{Step 4: Ihara's lemma for $\AC_{K^v,\mathfrak m}(\varrho,(2,1))$} \\
By now concerning the new filtration, those
illustrated on the right side of the figure \ref{fig-s1}, we denote again by $W$ 
the contribution of the
circle filled with diagonal lines and indexed by $1$:
$$0 \rightarrow W \longrightarrow V \longrightarrow W' \rightarrow 0.$$
We focus on a automorphic filtration of $W$ where the 
graded parts are lattices of the contribution of some 
$\Pi \in \AC_{K^v,\mathfrak m}(\varrho,(2,1)$: we have seen that
the first graded part, i.e. the subspace one, is such that its modulo $l$ reduction is
a non trivial extension of the irreducible non $\varrho$-generic constituant by the 
$\varrho$-generic one.

Consider any graded part $\Gamma(\Pi)$ and suppose that its modulo $l$
reduction has a non $\varrho$-generic subspace: we state that it is not a subspace of $W$.
Indeed we have seen that subspace lattice of the contribution of
$\Pi$ is $\Gamma(\Pi)^{gen}$ and
$$0 \rightarrow \Gamma(\Pi)^{gen} \longrightarrow \Gamma(\Pi) \longrightarrow T
\rightarrow 0,$$
where $T$ is $l$-torsion and non zero as by hypothesis $\Gamma(\Pi)$ is
supposed to be non isomorphic to $\Gamma(\Pi)^{gen}$. 
Tensoring with $\overline \Fm_l$, we then obtain that the non $\varrho$-generic
constituant of $\Gamma(\Pi) \otimes_{\overline \Zm_l} \overline \Fm_l$, which is
supposed to be a subspace of $W$ and which belongs to the
image of $\Gamma(\Pi)^{gen} \otimes_{\overline \Zm_l} \overline \Fm_l$.
But we know that this non $\varrho$-generic constituant is not a subspace of 
$\Gamma(\Pi)^{gen} \otimes_{\overline \Zm_l} \overline \Fm_l$: contradiction.

\medskip

\noindent \textit{Step 5: final exchanges to arrive at the second filtration} \\ 
We now exchange $W'$ with the two circles of the
first filtration of the figure \ref{fig-s3}: those filled with diagonal lines with
slope $1$ and those with slopes $\pm 1$. We then arrive at the second filtration
of this figure. To be in position to pass from the second filtration to the third one,
we need to give more informations on the arrow $\overline N \neq 0$
of the second filtration of the figure \ref{fig-s3}. 

Let consider the subspace represented by the first three graded parts, i.e.
we remove the quotient given by the contribution of
$\AC_{K^v,\mathfrak m}(\varrho,(3))$. Before all the exchanges,
the image of the considered $\overline N$ was equal to the 
contribution of the circle pointed by this arrow. In particular all the $(\varrho,2)$-small subquotients
were in the image of $\overline N$. But as the exchanges made are only related to
$\st_3(\varrho)$, we then deduce that after all the exchanges, all the $(\varrho,2)$-small 
subquotients of the circle pointed by the arrow are still in the image of $\overline N$.

\medskip

B- We now explain how to pass from the second filtration to the third one.

The situation is very similar, we just focus on the differences. 
Let denote again by $V$ the quotient gathering the
three first graded parts of this second filtration. Note that
\begin{itemize}
\item  the modulo $l$ reduction of the first graded part, i.e. the circle filled with
lines with slopes $\pm 1$, has a $\varrho$-generic socle; 

\item the rank of $\overline N$ can be directly compared to its $\overline \Qm_l$-version as explained above.

\end{itemize}

\noindent \textit{Step 1: the exchange is non trivial} \\
As step 1 of A, the exchange between the first two graded
parts is necessarly non trivial because after the exchange the
$\overline \Fm_l$-contribution of $\overline N$, concerning $(\varrho,2)$-small subquotients,
in the first two graded parts
is strictly less than its $\overline \Qm_l$-contribution.

\rem The easiest way maybe to count things, is to look at the contributions of $K_v$-type for
the partition $(2,1)$: note that $\st_3(\pi_v)$ never contributes to these contributions.

Concerning the torsion module $T$ produced by the various exchanges,
\begin{itemize}
\item[(i)] on one side as a quotient of the contribution of 
$\AC_{K^v,\mathfrak m}(\varrho,(2,1))$, it follows from A step 3, one of them has a 
$\varrho$-generic socle but compared to A we don't know whether it is always
a non trivial extension of $\st_3(\varrho)$ by $LT_\varrho(1,1)$ or it is
$\st_3(\varrho)$ alone with $LT_\varrho(1,1)$ appearing alone after;

\item[(ii)] on the other side, note that modulo $\varpi_L$, the representation 
$\pi_{v,1} \times \pi_{v,2} \times \pi_{v,3}$ has a subquotient which is not a 
constituant of any $\st_2(\pi_{v,1}) \times \pi_{v,2}$, where all the $\pi_{v,i}$ 
are supposed to be isomorphic to $\varrho$ modulo $\varpi_L$. As before we then deduce
that $T$ is killed by $\varpi_L$.
\end{itemize}
We then deduce that after the exchange, concerning the circle filled with
line of slope 1 and indexed with a $2$, its automorphic filtration separating the contributions
of $\Pi \in \AC_{K^v,\mathfrak m}(\varrho,(1,1,1))$, has a graded part
$\Gamma(\Pi_0)$ such that its modulo $\varpi_L$ reduction, as a 
$\Fm_L$-representation of $\GL_d(F_v)$, has a $\varrho$-generic socle.

\medskip

\noindent \textit{Step 2-3-4: lattices for elements of 
$\AC_{K^v,\mathfrak m}(\varrho,(1,1,1))$}  \\
The arguments of A-step 2 applies and gives us that the lattice associated to 
$\Pi_0$ viewed as a subspace of $V$ has also a $\varrho$-generic socle, cf. lemma \ref{lem-lattice-gen}.
Repeating the arguments of A-step 3, thanks to proposition \ref{prop-BLC}, 
we then deduce that whatever is 
$\Pi \in \AC_{K^v,\mathfrak m}(\varrho,(1,1,1))$, its lattice induced by $V$
as a subspace, has also a $\varrho$-generic socle. In particular in (i) above, we deduce
that all $l$-torsion modules appearing in the exchanges are necessary
a non trivial extension of $LT_\varrho(1,1)$ by $\st_3(\varrho)$ and the 
subspace lattice of any $\Pi \in \AC_{K^v,\mathfrak m}(\varrho,(1,1,1))$,
has a socle filtration with at least three graded parts where the
two first ones are $\st_3(\varrho)$ and $LT_\varrho(1,1)$.

Finally the previous arguments of A-step 4 show that no non $\varrho$-generic subquotient
of the contributions of elements of $\AC_{K^v,\mathfrak m}(\varrho,(1,1,1))$
can give an irreducible subspace of the modulo $l$ cohomology.

We then arrive at our final filtration where we succeeded to gather the contributions
of the various elements of $\AC_{K^v,\mathfrak m}(\varrho,\underline s_\varrho)$
such that no non $\varrho$-generic irreducible constituant of the various graded parts
could be a subspace of the modulo $l$ cohomology: Ihara's lemma is then
proved in this case.

\subsubsection{The general case}
\label{para-generalcase}

We now consider the general case and we follow closely the arguments explained
when $s_\varrho=3$. We first have to define the different filtrations and then
explain how to pass from one to its following.

For some fixed $r$ between $1$ and $s_\varrho$ we suppose that we have already proved
the following property:
\begin{itemize}
\item consider an automorphic representation $\Pi$ 
which contributes to $H^0(\sh_{K,\bar s_v},\Psi_\varrho)_{\mathfrak m} 
\otimes_{\overline \Zm_l} \overline \Qm_l$ with
$$\Pi_v \simeq \st_{l_1(\varrho)} \pi_{v,1} \times \cdots \st_{l_k(\varrho)} \pi_{v,k}
\times \psi$$ 
where $(l_1(\varrho) \geq l_2(\varrho),\cdots,l_k(\varrho))$ is a partition 
of $s_\varrho$ with $r_l(\pi_{v,i}) \simeq \varrho$
for $i=1,\cdots,k$ and the supercuspidal support of the modulo $l$ reduction
of $\psi$ does not contain $\varrho$;

\item suppose that $l_1(\varrho) \geq r$, then the $\GL_d(F_v)$-lattice induced by 
$$\Pi_v \otimes \rho_l(\pi_v)^\vee \hookrightarrow H^0(\sh_{K^v(\oo),\bar s_v},\Psi_{\varrho})_{\mathfrak m}$$
is such that the only irreducible subspace of its modulo $l$ reduction is
$\st_{s_\varrho}(\varrho) \times \psi$ where $\psi$ is some irreducible 
$\overline \Fm_l$-representation whose supercuspidal support does not contain
$\varrho$. 
\end{itemize}

For the same fixed $r$ between $1$ and $s_\varrho$ we consider the following filtration:
$$(0)=\Fil_\varrho^{-r(r+1)/2}(r) \subset \cdots \subset \Fil_\varrho^{0}(r) 
\subset \Fil_\varrho^1(r) =
H^0(\sh_{K^v(\oo),\bar s_v},\Psi_{\varrho})_{\mathfrak m}$$ 
such that 
\begin{itemize}
\item $\gr_\varrho^1(r) \otimes_{\overline \Zm_l} \overline \Qm_l$ is the direct sum
$$\bigoplus_{\underline s_\varrho > (r)}
\bigoplus_{\Pi \in \AC_{K^v,\mathfrak m}(\varrho,\underline s_\varrho)} 
(\Pi^\oo)^{K^v(\oo)} \otimes \sigma_\varrho(\Pi),$$
where $(l_1 \geq l_2 \geq \cdots) > (r)$ means $l_1 > r$;

\item for $k=1+2+\cdots+(t+1)-\delta$ with $0 \leq t \leq r-1$ and 
$0 \leq \delta \leq t$, the graded part $\gr^{-r(r+1)/2+k}_\varrho(r)$
is a lattice $\Gamma(t,\delta)^{init}$ 
of\footnote{The contribution of the automorphic representations
$\Pi \in \AC_{K^v,\mathfrak m}(\varrho,\underline s_\varrho)$
for $\underline s_\varrho=(l_1 \geq \cdots)$ with $l_1 >r$ was already put
in $\gr^0_\varrho(r)$ while the others which contributes to the cohomology of
$\PC(t+1+\delta)(\frac{t-\delta}{2})$ remains in this graded parts.}
\begin{multline*}
\bigoplus_{\underline s_\varrho=(r \geq \cdots)}
\bigoplus_{\Pi \in \AC_{K^v,\mathfrak m}(\varrho,\underline s_\varrho)} 
(\Pi^\oo)^{K^v(\oo)} \otimes \rho_l(\Pi_v)(\frac{r-1-2t+\delta}{2}) \\
\hookrightarrow
\bigoplus_{\pi_v/ r_l(\pi_v) \simeq \varrho}
H^0(\sh_{K^v(\oo),\bar s_v},\PC(r-\delta)(\frac{r-1-2t+\delta}{2}))_{\mathfrak m} 
\otimes_{\overline \Zm_l} \overline \Qm_l.
\end{multline*}
Moreover with the same notations for $0 \leq \delta \in \{ t-1,t \}$,
if we denote by $\Gamma(t,\delta)^0$ the lattice induced by 
$$\bigoplus_{\pi_v/ r_l(\pi_v) \simeq \varrho} H^0(\sh_{K^v(\oo),\bar s_v},
\PC(t+1+\delta,\pi_v)(\frac{t-\delta}{2}))_{\mathfrak m}$$ 
on the previous subspace, then we have 
$$\Gamma(t,\delta)^{init} \hookrightarrow \Gamma(t,\delta)^{0}$$
where the cokernel is $l$-torsion.

\rem In the case $s_\varrho=3$, 
\begin{itemize}
\item in the second filtration of the figure \ref{fig-s3},
$\gr^1_\varrho(2)$ is the first ellipse gathering the contributions of
$\AC_{K^v,\mathfrak m}(\varrho,(3))$. The image of circle pointed by 
(resp. at the origin of) the vector corresponds to $(t,\delta)=(0,0)$
(resp. $(1,0)$) and the last one is indexed by $(1,1)$.

\item In the third filtration of this figure, $\gr^1_\varrho(1)$ gathers both the contribution of
$\AC_{K^v,\mathfrak m}(\varrho,(3))$ and those of 
$\AC_{K^v,\mathfrak m}(\varrho,(2,1))$ in the large circle. There is then only
remaining graded part for $(t,\delta)=(0,0)$ which corresponds to the
contribution of $\AC_{K^v,\mathfrak m}(\varrho,(1,1,1))$ which corresponds
to the remaining part of the cohomology of 
$\bigoplus_{\pi_v/ r_l(\pi_v) \simeq \varrho} \PC(1,\pi_v)(0)$ when the contribution of
$\AC_{K^v,\mathfrak m}(\varrho,(2,1))$ has been removed.
\end{itemize}

\item For every $0 \leq t \leq r-1$,
the modulo $l$ monodromy operator $\overline N_\varrho$ induces
$$\gr_\varrho^{-\frac{r(r+1)}{2}+(1+2+\cdots+(t+1)-(t-1))}(r) \longrightarrow
\gr_\varrho^{-\frac{r(r+1)}{2}+(1+2+\cdots+t-(t-1))}(r),$$
which is surjective relatively to all the $(\varrho,r)$-small subquotients.

\item if $\Gamma_{K^v,\mathfrak m}(r,t,\delta)$ is the \textit{initial} lattice of 
$\gr^{-\frac{r(r+1)}{2}+(1+2\cdots+(t+1)-\delta)}_\varrho(r)$ induced by
$H^0(\sh_{K^v(\oo),\bar s_v},\PC(r-\delta)
(-\frac{r-1-2t+\delta}{2}))_{\mathfrak m}$
then 
$$\gr_\varrho^{-\frac{r(r+1)}{2}+(1+2+\cdots+(t+1)-\delta)}(r)\hookrightarrow
\Gamma_{K^v,\mathfrak m}(r,t,\delta)$$
the cokernel being $l$-torsion and the kernel of the modulo $l$
reduction of this injection has a $\varrho$-generic socle.
\end{itemize}

\subsubsection{Proof of the induction}

Note first that the case $r=s_\varrho$ is clearly satisfied. We now suppose that
the result is true for $r$ and we prove it is then true for $r-1 \geq 1$. 
We follow closely the arguments of the case $s_\varrho=3$.

\medskip 

\noindent \textit{Step 1: the exchange is non trivial} \\
Consider the three first graded parts 
$V:=\Fil^{-r(r+1)/2+2}_\varrho(r)$ where the nilpotent monodromy operator induces 
$\gr^{-r(r+1)/2+2}_\varrho(r) \longrightarrow \gr^{-r(r+1)/2}_\varrho(r)$
which remains an isomorphism modulo $l$. 
We then exchange the two first graded parts and then denote the new graded parts as
$$\widetilde{\gr}^{-r(r+1)/2+1}_\varrho(r,\underline s_\varrho(r,\max)),
\qquad \widetilde{\gr}^{-r(r+1)/2}_\varrho(r,\underline s_\varrho(r,\max)), \qquad
\gr^{-r(r+1)/2+2}_\varrho(r)$$
with
$$0 \rightarrow \widetilde{\gr}^{-r(r+1)/2+1}_\varrho(r) \longrightarrow
\gr^{-r(r+1)/2+1}_\varrho(r) \longrightarrow T \rightarrow 0,$$
and
$$0 \rightarrow \gr^{-r(r+1)/2}_\varrho(r) \longrightarrow 
\widetilde{\gr}^{-r(r+1)/2}_\varrho(r) \longrightarrow T \rightarrow 0,$$
where $T$ is torsion. The situation is exactly similar as step 1 for the case 
$s_\varrho=3$, so that by
considering the modulo $l$ reduction of the nilpotent monodromy operator on $V$,
we deduce that the exchange is non trivial, i.e. $T \neq 0$.
Considering the modulo $l$ reduction of the second short exact sequence,
the induction hypothesis tells us that the socle of $T[\varpi_L]$
is $\varrho$-generic.

Consider an automorphic filtration of $\gr^{-r(r+1)/2+1}_\varrho(r)$ whose
graded parts $\grr_\varrho^k(r)$ correspond to the contribution of some $\Pi \in
\AC_{K^v,\mathfrak m}(\varrho,\underline s_\varrho)$ for some
partition $\underline s_\varrho=(r-1 \geq \cdots)$ of $s_\varrho$
starting with $r-1$. We infer a similar filtration of
$\widetilde \gr^{-r(r+1)/2+1}_\varrho(r)$ with graded parts
$$0 \rightarrow \widetilde \grr^k_\varrho(r) \longrightarrow \grr^k_\varrho(r) \longrightarrow T^k_\varrho(r) \rightarrow 0.$$
As $T$ is also a quotient of 
$\widetilde \gr^{r(r+1)/2-1}_\varrho(r)$, we then deduce that there exists
$k$ such that $T^k_\varrho(r)[\varpi_L]$ has a non zero socle which is 
$\varrho$-generic.

\medskip

\noindent \textit{Step 2: the subspace lattice of $\Pi_0$ has also a 
$\varrho$-generic socle} \\
The same arguments of the case $s_\varrho=3$ apply without any change using 
lemma \ref{lem-lattice-gen}.

\medskip

\noindent \textit{Step 3: every subspace lattice attached to any 
$\Pi \in \AC_{K^v,\mathfrak m}(\varrho,\underline s_\varrho)$ for any
partition $\underline s_\varrho=(r-1 \geq \cdots)$, has also a $\varrho$-generic
socle} \\
The main ingredient is the same as in the case $s_\varrho=3$, that is
proposition \ref{prop-BLC}, but now we have to struggle 
with the fact that usually there are many partitions 
$\underline s_\varrho=(r-1 \geq \cdots)$ starting with $r-1$: the problem is that
proposition \ref{prop-BLC} needs the $K_v$-type to have multiplicity one
in $(\Pi_v)_{|K_v}$.

\smallskip

\textit{(i) From some $\underline s_\varrho^0$ to $(r-1,s_\varrho-r+1)$} \\
By now we proved the existence of some 
$\Pi_0 \in \AC_{K^v,\mathfrak m}(\varrho,\underline s_\varrho^0)$ for which its
subspace lattice $\Gamma_0^{gen}$ induced by the cohomology is such that
its modulo $l$ reduction has a $\varrho$-generic socle. The local
component of $\Pi_0$ at $v$ looks like
$$\Pi_{0,v} \simeq \st_{r-1}(\pi_{v,0}) \times \cdots \cdots \st_{l_k}(\pi_{v,k})
\times \psi$$
where 
\begin{itemize}
\item $\underline s_\varrho^0=(r-1 \geq l_1 \geq \cdots \geq l_k)$, 

\item the modulo $l$ reduction of $\pi_{v,i}$ is isomorphic to $\varrho$ for
$i=0,\cdots,k$

\item and $\varrho$ does not belong to the supercuspidal support of the modulo
$l$ reduction of $\psi$.
\end{itemize} 
Consider now 
$$\Pi_{v,+} \simeq \st_{r-1}(\pi_{v,0}) \times \st_{l_1} (\pi_{v,0} \{ \delta_1/2 \})
\times \cdots \st_{l_k}(\pi_{v,0} \{ \delta_k/2 \}) \times \psi$$
where $\delta_1=s_\varrho-r$, $\delta_2=\delta_1-2l_1$... and
$\delta_k=\delta_1-2l_1-\cdots - 2l_{k-1}=r-s_\varrho$. 

\rem The $\delta_i$ are chosen so that 
\begin{equation} \label{eq-st-quotient}
\st_{s_\varrho-r+1}(\pi_{v,0}) \hookrightarrow
\st_{l_1} (\pi_{v,0} \{ \delta_1/2 \}) \times \cdots \times
\st_{l_k}(\pi_{v,k} \{ \delta_k/2 \}).
\end{equation}

\begin{lem} 
There is exist a bijection between the stable lattices $\Gamma$ of
$$\st_{r-1}(\pi_{v,0}) \times \st_{l_1} (\pi_{v,1}) \times \cdots \times 
\st_{l_k}(\pi_{v,k})$$ 
and those $\Gamma'$ of
$$\st_{r-1}(\pi_{v,0}) \times \st_{l_1} (\pi_{v,0} \{ \delta_1/2 \})
\times \cdots \times \st_{l_k}(\pi_{v,0} \{ \delta_k/2 \})$$
such that $\overline \Gamma:=\Gamma/\varpi_L \Gamma \simeq
\Gamma'/\varpi_L \Gamma'=:\overline \Gamma'$.
\end{lem}

\begin{proof}
Let denote by $\Gamma_{ind}$ (resp. $\Gamma'_{ind}$) 
the induced stable lattice of 
$\st_{r-1}(\pi_{v,0}) \times \cdots \times \st_{l_k}(\pi_{v,k})$
(resp. of $\st_{r-1}(\pi_{v,0}) \times \cdots \times \st_{l_k}(\pi_{v,0} \{ \delta_k/2)$) 
and note that their modulo $l$
reduction is semi-simple and so isomorphic.
Consider then a stable lattice $\Gamma$ contained in $\Gamma_{ind}$
such that $\tau:=\Gamma_{ind}/\Gamma$ is $l$-torsion and let $\Gamma'$
be defined by the following pullback:
$$\xymatrix{
& \tau \ar@{=}[r] & \tau \\
l \Gamma'_{ind} \ar@{^{(}->}[r] & \Gamma'_{ind} \ar@{->>}[r] \ar@{->>}[u]  & 
\tau \oplus \tau' \ar@{->>}[u] \\
l \Gamma'_{ind} \ar@{^{(}->}[r] \ar@{=}[u] & \Gamma' \ar@{^{(}-->}[u] \ar@{-->}[r] &
\tau'. \ar@{^{(}->}[u]
}$$
We can then repeat the argument with some new $(\Gamma,\Gamma')$
instead of $(\Gamma_{ind},\Gamma'_{ind})$ to cover all the cases.
\end{proof}

In particular for the lattice $\Gamma=\Gamma^{gen}$ such that 
modulo $\varpi_L$, it has an irreducible generic socle, we obtain
$\Gamma^{gen,'}$ those modulo $\varpi_L$ reduction also has an 
irreducible generic socle. Let denote by $\Gamma^{gen,st}$ the lattice of
$\st_{r-1}(\pi_{v,0}) \times \st_{s_\varrho-r+1}(\pi_{v,0})$
given by $\Gamma^{gen,'}$ through the injective map (\ref{eq-st-quotient}).

By local Langlands correspondence to
$\st_{r-1}(\pi_{v,0}) \times \st_{l_1} (\pi_{v,1}) \times \cdots \times
\st_{l_k}(\pi_{v,k})$
is associated an inertial type and, cf.notation \ref{nota-SZ}, an $K_v$-type
$\sigma_{L}$ obtained by induction from the $K_v$-types 
of the Steinberg representations: note that $\sigma_L$ has multiplicity one in this
space. Replacing the $\pi_{v,i}$ by $\pi_{v,0}$ in
the construction of $\sigma_L$ we obtain a subspace $\sigma'_L$ of
$\st_{r-1}(\pi_{v,0}) \times \st_{l_1} (\pi_{v,0} \{ \delta_1/2 \})
\times \cdots \times \st_{l_k}(\pi_{v,0} \{ \delta_k/2 \})$.  
By construction in notation \ref{nota-SZ}, the $K_v$-type
$\sigma_{st,L}$ of $\st_{r-1}(\pi_{v,0}) \times \st_{s_\varrho-r+1}(\pi_{v,0})$
is a subspace of $\sigma'_L$: note that
$\sigma_{st_L}$ has multiplicity one in this space.

For a lattice $\Gamma$ of 
$\st_{r-1}(\pi_{v,0}) \times \st_{l_1} (\pi_{v,1}) \times \cdots \times 
\st_{l_k}(\pi_{v,k})$
we then obtain a lattice $\sigma$ of $\sigma_{L}$ and, using the
previous lemma, a lattice $\sigma'$ of $\sigma'_L$: note that
$\overline \sigma':=\sigma'/\varpi_L \sigma' \simeq 
\sigma/\varpi_L \sigma=:\overline \sigma$ 
as a $\Fm_L$-representation of 
$K_v=\GL_d(\OC_v)$. We also obtained a lattice $\sigma_{st}$ of
$\sigma_{st,L}$ with 
$$\overline \sigma_{st}:=\sigma_{st}/\varpi_L \sigma_{st} \hookrightarrow
\overline \sigma' \simeq \overline \sigma.$$
Starting with $\Gamma^{gen}$ we denote the associated lattices by
$\sigma^{gen}$, $\sigma^{gen,'}$ and 
$\sigma_{st}^{gen}$.
Consider a stable lattice $\Gamma_{st,+}$ of $\Gamma_{st,L}$ such that
$$\Gamma_{st}^{gen} \subsetneq 
\Gamma_{st,+} \subsetneq \varpi_L^{-1} \Gamma_{st}^{gen}.$$
Note that the cokernel $\Gamma_{st,+}/\Gamma_{st}^{gen}$ contains
the maximal $K_v$-type $\overline \sigma_{\max}$ of $\st_{s_\varrho}(\varrho) \times \overline \psi$ for some $\overline \psi$ whose supercuspidal support does not
contain $\varrho$. We then deduce that the cokernel of 
$\sigma_{st}^{gen} \hookrightarrow \sigma_{st,+}$ is non zero as it contains
at least $\overline \sigma_{\max}$.
We then construct the lattice $\Gamma'_{+}$ by pushout
$$\xymatrix{
\Gamma_{st}^{gen} \ar@{^{(}->}[r] \ar@{^{(}->}[d] & \Gamma_{st,+} 
\ar@{^{(}-->}[d] \\
\Gamma^{gen,'} \ar@{^{(}-->}[r] & \Gamma'_+,
}$$
giving also rise to a lattice $\Gamma_+$. We then obtained lattices
$\sigma_{st,+}$, $\sigma_+'$ and $\sigma_{+}$. With the notations of
\S \ref{para-lattice},
consider then the following commutative diagram
\begin{equation} \label{diag1}
\xymatrix{
& \overline M_{K^v}(\overline \sigma^v \overline \sigma_{st}^{gen}) \ar[r] 
\ar@{^{(}->}[d] & 
\overline M_{K^v}(\overline \sigma^v \overline \sigma_{st,+}) \ar@{^{(}->}[d] \\
\overline M_{K^v}(\overline \sigma^v \overline \sigma^{gen,'}) \ar@{=}[r] &
\overline M_{K^v}(\overline \sigma^v \overline \sigma^{gen,'}) \ar[r] & 
\overline M_{K^v}(\overline \sigma^v \overline \sigma_{+}') \ar@{=}[r] &
\overline M_{K^v}(\overline \sigma^v \overline \sigma_{+})
}
\end{equation}
As explained in the proof of proposition \ref{prop-BLC}, the case of
$\underline s_\varrho^0$ gives us that
the bottom horizontal map is zero. The top one
is either zero or an isomorphism but the commutative diagram imposes it
is also zero which gives us that, cf. the proof of \ref{prop-BLC} using the multiplicity
one property, $\Gamma_{st,+}$ can not be a subspace
lattice for some $\Pi \in \AC_{K^v,\mathfrak m}(\varrho,(r-1,s_\varrho-r-1))$.
Finally we are able to eliminate stable lattice on the right of $\Gamma^{gen}_{st}$
which are not isomorphic to it.

To deal with lattices $\Gamma_{st,-}$ contained in $\Gamma_{st}^{gen}$
$$\varpi_L \Gamma_{st}^{gen} \subsetneq
\Gamma_{st,-} \subsetneq \Gamma_{st}^{gen},$$
i.e. to go on the left, we repeat the same arguments but now playing with
$$\Pi_{v,-} \simeq \st_{r-1}(\pi_{v,0}) \times \st_{l_1} (\pi_{v,0} \{ \delta_1/2 \})
\times \cdots \st_{l_k}(\pi_{v,0} \{ \delta_k/2 \}) \times \psi$$
where $\delta_1=r-s_\varrho$, $\delta_2=\delta_1+2l_1$... and
$\delta_k=\delta_1+2l_1+\cdots +2l_{k-1}=s_\varrho-r$ so that  
\begin{equation} \label{eq-st-quotient2}
\st_{l_1} (\pi_{v,0} \{ \delta_1/2 \}) \times \cdots \times
\st_{l_k}(\pi_{v,k} \{ \delta_k/2 \}) \twoheadrightarrow \st_{s_\varrho-r+1}(\pi_{v,0}).
\end{equation}
We construct $\Gamma'_-$ by pullback
$$\xymatrix{
\Gamma_{st,-} \ar@{^{(}->}[r] & \Gamma_{st}^{gen} \\
\Gamma'_- \ar@{-->>}[u] \ar@{^{(}-->}[r] & \Gamma^{gen,'}. \ar@{->>}[u]
}$$
With similar notations as above, we have a commutative diagram
\begin{equation}\label{diag2}
\xymatrix{
& \overline M_{K^v}(\overline \sigma^v \overline \sigma_{st,-}) \ar[r] & 
\overline M_{K^v}(\overline \sigma^v \overline \sigma_{st}^{gen}) \\
\overline M_{K^v}(\overline \sigma^v \overline \sigma_-) \ar@{=}[r] &
\overline M_{K^v}(\overline \sigma^v \overline \sigma_-') \ar[r] \ar@{->>}[u] & 
\overline M_{K^v}(\overline \sigma^v \overline \sigma^{gen,'}), \ar@{->>}[u]
\ar@{=}[r] & \overline M_\oo(\overline \sigma^v \overline \sigma^{gen})
}
\end{equation}
where the horizontal maps are either zero or an isomorphism which then allows
us to point an unique possible stable lattice for elements of
$\AC_{K^v,\mathfrak m}(\varrho,(r-1,s_\varrho-r-1))$.

\rem Note that the lattices $\Gamma_{st}^{gen}$ constructed by pushout and
pullback are not necessary the same despite what the notation suggests.

\smallskip

\textit{(ii) From $(r-1,s_\varrho-r+1)$ to any $\underline s_\varrho$} \\
We repeat the same arguments but starting now from lattices
$$\varpi_L \Gamma^{gen,'} \subsetneq
\Gamma'_- \subsetneq \Gamma^{gen,'} \subsetneq \Gamma'_+ \subsetneq
\varpi_L^{-1} \Gamma^{gen,'}.$$
Using pullback and pushout, we then construct lattices $\Gamma_{st,\pm}$
and we conclude through the commutative diagrams (\ref{diag1}) and (\ref{diag2}).
\medskip

\noindent \textit{Step 4: Ihara's lemma for $\AC_{K^v,\mathfrak m}(\varrho,
(r-1 \geq \cdots ))$} \\
The arguments are exactly the same as in the case $s_\varrho=3$. Precisely
consider as before an automorphic filtration of 
$H^0(\sh_{K^v(\oo),\bar s_v},\overline \Zm_l)_{\mathfrak m}$ such that the graded
parts correspond to some automorphic representation. Consider such a
graded part associated to 
$\Pi \in \AC_{K^v,\mathfrak m}(\varrho,\underline s_\varrho)$ where
$\underline s_\varrho=(r-1 \geq \cdots)$ is a partition of $s_\varrho$
starting with $r-1$. Let denote by $\Gamma(\Pi)$ the associated lattice and 
suppose that its modulo $l$ reduction has a non $\varrho$-generic subspace 
$\tau$ which is
a subspace of $H^0(\sh_{K^v(\oo),\bar s_v},\overline \Fm_l)_{\mathfrak m}$:
we want to prove it is absurd. Let
$\Gamma(\Pi)^{gen}$ be the subspace lattice associated to $\Pi$ so that
$$0 \rightarrow \Gamma(\Pi)^{gen} \longrightarrow \Gamma(\Pi) \longrightarrow T
\rightarrow 0,$$
where $T$ is $l$-torsion. The modulo $l$ reduction of $\Gamma(\Pi)$ is such that
$$0 \rightarrow W \longrightarrow \Gamma(\Pi) \otimes_{\overline \Zm_l}
\overline \Fm_l \longrightarrow T \rightarrow 0,$$
where the socle of $T$ is generic and where $W$ is the image of
$$\Gamma(\Pi)^{gen} \otimes_{\overline \Zm_l} \overline \Fm_l 
\longrightarrow \Gamma(\Pi) \otimes_{\overline \Zm_l} \overline \Fm_l.$$ 
and non zero as we have seen that the modulo $l$
We then remark that $\tau$ is necessary a subspace of $W$ but as
$\tau$ is not a subspace of 
$\Gamma(\Pi)^{gen} \otimes_{\overline \Zm_l} \overline \Fm_l$
it cannot be a subspace of 
$H^0(\sh_{K^v(\oo),\bar s_v},\overline \Fm_l)_{\mathfrak m}$.

\medskip

\noindent \textit{Step 5: Final exchanges to conclude the induction} \\
We then exchange every contribution coming from
$\AC_{K^v,\mathfrak m}(\varrho,\underline s)$ for any
$\underline s=(r \geq \cdots)$ until we arrive at the filtration for $r-1$.
We then just have to check the hypothesis on the modulo $l$
monodromy operator.  Before all the exchanges we know that for every 
$0 \leq t \leq r-2$:
$$\overline N: \gr_\varrho^{-\frac{r(r+1)}{2}+(1+2+\cdots +t - (t-1))}(r) 
\longrightarrow 
\gr_\varrho^{\frac{-r(r+1)}{2}+(1+2+\cdots+(t-1)-(t-1))}(r),$$
is surjective relatively to all the $(\varrho,r)$-small subquotients and also for
the $(\varrho,r-1)$-small ones. After we have to check that
\begin{equation} \label{eq-ngr}
\overline N: \gr_\varrho^{-\frac{r(r+1)}{2}+(1+2+\cdots+t-(t-1))}(r-1) 
\longrightarrow \gr_\varrho^{-\frac{r(r+1)}{2}+(1+2+\cdots+(t-1)-(t-1))}(r-1),
\end{equation}
remains surjective relatively to all the $(\varrho,r-1)$-small subquotients.
By construction we have
$$\xymatrix{
\bigoplus_{\underline s_\varrho \leq (r-1)}
\bigoplus_{\Pi \in \AC_{K^v,\mathfrak m}(\varrho,\underline s_\varrho)} 
(\Pi^\oo)^{K^v} \otimes L(\Pi_v)(...) \ar@{^{(}->}[r] \ar@{^{(}->}[dr] & 
\gr_\varrho^{-\frac{r(r+1)}{2}+(1+2+\cdots + t-(t-1))}(r) \otimes_{\overline \Zm_l} \overline \Qm_l \\
& \gr_\varrho^{-\frac{r(r+1)}{2}+(1+2+\cdots + (t-1)-(t-1))}(r) 
\otimes_{\overline \Zm_l} \overline \Qm_l
}$$
and if $\Gamma_1$ (resp. $\Gamma_2$) 
is the lattice induced on this subspace by
$\gr_\varrho^{-\frac{r(r+1)}{2}+(1+2+\cdots +t-(t-1))}(r)$
(resp. by $\gr_\varrho^{-\frac{r(r+1)}{2}+(1+2+\cdots+(t-1)-(t-1))}(r)$), then we have
$$0 \rightarrow \gr_\varrho^{-\frac{r(r+1)}{2}+(1+2+\cdots+(t-1)-(t-2))}(r-1) \longrightarrow 
\Gamma_1 \longrightarrow T_1 \rightarrow 0,$$
$$0 \rightarrow \gr_\varrho^{-\frac{r(r+1)}{2}+(1+2+\cdots+(t-2)-(t-2))}(r-1) \longrightarrow  
\Gamma_2\longrightarrow T_2 \rightarrow 0,$$
where none of the irreducible constituant of $T_1$ and $T_2$ 
are $(\varrho,r-1)$-small
in the sense of definition \ref{defi-rhotsmall}.
By the induction hypothesis, the monodromy operator induces 
$\Gamma_1 \longrightarrow \Gamma_2$ such that, after tensoring with
$\overline \Fm_l$, the cokernel does not have any $(\varrho,r-1)$-subquotients.
We then deduce that the induce map
$$N:\gr_\varrho^{-\frac{r(r+1)}{2}+(1+2+\cdots+(t-1)-(t-2))}(r-1) \longrightarrow
\gr_\varrho^{-\frac{r(r+1)}{2}+(1+2+\cdots+(t-2)-(t-2))}(r-1)$$
is such that the cokernel of its modulo $l$ reduction does not have
any $(\varrho,r-1)$-subquotients.

\subsection{Genericity for KHT-Shimura varieties} 

Consider an irreducible subspace
$\tau$ of $H^0(\sh_{K^v(\oo),\bar s_v},\Psi)_{\mathfrak m}$. To prove that
$\tau$ is generic, we are led to prove it is $\varrho$-generic for every $\varrho$
in its supercuspidal support. Recall that 
$$H^0(\sh_{K^v(\oo),\bar s_v},\Psi)_{\mathfrak m} \simeq \bigoplus_{\varrho \in \cusp_{\overline \Fm_l}} H^0(\sh_{K^v(\oo),\bar s_v},\Psi_\varrho)_{\mathfrak m}$$
and, from the typicity property, for $\varrho$ belonging to the supercuspidal
support of $\tau$, $\tau$ is also a subspace of
$H^0(\sh_{K^v(\oo),\bar s_v},\Psi_\varrho)_{\mathfrak m}$.
From step 4 of the previous section, 
we know that $\tau$ has to be $\varrho$-generic.

\subsection{Breuil's lattice conjecture for $l \neq p$}

Consider an inertial type $\tau_v$ and its associated $K_v$-type 
$\sigma_L(\tau_v)$. Consider also
a system of Hecke eigenvalues $\lambda: \Tm_m \longrightarrow  \Zm_L$ associated to
some automorphic representation $\Pi$ which appears in middle cohomology group
of $\Sh_{K^v(\oo),\bar \eta_n}$ with coefficients in $\LC^\vee_{\sigma_0(v)}$ for
$\sigma_0(v)=\sigma^v \sigma_v$ where $\sigma^v$ is a continuous finitely generated
representation of $K^v$. When $\sigma_L(\tau_v)$ appears with multiplicty one in $\Pi_v$,
we can define
$$\sigma_\lambda(\tau_v):=M_{K^v}(\sigma^v)^*[\lambda] \cap \sigma_L(\tau_v),$$
which is a stable lattice of $\sigma_L(\tau_v)$: from theorem \ref{thm:ILL} (iii), we can apply
this construction to the maximal inertia type.
Proposition \ref{prop-BLC} tells us that this lattice
only depends on the modulo $l$ reduction of $\lambda$.
One possible translation of 
Breuil's lattice conjecture to our situation could be the following.

\begin{prop} \label{prop-BLC2}
The lattice $\sigma_\lambda(\tau_v)$ depends only on the local datum $\Pi_v$.
\end{prop}

\begin{proof}
Consider $\Pi_1$ and $\Pi_2$ associated to two systems of Hecke eigenvalues
$\lambda_1$ and $\lambda_2$ as above: we moreover suppose that $\Pi_{v,1} \simeq \Pi_{v,2}$. 
From Ihara's lemma
we know that the lattice of $\Pi_{v,1}$ is the one such that the socle of its modulo
$l$ reduction is irreducible and generic: this lattice is then isomorphic to those
of $\Pi_{v,2}$ which proves the statement.
\end{proof}

However note that the previous proposition is not so interesting as there are very few stable lattices
compare to the case where $l=p$.

%% file: figure-filtration.tex
\psset{xunit=1cm,yunit=1cm,algebraic=true,dimen=middle,dotstyle=o,dotsize=5pt 0,linewidth=2pt,arrowsize=3pt 2,arrowinset=0.25}
\begin{pspicture*}(-11.5,-4.8)(6.44,6.01)
\pscircle[linewidth=2pt,hatchcolor=black,fillstyle=hlines,hatchangle=45,hatchsep=0.19](-11,0){0.2}
\pscircle[linewidth=2pt,hatchcolor=black,fillstyle=crosshatch,hatchangle=45,hatchsep=0.19](-10,-1){0.2}
\pscircle[linewidth=2pt,hatchcolor=black,fillstyle=crosshatch,hatchangle=45,hatchsep=0.19](-10,1){0.2}
\pscircle[linewidth=2pt](-9,2){0.2}
\pscircle[linewidth=2pt](-9,0){0.2}
\pscircle[linewidth=2pt](-9,-2){0.2}
\psline[linewidth=2pt]{->}(-9,0)(-9,-2)
\psline[linewidth=2pt]{->}(-9,2)(-9,0)
\psline[linewidth=2pt]{->}(-10,1)(-10,-1)
\rput[tl](-8.71,1.57){$\overline N \neq 0$}
\rput[tl](-8.44,-0.59){$\overline N \neq 0$}
\rput[tl](-10.58,0.4){$\overline N \neq 0$}
\pscircle[linewidth=2pt,hatchcolor=black,fillstyle=hlines,hatchangle=45,hatchsep=0.19](-5,0){0.2}
\pscircle[linewidth=2pt,hatchcolor=black,fillstyle=crosshatch,hatchangle=45,hatchsep=0.19](-4,1){0.2}
\pscircle[linewidth=2pt,hatchcolor=black,fillstyle=crosshatch,hatchangle=45,hatchsep=0.19](-4,-1){0.2}
\pscircle[linewidth=2pt](-3,2){0.2}
\pscircle[linewidth=2pt](-3,3){0.2}
\pscircle[linewidth=2pt](-3,4){0.2}
\psline[linewidth=2pt]{->}(-3,3)(-3,2)
\psline[linewidth=2pt]{->}(-3,4)(-3,3)
\rput[tl](-2.43,4.12){$\overline N = 0$}
\rput[tl](-2.39,2.97){$\overline N = 0$}
\psline[linewidth=2pt]{->}(-4,1)(-4,-1)
\rput[tl](-3.66,0.5){$\overline N \neq 0$}
\rput[tl](-5.45,0.537){1}
\rput[tl](-4.46,1.49){1}
\rput[tl](-3.66,-0.53){1}
\rput[tl](-3.52,2.448){1}
\rput[tl](-3.49,3.5){1}
\rput[tl](-3.47,4.53){1}
\pscircle[linewidth=2pt,hatchcolor=black,fillstyle=hlines,hatchangle=45,hatchsep=0.19](0,0){0.2}
\pscircle[linewidth=2pt,hatchcolor=black,fillstyle=crosshatch,hatchangle=45,hatchsep=0.19](1,1){0.2}
\pscircle[linewidth=2pt,hatchcolor=black,fillstyle=crosshatch,hatchangle=45,hatchsep=0.19](1,2){0.2}
\pscircle[linewidth=2pt](2,3){0.2}
\pscircle[linewidth=2pt](2,4){0.2}
\pscircle[linewidth=2pt](2,5){0.2}
\psline[linewidth=2pt]{->}(1,2)(1,1)
\psline[linewidth=2pt]{->}(2,4)(2,3)
\psline[linewidth=2pt]{->}(2,5)(2,4)
\pscircle[linewidth=2pt,hatchcolor=black,fillstyle=hlines,hatchangle=45,hatchsep=0.19](0.23,0.97){0.2}
\rput[tl](-0.37,0.57){2}
\rput[tl](-0.097,1.59){3}
\rput[tl](0.58,0.95){2}
\rput[tl](0.58,2.428){2}
\rput[tl](1.4,3.2){1}
\rput[tl](1.423,4.18){1}
\rput[tl](1.48,5.21){1}
\rput[tl](1.36,2){$\overline N=0$}
\rput[tl](2.39,4.066){$\overline N= 0$}
\rput[tl](2.39,4.98){$\overline N = 0$}
\rput[tl](-7.03,0.284){$\leadsto$}
\rput[tl](-1.79,0.4){$\leadsto$}
\pscircle[linewidth=2pt](0.85,1.47){1.1}
\rput{89.6}(2,3.99){\psellipse[linewidth=2pt](0,0)(1.769,0.87)}
\rput{89.6}(-3,2.9){\psellipse[linewidth=2pt](0,0)(1.769,0.87)}
\rput[tl](0.35,1.6){$\oplus$}
\rput[tl](-2.6,5.82){$\bigoplus_{\Pi_v \simeq \mathrm{St}_3} \Pi^K \otimes \sigma_\varrho(\Pi)=$}
\rput[tl](0.9,0.51){$=\bigoplus_{\Pi_v \simeq \mathrm{St}_{2,1}} \Pi^K \otimes \sigma_\varrho(\Pi)$}
\rput[tl](0,-0.5){$\verteq$}
\rput[tl](-1.44,-0.7){$\bigoplus_{\Pi_v \simeq \mathrm{St}_{1,1,1}} \Pi^K \otimes \sigma_\varrho(\Pi)$}
\pscircle[linewidth=2pt,hatchcolor=black,fillstyle=hlines,hatchangle=45,hatchsep=0.19](-10.89,-3.21){0.2}
\rput[tl](-10.54,-3.2){$\hookrightarrow$}
\pscircle[linewidth=2pt,hatchcolor=black,fillstyle=hlines,hatchangle=45,hatchsep=0.19](-9.44,-3.26){0.2}
\pscircle[linewidth=2pt,hatchcolor=black,fillstyle=hlines,hatchangle=45,hatchsep=0.19](-8.23,-3.31){0.2}
\pscircle[linewidth=2pt,hatchcolor=black,fillstyle=hlines,hatchangle=45,hatchsep=0.19](-10.5,-2.9){0.2}

\pscircle[linewidth=2pt,hatchcolor=black,fillstyle=crosshatch,hatchangle=45,hatchsep=0.19](-4.51,-3.4){0.2}
\pscircle[linewidth=2pt,hatchcolor=black,fillstyle=crosshatch,hatchangle=45,hatchsep=0.19](-3.2,-3.43){0.2}
\pscircle[linewidth=2pt](-1.22,-3.38){0.2}
\pscircle[linewidth=2pt](0,-3.42){0.2}
\rput[tl](-9.76,-2.54){1}
\rput[tl](-11.2,-2.5){2}
\rput[tl](-10.4,-2.3){3}

\rput[tl](-4.7,-2.7){1}
\rput[tl](-0.23,-2.7){1}
\rput[tl](-9.2,-3.2){$\hookrightarrow$}
\rput[tl](-4.21,-3.2){$\hookrightarrow$}
\rput[tl](-0.97,-3.2){$\hookrightarrow$}
\rput[tl](-8.95,-4.35){cokernels are killed by $l$ and their socle are generic}
\end{pspicture*}

%% file: figure-filtration2.tex
%

\psset{xunit=1cm,yunit=1cm,algebraic=true,dimen=middle,dotstyle=o,dotsize=5pt 0,linewidth=2pt,arrowsize=3pt 2,arrowinset=0.25}
\begin{pspicture*}(-9.6,-4)(0,3.2)
\pscircle[linewidth=2pt,hatchcolor=black,fillstyle=crosshatch,hatchangle=45,hatchsep=0.2](-8,0){0.2}
\pscircle(-7,1){0.2}
\pscircle(-7,-1){0.2}
\pscircle[linewidth=2pt,hatchcolor=black,fillstyle=crosshatch,hatchangle=45,hatchsep=0.2](-3,0){0.2}
\pscircle(-2,2){0.2}
\pscircle(-2,1){0.2}
\psline[linewidth=2pt]{->}(-6.98,0.8)(-6.98,-0.8)
\psline[linewidth=2pt]{->}(-2,1.8)(-2,1.2)
\rput[tl](-6.66,0.8){$\overline N \neq 0$}
\rput[tl](-1.48,1.94){$\overline N=0$}
\rput[tl](-3.4,0.4){$1$}
\rput[tl](-2.6,1.2){$1$}
\rput[tl](-4.98,0.58){$\leadsto$}
\rput{-43.15}(-7.52,-0.53){\psellipse[linewidth=2pt](0,0)(1.3,0.7)}
\rput{43.24}(-2.54,0.43){\psellipse[linewidth=2pt](0,0)(1.34,0.68)}
\pscircle[linewidth=2pt,hatchcolor=black,fillstyle=crosshatch,hatchangle=45,hatchsep=0.2](-7,-3){0.2}
\pscircle[linewidth=2pt,hatchcolor=black,fillstyle=crosshatch,hatchangle=45,hatchsep=0.2](-6,-3){0.2}
\pscircle(-3,-3){0.2}
\pscircle(-4,-3){0.2}
\rput[tl](-6.8,-3){$\hookrightarrow$}
\rput[tl](-3.8,-3){$\hookleftarrow$}
\rput[tl](-5.76,-2.8){$\twoheadrightarrow T \twoheadleftarrow$}
\rput[tl](-7.34,-2.3){$1$}
\rput[tl](-3.82,-2.3){$1$}
\end{pspicture*}